\documentclass[11pt,reqno]{amsart}  
\usepackage{color, amsmath,amssymb, amsfonts, amstext,amsthm, latexsym}
\usepackage{graphicx}
\usepackage{caption}

\allowdisplaybreaks
\usepackage{xcolor}
\usepackage{hyperref}
\hypersetup{
    colorlinks=true,
    linkcolor=black,
    filecolor=magenta,      
    urlcolor=cyan,
}

\usepackage[normalem]{ulem}

\newcommand{\R}{{\mathbb R}}

\newcommand{\NN}{{\mathbb N}}
\newcommand{\RR}{{\mathbb R}}

\newcommand{\EX  }{{\mathbb E}}
\newcommand{\EE}{{\mathbb E}}

\newcommand{\sech}{\operatorname{sech}}
\newcommand{\defeq}{\stackrel{\rm{def}}{=}}

\numberwithin{equation}{section}
\newtheorem{theorem}{Theorem}[section]

\newtheorem{lemma}[theorem]{Lemma}
\newtheorem{remark}[theorem]{Remark}

\newtheorem{Th}[theorem]{Theorem}
\newtheorem{Prop}[theorem]{Proposition}

\title[Stochastic focusing NLS]{Well-posedness of the focusing stochastic nonlinear\\ Schr\"odinger equation: $L^2$-critical and supercritical cases}

\author[A. Millet]{Annie Millet}
\address{SAMM, EA 4543,
Universit\'e Paris 1 Panth\'eon Sorbonne, 90 Rue de
Tolbiac, 75634 Paris Cedex France {\it and} LPSM, UMR 8001 
  Universit\'e de Paris et Sorbonne Universit\'e, France}
\email{annie.millet@univ-paris1.fr} 

\author[S. Roudenko]{Svetlana Roudenko}
\address{Department of Mathematics \& Statistics\\Florida International University,  Miami, FL, USA}
\curraddr{}
\email{sroudenko@fiu.edu}



\begin{document}
\begin{abstract}
We study the focusing $L^2$-critical and supercritical stochastic nonlinear Schr\"odinger equation subject to additive or multiplicative noise. We investigate global or long time behavior of solutions 
in $H^1$, which would correspond to global well-posedness in the deterministic case, with either deterministic or random initial data, and establish quantitative information about the well-posedness time, its probability and bounds on the solution in both cases.  We then give criteria for finite time blow-up with positive probability for an $H^1$-valued initial data with positive energy in both cases. 
\end{abstract}

\maketitle

\tableofcontents

\section{Introduction}\label{s1}
We study the focusing stochastic nonlinear Schr\"odinger (SNLS) equation  
subject to random perturbations of additive or multiplicative noise
\begin{align}\label{E:NLS}
\begin{cases} 
iu_t -\big(  \Delta u + |u|^{2\sigma} u \big) = \epsilon \, f(u), 
\quad (t,x)\in  [0,\infty) \times {\mathbb R}^n \, , \\
u(0,x)=u_0,
\end{cases}
\end{align}
where the initial data $u_0$ can be either deterministic or stochastic of $H^1(\mathbb R)$-type and the term $f(u)$ stands for a stochastic perturbation driven 
by the noise $W(dt,dx)$ white in time with some spatial regularity.

The deterministic setting ($\epsilon = 0$) has been extensively investigated in the last several decades. The existence and uniqueness of 
solutions goes back to Ginibre-Velo \cite{GV1979}, see also \cite{GV1985}, \cite{K1987}, \cite{T1987}, \cite{CW1988}, \cite{CW1990}, 
and the books \cite{Caz-book}, \cite{LinPon2015} for further references.

During their lifespans, solutions to the deterministic equation conserve the mass $M(u)$ and the energy (or Hamiltonian)  defined as
\begin{equation}\label{M-H}
 M(u)=\|u\|_{L^2(\RR^n)}^2 \quad \mbox {\rm and } \quad 
H(u)=\frac{1}{2} \|\nabla u\|_{L^2(\RR^n)}^2 - \frac{1}{2\sigma +2} \|u\|_{L^{2\sigma+2}(\RR^n)}^{2\sigma +2}.
\end{equation}
(The momentum is also conserved, but we omit it for this paper.)

The deterministic equation is invariant under the scaling: if $u(t,x)$ is a solution to \eqref{E:NLS} with $\epsilon=0$, 
then so is $u_\lambda(t,x) = \lambda^{1/\sigma} u(\lambda^2 t, \lambda x)$. Under this scaling the Sobolev $\dot{H}^s$ norm of 
solutions is invariant if $s = s_c$, the scaling index, defined as 
\begin{equation}\label{E:sc}
s_c = \frac{n}2 - \frac1{\sigma}.
\end{equation}
According to the sign and value of $s_c$, the NLS equation is classified as the $L^2$ (or mass)-subcritical equation if $s_c<0$ (or $\sigma < \frac{2}{n}$); 
the 
$L^2$-critical provided $s_c=0$ (or $\sigma=\frac{2}{n}$); the $\dot{H}^1$ (or energy)-critical if $s_c=1$ (or $\sigma = \frac{2}{n-2}$ for $n\geq 3$);  an intercritical (or mass-supercritical and energy-subcritical) if $s_c\in (0,1)$; and energy-supercritical if $s_c>1$. 

In the random setting ($\epsilon \neq 0$), well-posedness of solutions 
in both multiplicative (Stratonovich) and additive noise cases of the SNLS equation \eqref{E:NLS} was studied by de Bouard \& Debussche in a series of papers, \cite{deB_Deb_CMP}--\cite{deB_Deb_AnnProb}  
with further analysis and numerics in \cite{deB_Deb_DiM2001}, \cite{Deb_DiM2002}.
In these references,  the existence and uniqueness of local solutions 
in $H^1(\mathbb R^n)$ was obtained;  see Theorem \ref{local_wp} for a multiplicative stochastic  perturbation and Theorem \ref{lwp_add} for an additive one. 
Furthermore, they were able to extend local solutions to global in the $L^2$-subcritical case \cite{deB_Deb_CMP}, \cite{deB_Deb_H1}, as well as to prove 
(and confirm numerically) existence of blow-up solutions in the $L^2$-critical and intercritical cases for negative energy initial data \cite{deB_Deb_DiM2001}, \cite{deB_Deb_PTRF}, 
\cite{deB_Deb_AnnProb}, \cite{Deb_DiM2002}. 
Moreover, their striking result in \cite{deB_Deb_AnnProb} was that in the intercritical setting for the SNLS equation subject to a multiplicative perturbation driven by a quite regular noise, 
blow up occurs instantaneously with positive probability 
 for any data  in a subset of $H^2 $ (regardless of the energy sign); see a more detailed discussion about the blow-up in the beginning of Section \ref{S:B}. 
\smallskip

In this paper we study global behavior of $H^1$ solutions in the {\it $L^2$-critical} and {\it intercritical} stochastic NLS equation ($0 \leq s_c <1)$ in both multiplicative and additive noise cases
 (for precise definitions, see \S \ref{S:noise-multi} for multiplicative noise and \S \ref{S:noise-add} for additive noise). The $\dot{H}^1$-critical case ($s_c=1$) is studied in the sequel \cite{MR2025b}.  
\smallskip

In the deterministic setting the extension from local well-posedness to global in the $L^2$-critical and supercritical cases of the dispersive equations has 
been a subject of intense investigation in the last thirty years, in particular, identifying various thresholds for the global in time solutions vs.
finite time blow-up solutions. It has been known since 70s that in the $L^2$-critical and supercritical cases the $H^1$ solutions of the NLS equation 
may not exist globally for all time, since a blow-up in finite time may occur; the existence of such solutions is shown by a convexity argument on the 
variance quantity \eqref{def_V} \cite{VPT, Zakharov, Glassey}, see also books \cite{SS1999}, \cite{Fibich2015}.
An important ingredient in identifying the existence of blow-up solutions 
and in obtaining thresholds for the global versus local in time existence, is the so-called {\it ground state} 
solution of the form $u(t,x) = e^{it}Q(x)$, where 
$Q$ is an $H^1(\mathbb R^n)$ positive 
solution to 
\begin{equation}\label{E:Q}
-Q+\Delta Q +Q^{2\sigma +1} = 0.
\end{equation}
This solution is unique 
(e.g., see \cite{Str77}, \cite{BL}, \cite{W83}, \cite{Kwong}); when $s_c <1$ (in the energy-subcritical case) the ground state has an exponential decay $\sim e^{-|x|}$ at infinity; in one dimension it is explicit $Q(x) = (1+\sigma)^\frac1{2\sigma} \sech^{\frac1{\sigma}}(\sigma \, x)$, $\sigma>0$.

In the $L^2$-critical case (deterministic NLS) solutions with the mass below that of the ground state exist globally. 
This relies on the fact that the ground state is the minimizer of the corresponding Gagliardo-Nirenberg inequality and identifies its sharp constant;  see \eqref{GN} and \eqref{constant_GN}.    
This was shown by Weinstein in \cite{W83} and has been used extensively since then. 
We are able to obtain the same (a.s.) global existence in $H^1$ result in the multiplicative Stratonovich noise case (which conserves mass a.s. with deterministic initial data 
(Theorem \ref{gwp_critical}) and with random initial data whose mass is a.s. less than the mass of the ground state 
(for a precise statement, see Theorem \ref{gwp_critical_randomIC}).
In the case of an additive noise we deduce a certain control on the time existence of solutions 
so that on this time interval solutions remain controlled by the ground state mass 
(since the mass is not conserved in the additive setting);  
see Theorem \ref{Th_mass_crit_add}. Results for the global solutions in the $L^2$-critical case have been obtained in other works for the multiplicative (Stratonovich) noise (thus, preserving mass), however, with either stronger regularity assumptions on noise (e.g., in \cite[Remark 1.7]{FX}, 
for a sufficiently fast decaying covariance operator $\phi$ the authors show that in 1D SNLS the initial data with the mass below that of the $Q$ produces an $L^2$ global solution) or for the finite-dimensional noise  (e.g., see \cite{BRZ2014}, \cite{BRZ2016}), 
where a rescaling transformation of a Doss–Sussman type $y(t) = e^{-W}u(t)$ is applied to convert SNLS to a random NLS.

In this paper it suffices to consider the multiplicative noise \eqref{R-noise} with the minimal regularity condition (H1) and boundedness (H2), see \S \ref{S:noise-multi}.  
For the additive setting we consider a complex-valued noise with the regularity assumption of a Hilbert-Schmidt operator from $L^2$ to $H^1$ 
(to enable the use of energy estimates). We mention that local well-posedness for lower regularity $H^s$, $s<1$, 
and rougher noise have been considered in \cite{OPW}. For this paper we are interested in finite energy solutions, 
hence, the $H^1$ local well-posedness proved by de Bouard-Debussche \cite{deB_Deb_H1} suffices.

More importantly, we address the {\it intercritical} case, where we show a dichotomy similar to the deterministic case. 
The dichotomy in the nonlinear dispersive equations was originally introduced in the energy-critical NLS by Kenig \& Merle \cite{KenMer},
where the authors considered solutions below the energy of the ground state and characterized two possible solution behaviors: 
global existence (and thus scattering) or finite time blow-up. In \cite{HR2007}, \cite{HR2008}, \cite{DHR} the second author with her 
collaborators showed a dichotomy of solutions in the intercritical case of the NLS equation ($0<s_c<1$) using the scaling-invariant quantities: 
\begin{equation*}
H(u) M({u})^{\alpha} \quad \mbox{and} \quad  
\|{\nabla u}\|_{L^2(\mathbb{R}^n)} \|{u}\|_{L^2(\mathbb{R}^n)}^{\alpha},
\end{equation*} 
where
\begin{equation}\label{E:alpha}
\alpha = \frac{1-s_c}{s_c} \equiv \frac{2-(n-2)\sigma}{n\sigma -2}.
\end{equation}
(They also proved the $H^1$ scattering in 3D cubic NLS in both radial \cite{HR2008} and non-radial \cite{DHR} cases, which was later extended 
to all intercritical cases in \cite{Guevara}, \cite{FX-Caz},\cite{AN2013}. However, scattering is not the subject of this paper.)   

\begin{theorem}[\cite{HR2007, HR2008, DHR}]\label{T:main-deter}
Let $u_0\in H^1(\R^n)$ and $0<s_c<1$. Suppose 
\begin{equation}\label{E:ME}
H(u_0) M(u_0)^\alpha < H(Q) M(Q)^\alpha. 
\end{equation}  
\begin{itemize}
\item {\underline{\rm Part 1.}} 
If 
\begin{equation}\label{E:gm1}
\|{\nabla u_0}\|_{L^2(\R^n)} \|u_{0}\|^{\alpha}_{L^2(\R^n)}<\|{\nabla Q}\|_{L^2(\R^n)}\|{Q}\|^{\alpha}_{L^2(\R^n)},
\end{equation}
then the solution $u(t)$ is global in both time directions and
\begin{equation}\label{E:gm1-t}
\|{\nabla u(t)}\|_{L^2(\R^n)} \|u_0\|^{\alpha}_{L^2(\R^n)} < \|{\nabla Q}\|_{L^2(\R^n)}\|{Q}\|^{\alpha}_{L^2(\R^n)},
\end{equation}
for all time $t \in \mathbb R$. Moreover, the solution scatters in $H^1(\mathbb R^n)$.

\item {\underline{\rm Part 2.}} 
If 
\begin{equation}\label{E:gm2}
\|{\nabla u_0}\|_{L^2(\R^n)} \|u_{0}\|^{\alpha}_{L^2(\R^n)} > \|{\nabla Q}\|_{L^2(\R^n)}\|{Q}\|^{\alpha}_{L^2(\R^n)},
\end{equation}
then 
\begin{equation}\label{E:gm2-t}
\|{\nabla u(t)}\|_{L^2(\R^n)} \|u_{0}\|^{\alpha}_{L^2(\R^n)} > \|{\nabla Q}\|_{L^2(\R^n)}\|{Q}\|^{\alpha}_{L^2(\R^n)},
\end{equation}
for all $t$ in the maximal time interval of existence and if $|{x}| u_0\in L^2(\mathbb R^n)$ or $u$ is radial, then the solution blows up in finite time in both time directions.
\end{itemize}
\end{theorem}

In this paper we obtain some analog of Theorem \ref{T:main-deter} for the stochastic NLS for both multiplicative and additive noise cases. In particular, under the initial mass-energy assumption \eqref{E:ME}, we show that a similar global behavior as in Part 1 holds on a positive time interval with a positive probability, see Theorems \ref{th-wp-inter-multi}, \ref{th-wp-inter-multi-randomIC} in the multiplicative noise case and Theorem \ref{th-wp-inter-add} in the additive noise setting. Note that one of the main difficulties that we encounter is that the energy is not conserved in either cases of noise, and the mass is not conserved in the additive case. Thus, we are not able to obtain a completely global existence result as in the deterministic case  \eqref{E:gm1}-\eqref{E:gm1-t}, but instead we provide quantitative estimates on the maximal time $T$ of solution existence, the positive probability with which it occurs, and the time bounds on the $H^1$ norm of the solution.   
The maximal time $T$ inversely depends on the size of the noise, and thus, as it goes to zero, the time $T$ approaches infinity, consistent with the deterministic case.

We also obtain an analog of \eqref{E:gm2}-\eqref{E:gm2-t} in Part 2 (blow-up) for deterministic and random initial $H^1$ data with positive energy that blows up in finite time with positive probability, see Theorems \ref{th_blowup1}, \ref{th_blowup2} for the multiplicative case and Theorem \ref{th_blowup1-add} in the additive noise setting. Such blow-up is possible if the size of the noise is small.  
We remark that our blow-up results extend the ones 
from  \cite{deB_Deb_PTRF} and \cite{deB_Deb_AnnProb}. For a {\it multiplicative} noise, in \cite[Remark~4.2]{deB_Deb_AnnProb}
the authors proved blow-up with positive probability for the initial conditions with negative energy, under 
regularity conditions on the noise same as ours and finite variance initial data (i.e., $u_0 \in \Sigma 
: = H^1(\R^n) \cup \{|x|u \in L^2(\R^n)\}$).  
In \cite{deB_Deb_AnnProb} the authors showed that in the intercritical case in dimensions 2 and 3 (i.e., when $s_c>0$ or $\sigma>1$ in dimension 2 and $\sigma \in (\frac{2}{3},2)$ in dimension 3),
the blow-up occurs instantaneously 
(that is, before any fixed positive time $T$) with positive probability for {\it any} initial data on a set $\Sigma^2:=H^2(\R^n) \cap \{|x|^2 u \in L^2(\R^n)\}$, a {subset} of $\Sigma$. This is a stronger result than the one we obtain in  Theorems \ref{th_blowup1}, \ref{th_blowup2}, 
but it holds on a smaller class of initial data $\Sigma^2$ (and thus, for more regular solutions than the finite energy ones) and more regular noise (Hilbert-Schmidt from $L^2$ to $H^2$ and non-degenerate), compared to our results. 
In the case of an {\it additive} noise, \cite[Proposition~3.2]{deB_Deb_PTRF} proved that when the covariance operator $\phi$ is Hilbert-Schmidt from $L^2$ to $\Sigma$, $\gamma$-radonifying from
$L^2$ to $L^{4\sigma +2}$ and bounded from $L^2$ into $H^2\cap L^\infty$, the blow-up occurs again instantaneously 
provided that the energy of the initial condition is ``very negative". Once more our blow-up is not instantaneous, but our noise is less regular and the blow-up can occur for positive energy initial data.  

We also mention that dynamics of global behavior of solutions in SNLS with various approximations of white noise, including blow-up solutions, profiles, rates, location, 
we have investigated numerically in \cite{MRY} and with several types of colored in space noise in \cite{MRRY}. These works extended the initial numerical investigations for SNLS in 
\cite{deB_Deb_DiM2001} and \cite{Deb_DiM2002}. 
\smallskip

To simplify notation, we often omit the space $\mathbb R^n$ from the Sobolev or Lebesgue spaces notation such as $H^1(\mathbb R^n)$ and just write  $H^1$, etc. 
\smallskip

The paper is organized as follows: in Section \ref{S:Q} we review the important properties of the ground state $Q$ and some interpolation inequalities, 
connecting with the $H^1$ global existence (deterministic) results in the $L^2$-critical and intercritical cases. In Section \ref{S:M} we review the 
multiplicative (Stratonovich) noise and obtain an analog of global well-posedness in that case for both $L^2$-critical and intercritical cases. 
In Section \ref{S:A} we consider the additive noise and obtain an analog of global well-posedness in that case. 
Note that since in that case neither mass nor energy are preserved, the statement 
of the result in the additive noise case is 
more involved than in the case of multiplicative noise.
Finally, in Section \ref{S:B} we study the blow-up for $H^1$ data for both cases of noise.

\subsection*{Acknowledgments} 
This work started when the first author visited Florida International University in 2019-20.  
She would like to thank FIU for the support and excellent working conditions. 
A.M.'s research has been conducted within the FP2M federation (CNRS FR 2036). S.R. was partially supported by the NSF grants DMS-2055130 and DMS-2452782.

\section{Preliminaries on the ground state}\label{S:Q}

We first recall the Gagliardo-Nirenberg interpolation inequality,  $\sigma > 0$ (and $\sigma < \frac{2}{n-2}$, if $n \geq 3$)  (e.g., \cite[Thm 4.1, Rem 4.2]{Fa_Li_Po}) 
\begin{equation}\label{GN}
\|u\|_{L^{2\sigma +2}(\RR^n)}^{2\sigma +2} \leq C_{GN} \, \|\nabla u\|_{L^2(\RR^n)}^{n \sigma} \; \|u\|_{L^2(\RR^n)}^{2 - (n-2)\sigma}.
\end{equation}
When $\sigma = \frac{2}{n-2}$ ($n \geq 3$), this inequality becomes the famous Sobolev embedding of 
$\dot{H}^1(\RR^n)$ into 
$ L^{\frac{2n}{n-2}}(\RR^n)$, 
\begin{equation}\label{E:Sobolev}
\|u\|_{L^{\frac{2n}{n-2}}(\RR^n)} \leq C_n \| \nabla u\|_{L^2(\RR^n)}.
\end{equation}
In order to record the sharp constant $C_{GN}$ in \eqref{GN}, we recall the Pokhozhaev identities
(which are obtained by multiplying \eqref{E:Q} either by $Q$ or $x\!\cdot\!\nabla Q$ and integrating by parts)  
\begin{align*}
\|Q\|_{L^{2\sigma +2}(\RR^n)}^{2\sigma +2} &= \|Q\|_{L^2(\RR^n)}^2 + \| \nabla Q\|_{L^2(\RR^n)}^2,  \\
\frac{n}{2\sigma +2} \|Q\|_{L^{2\sigma +2}(\RR^n)}^{2\sigma +2} & = \frac{n}{2} \|Q\|_{L^2(\RR^n)}^2  + \frac{n-2}{2} \| \nabla Q\|_{L^2(\RR^n)}^2.
\end{align*}
Solving for the $L^2(\RR^n)$ norm of the gradient of $Q$ and the $L^{2\sigma +2}(\RR^n)$ norm of $Q$ 
in terms of its mass, we get 
\begin{align}
\|  \nabla Q\|_{L^2(\RR^n)}^2 &= 
\frac{n \sigma}{2- (n-2) \sigma}
\|Q\|_{L^2(\RR^n)}^2 
\equiv \frac{n}{2 (1-s_c)} \|Q\|^2_{L^2(\mathbb R^n)},  
\label{E:normQ1}
\\ 
\|Q\|_{L^{2\sigma +2}(\RR^n)}^{2\sigma +2} & = {\frac{2(\sigma +1)}
{2- (n-2)\sigma}} \|Q\|_{L^2(\RR^n)}^2  \equiv \Big( 1+ \frac{n}{2(1-s_c)}\Big) \|Q\|^2_{L^2(\RR^n)}  . \label{E:normQ2}
\end{align}
For convenience, we record the energy of the ground state in terms of its mass or its gradient,
\begin{align}\label{H_Q}
H(Q)= & \frac{s_c}{2(1-s_c)} \,
\|Q\|_{L^2(\RR^n)}^2  
\quad \mbox{or} \quad 
H(Q)=  \frac{s_c}{n} \,
\| \nabla Q \|_{L^2(\RR^n)}^2.
\end{align}

Since the ground state $Q$ is the minimizer of the Gagliardo-Nirenberg inequality \eqref{GN}, 
we deduce, together with \eqref{E:normQ1} and \eqref{E:normQ2}, the optimal constant $C_{GN}$ in 
\eqref{GN}:  
\begin{equation}\label{constant_GN}
\hspace{2cm}  C_{GN} = \frac{K}{\|Q\|_{L^2(\RR^n)}^{2\sigma}}, \quad \mbox{\rm where } ~ 
K=\frac{2(\sigma+1) [ 2-(n-2)\sigma ]^{\frac{n\sigma -2}{2}}}{(n\sigma)^{\frac{n\sigma}{2}}}.
\end{equation} 
One can similarly deduce a sharp constant for the Sobolev inequality \eqref{E:Sobolev}, but this is not needed in this paper.

In the $L^2$-critical case ($s_c=0$ or $\sigma=\frac2{n}$), note that $K=\sigma+1$ and $C_{GN} = \tfrac{\sigma+1}{\|Q\|^{4/n}}$, so we have
\begin{equation}\label{E:sc=0}
\| \nabla u(t)\|_{L^2(\RR^n)}^2 \Big(1 - {\|Q\|_{L^2(\RR^n)}^{-4/n}} \; 
\|u (t)\|_{L^2(\RR^n)}^{4/n} \Big) \leq 2\, H(u(t)), 
\end{equation}
which guarantees the boundedness of the $\dot{H}^1$ norm of the solution $u(t)$ (i.e., the $L^2$ norm of $\nabla u(t)$),  
provided that the mass is conserved (or at least upper bounded) as well as the energy; then global well-posedness follows, assuming 
\begin{equation}\label{E:mass-W}
\|u_0\|_{L^2(\mathbb R^n)}^2<\|Q\|^2_{L^2(\mathbb R^n)},
\end{equation}
a well-known condition for the global well-posedness in the deterministic $L^2$-critical case, \cite{W83}.  
Unlike the additive noise case, where neither mass nor energy are conserved, the results obtained for a multiplicative stochastic perturbation, for which mass is a.s. conserved, 
will resemble the  deterministic NLS.

From the definition of energy \eqref{M-H} and the Gagliardo-Nirenberg inequality, 
we deduce the following bound on the gradient in the intercritical range $0< s_c <1 $,
\begin{equation}\label{upper_gradient}
\|\nabla u (t)\|_{L^2(\RR^n)}^2 \leq 2 H(u(t)) + \frac{ C_{GN}}{\sigma +1} \|\nabla u(t)\|^{n\sigma}_{L^2(\RR^n)} \|u(t)\|_{L^2(\RR^n}^{2-(n-2)\sigma},
\end{equation}
which implies (denoting by $B: = \frac{C_{GN}}{\sigma + 1} $ and recalling $\alpha = \frac{1-s_c}{s_c}$) 
\begin{equation}\label{E:X}
X(t)^2 - B X(t)^{{n\sigma}} \leq 2 H(u(t)) M(u(t))^{\alpha},
\end{equation} 
where  $X(t)$ is defined as
\begin{equation}\label{E:defX}
X(t) \defeq   \|\nabla u(t)\|_{L^2(\RR^n)} \|u(t)\|_{L^2(\RR^n)}^{\alpha}, 
\end{equation}
and is scale-invariant in the deterministic case, see depiction on the left of Figure \ref{F:1}.

Here, one may note that a smallness of $X(t)$ in this intercritical case would guarantee global existence of the solution $u(t)$ in the deterministic framework, 
since the right-hand side in \eqref{E:X} is bounded by the conserved quantities in that case. However, since the energy is not conserved in the stochastic framework and the mass is only (a.s.) 
conserved in the case of a multiplicative (Stratonovich) perturbation, we have to control the time dependence of the energy $H(u(t))$ (as well as the mass in the additive stochastic perturbation). 
The arguments used below are inspired by \eqref{E:X}, but the proofs and statements of the results are more intricate compared to the $L^2$-critical case, especially for an additive stochastic 
perturbation.

\section{Maximal existence time - Multiplicative noise}\label{S:M}

In this section, we study the case of the nonlinear Schr\"odinger (NLS) equations subject to a multiplicative random perturbation, which is defined in terms of a Stratonovich integral. The driving noise is white in time and has spatial regularity. As proved in \cite[Prop. 4.4]{deB_Deb_H1}, the noise given via Stratonovich integral implies that the mass is a.s. preserved, which has a physical meaning. 

\subsection{Preliminaries on multiplicative noise and local well-posedness.} \label{S:noise-multi}
Let $(\Omega, {\mathcal F},P)$ be a probability space endowed with a filtration $({\mathcal F}_t)_{t\geq 0}$.
Let $(\beta_k)_{k\in \NN}$ be a sequence of independent real-valued Brownian motions on $[0,\infty)$ with respect to the filtration 
$({\mathcal F}_t)_{t\geq 0}$ and let $(e_k)_{k\in \NN}$ be an orthonormal basis of the set $L^2(\RR^n;\RR)$ of real-valued square integrable functions on $\RR^n$. We define a Wiener process on the space $L^2(\RR^n;\RR)$ as 
\begin{equation} \label{R-noise}
W(t,\cdot,\omega)=\sum_{k=0}^\infty \beta_k(t,\omega) \phi \, e_k(\cdot),
\end{equation}
where $\phi$ is a bounded operator from $L^2(\RR^n;\RR)$ into itself. 

Note that if $\phi$ is defined through a kernel ${\mathcal{K}}$, that is, for any square integrable function $u$, 
$\phi u(x)=\int_{\RR^n}{\mathcal K}(x,y) u(y) dy$, 
then the correlation function of the noise is 
\[ \EX\Big( \frac{\partial W}{\partial t}(t,x) \frac{\partial W}{\partial t}(s,y)\Big) = c(x,y) \delta_{t-s}, \quad 
\mbox{\rm with} \quad c(x,y) = \int_{\RR^n} {\mathcal K}(x,z) {\mathcal K}(y,z) dz.\]

We suppose that the operator $\phi$ satisfies the following assumption (H1):
\smallskip

\noindent {$\bullet$ \bf Condition (H1)} ({\it Regularity of the multiplicative noise})
\begin{equation}\label{H1} 
\hspace{1cm}(\rm {\bf H1}) \hspace{2cm}
\phi \in L^{0,1}_{2,\RR} \cap 
R \big(L^2_{\RR}(\RR^n), W_{\RR}^{1,\kappa} (\RR^n)\big). \hspace{5cm}
\end{equation} 

Here, $L^{0,1}_{2,\RR}$ denotes the set of Hilbert-Schmidt 
operators from $L^2_{\RR}(\RR^n)$ into itself and 
$ R \big(L^2_{\RR}(\RR^n), W_{\RR}^{1,\kappa} (\RR^n)\big)$ denotes the set of $\gamma$-radonifying operators from $L^2(\RR^n;\RR)$  
to the Sobolev space
$W^{1,\kappa}(\RR^n)$ for
some $\kappa >0$ to be specified later.

Set
$$
\|\phi\|_\kappa \defeq \|\phi\|_{L^{0,1}_{2,\RR}} + \|\phi\|_{ R(L^2_{\RR}(\RR^n), W_{\RR}^{1,\kappa}(\RR^n) )}
$$
and
$$
F_\phi(x)=\sum_{k\geq 0} \big( \phi e_k(x) \big)^2.
$$
We consider the complex-valued process $u(t)$, solution to the initial value problem of the following stochastic NLS equation
\begin{align}\label{NLS_Stra}
i d_t u(t) - \big( \Delta u(t) + |u(t)|^{2\sigma} u(t) \big) dt = &\;  u(t) \circ dW(t) \\
\defeq & \; u(t) dW(t) - \frac{i}{2} F_\phi u(t) dt, \label{NLS_Ito} \\
u(0)=&\; u_0.  \nonumber 
\end{align}
In the equation \eqref{NLS_Stra} the term on the right-hand side $u(t) \circ dW(t)$ denotes the Stratonovich integral,  and in \eqref{NLS_Ito} the term $u(t) dW(t)$ is the classical It\^o integral; the term
$\frac{i}{2} F_\phi u(t) dt$ is the Stratonovich-It\^o correction term. Note that the factor $-i$ is due to the fact that, before computing  this correction, we  have to multiply the equation \eqref{NLS_Stra} by $-i$, and afterwards, multiply the corresponding equation  by $i$ 
to deduce \eqref{NLS_Ito}. 
\smallskip

We make a stronger assumption on the
covariance operator $\phi$ of the stochastic perturbation in order to describe the time evolution of the expected maximal energy of the solution,
namely, $\EE(\sup_{s\leq t} H(u(s)))$, see Lemma \ref{time_evol_H_2.4} below.
For that 
we describe a property of $\gamma$-radonifying operators.  The following result is due to Kwapie\'n \& Szyma\'nski \cite{Kwa_Szy} (see also \cite[Thm 3.5.10]{Bog} and \cite[Thm 3.23]{vanNee}).

\begin{Th}[\cite{Kwa_Szy}]\label{Kwapien}
Let ${\mathcal H}$ be a separable Hilbert space, ${\mathcal B}$ be 
a Banach space and $L$ be a $\gamma$-radonifying operator
from ${\mathcal H}$ to ${\mathcal B}$. Then there exists an orthonormal basis $\{ e_k\}_{k\geq 0}$ of ${\mathcal H}$ such that
\[ \sum_{k\geq 0} \| L e_k\|_{\mathcal B}^2<\infty.\]
\end{Th}

We next pose a boundedness assumption 
on the noise, which is helpful to control some terms in the energy in order to have a global solution in $H^1$ 
in the intercritical case. 
\medskip

\noindent {$\bullet$ \bf Condition (H2)} ({\it Boundedness 
of the driving noise)} 
\smallskip

There exists an orthonormal basis $\{ e_k\}_{k\geq 0}$ of $L^2_{\RR}(\RR^n)$ such that
$  f^1_\phi: =\sum_{k\geq 0} |\nabla(\phi e_k) |^2$ is a bounded function: 
\begin{equation}\label{M_phi}
(\rm \bf {H2}) \hspace{2cm} M_\phi \defeq \| f^1_\phi\|_{L^\infty(\RR^n)} =\sup_{x\in \RR^n}  \sum_{k\geq 0} |\nabla(\phi e_k) (x) |^2 < \infty. \qquad
\qquad 
\end{equation} 

\begin{remark}
It is easy to see that if $\phi\in L_{2,\RR}^{0,1}$  is 
$\gamma$-radonifying from $L^2_{\RR}(\RR^n)$ to $\dot{W}_{\RR}^{1,\infty}(\RR^n)$, it satisfies the condition {\bf (H2)}.
In that case, using Theorem \ref{Kwapien}, one deduces the existence of an orthonormal basis $\{ e_k\}_{k\geq 1}$  of $L^2_{\RR}(\RR^n)$ such that 
$M_\phi \leq \sum_{k\geq 0} \| \nabla (\phi e_k)\|_{L^\infty(\RR^n)}^2 <\infty$
and that  $\phi$ is $\gamma$-radonifying from 
$L^2_{\RR}(\RR^n)$ to any $W_{\RR}^{1,\kappa}$ space
for $\kappa \in [2,+\infty)$.
\end{remark}

\smallskip

To recall some known results on the local well-posedness 
of solutions to the SNLS equation with multiplicative noise \eqref{NLS_Stra}, we introduce some notation.  
\smallskip

\noindent {$\bullet$ \bf Condition (H3)} ({\it Range of nonlinearity powers  $\sigma$}) 
\smallskip

Let the power of nonlinearity $\sigma$ satisfy the following assumption: 
\begin{equation*}\label{H2} 
(\rm \bf {H}3) \hspace{2cm}
\left\{ \begin{array}{ll}
\sigma >0 & \mbox{ \rm if } \; n=1 \; \mbox{ \rm or }\; n=2, \\
0<\sigma < 2& \mbox{ \rm if }\;  n=3,\\
\frac{1}{2}\leq \sigma <\frac{2}{n-2}
& \mbox{ \rm if }\; n=4,5.
\end{array} \right. \qquad \qquad 
\end{equation*}
Note that the above range of $\sigma$ includes the $L^2$-critical and the intercritical range ($0 \leq s_c < 1$) 
in dimensions $1\leq n \leq 4$, and in dimension $n=5$ an extra restriction $\sigma \geq \frac{1}{2}$ makes part of the intercritical range $\frac25 \leq \sigma < \frac12$ not accessible yet by the local well-posedness (i.e., the local theory is only available for the range $\frac12 \leq s_c <1$ in 5d, see below Theorem \ref{local_wp}).
We therefore set the notation for range $\mathfrak R$ for the $L^2$-critical powers of nonlinearity as well as for the intercritical range of powers, based on the condition {\bf (H3)}:
\begin{equation}\label{R:crit}
\mathfrak R_{crit} = 
\left\{ \begin{array}{ll}
\sigma = \frac2{n} & \mbox{ \rm if } \; n=1, 2,3,4, 
\end{array} \right. \qquad \qquad
\end{equation}
\begin{equation}\label{R:inter}
\mathfrak R_{inter} = 
\left\{ \begin{array}{ll}
\frac2{n} < \sigma <\infty & \mbox{ \rm if }\;  n=1, 2, 
\\
\frac2{n}<\sigma < \frac2{n-2}& \mbox{ \rm if }\;  n=3,4, \\
\frac{1}{2}\leq \sigma <\frac{2}{n-2} 
& \mbox{ \rm if }\; n=5.
\end{array} \right.
\end{equation}

\bigskip

A pair $(r,p)$ of positive numbers is referred to as an admissible pair (or a Strichartz pair for the Lebesgue space $L^r_t L^p_x$) if 
\begin{equation}\label{E:pair}
\qquad \qquad \frac{2}{r} + \frac{n}{p}=\frac{n}2, 
\quad \mbox{provided} \quad p \geq 2. 
\end{equation}

We are now ready to state Theorem 4.1 from \cite{deB_Deb_H1}, which gives sufficient conditions for the existence and uniqueness of a local solution $u$ to the SNLS with multiplicative noise \eqref{NLS_Stra} in dimensions one to five. 
\smallskip

\begin{Th}	(\cite[Thm 4.1]{deB_Deb_H1})\label{local_wp}
Let $\phi$ satisfy Condition {\bf (H1)} with $\kappa >  2n$  and  $\sigma$ satisfy Condition {\bf (H3)}.
Then there exists an admissible pair $(r,p)$ such that for any ${\mathcal F}_0$-measurable initial condition $u_0$ 
taking values in $H^1(\RR^n)$, there exists
a stopping time $\tau^*(u_0)>0$ and a unique solution to \eqref{NLS_Stra} starting from $u_0$, which belongs a.s. to  $C([0;\tau];H^1)
\cap L^r([0,\tau] ; W^{1,p})$ for every  stopping time $\tau<\tau^*(u_0)$. Moreover, we have a.s.
\[ \tau^*(u_0,\omega)=+\infty \quad \mbox{\rm or } \quad \limsup_{t\to \tau^*(u_0,\omega)} \|u(t)\|_{H^1}=+\infty.\]
\end{Th}

\subsection{Estimates for mass and energy}

Proposition 4.4 in \cite{deB_Deb_H1} proves that under the conditions of Theorem \ref{local_wp} and due to the fact that 
in the equation \eqref{NLS_Stra} we are using a  Stratonovich integral, 
we have a.s. mass conservation of the local solution $u(t)$ to this equation 
 for all times  up to the stopping time $\tau^\ast(u_0)$. 
 More precisely,
\begin{equation}\label{mass_conserv}
M\big( u(t) \big) = M\big( u(0) \big) \quad \mbox{\rm a.s. for any }\; t<\tau^*(u_0).
\end{equation}
Furthermore, Proposition 4.5 in \cite{deB_Deb_H1} proves that under the assumptions of Theorem \ref{local_wp}, one has 
the following time evolution of the energy of the local solution $u(t)$ to the equation \eqref{NLS_Stra}:
\begin{align}\label{energy}
H\big( u(\tau) \big) = &\; H\big( u_0\big) - \mbox{\rm Im } \!\int_{\RR^n} \!\int_0^\tau\!  \bar{u}(s,x)  \nabla u(s,x)  \cdot \nabla dW(s) dx \nonumber \\
&\; +\frac{1}{2} \sum_{k\geq 0} \int_0^\tau \int_{\RR^n} |u(s,x)|^2 \, |\nabla (\phi e_k)(x)|^2 dx ds
\end{align}
for any stopping time $\tau < \tau^*(u_0)$ a.s. 

Note that if $\phi$ satisfies the {\bf (H2)}  condition, we can bound the last term in \eqref{energy}  as
$$
\sum_{k\geq 0}    \int_0^\tau \int_{\RR^n} |u(s,x)|^2 \, |\nabla \phi e_k(x)|^2 dx ds \leq M_\phi\, M\big( u_0) \, \tau \quad \mbox{for}
\quad  \tau < \tau^*(u_0).
$$
Such bounds will be useful in the energy estimates later.  
\smallskip

Using the mass conservation \eqref{mass_conserv} and the time evolution of energy \eqref{energy}, as well as the Gagliardo-Nirenberg inequality \eqref
{GN},
Theorem 4.6 in \cite{deB_Deb_H1} shows that under the assumptions of Theorem \ref{local_wp}, 
if $\sigma < \frac{2}{n}$ ($L^2$-subcritical case), then the solution of \eqref{NLS_Stra} given in Theorem \ref{local_wp} is global, 
that is, $\tau^*(u_0)=+\infty$ a.s.

We next suppose that the assumptions of Theorem \ref{local_wp} are satisfied with  $\sigma \in \mathfrak R_{crit}$ or $\sigma \in \mathfrak R_{inter}$. 
We aim to prove either global in time well-posedness or well-posedness on some time interval $[0,T]$ 
with strictly positive probability. 
This will require certain constraints on $M(u_0)$, $H(u_0)$ and $\| \nabla u_0\|_{L^2(\RR^n)}$ relative to
$M(Q)$, $H(Q)$ and $\|\nabla Q\|_{L^2(\RR^n)}$, where $Q$ is the ground state as in \eqref{E:Q} and discussed in Section \ref{S:Q}. The relations are similar to those in the deterministic case and while the approach is inspired by that case too, treating the stochastic equation brings several challenges, in particular, because the energy is not conserved.
\medskip

The following lemma provides bounds for the expected value of the energy of the solution uniformly in time.
As in the deterministic setting of the intercritical case,   the energy is multiplied by some power of mass. For the next lemma this power can be any positive value, however, 
later this power will be chosen as $\alpha$ in \eqref{E:alpha} so that the mass-energy expression would correspond to a similar quantity in the deterministic framework of the intercritical case.

\begin{lemma}	\label{time_evol_H_2.4}
Let the assumptions of Theorem \ref{local_wp} be satisfied and suppose that $\phi$ also satisfies the condition {\bf (H2)}.  Let  
$u_0$ be ${\mathcal F}_0$-measurable,   $H^1(\RR^n)$-valued such that  $\EE(\|u_0\|_{H^1(\RR^n)}^2\big) <\infty$   
and let $u$ denote the unique local solution of \eqref{NLS_Stra}. Then for
$\tau < \tau^*(u_0)$, we have that for any $\alpha \geq 0$
\begin{align}	\label{upp_bd_H}	
\EX \Big(M(u_0)^\alpha  \sup_{s\leq \tau} H\big( u(s)\big) \Big) \leq & \; \EX \big(M(u_0)^\alpha H(u_0)\big) +
 \frac{1}{2} M_\phi  \, \EX \big( M\big( u_0)^{\alpha+1} \tau\big)  \nonumber \\
 & \; + 3 \sqrt{M_\phi} \, \EX \Big( \sqrt{\tau} M( u_0)^{\alpha + \frac{1}{2}}  \sup_{s\leq \tau} \| \nabla u(s) \|_{L^2(\RR^n)}\Big).
\end{align}
\end{lemma}

\begin{remark}
Even though the following  assumptions are not required in the statement of this lemma, to use it we will further require  
$\EX \big(M(u_0)^\alpha H(u_0)\big) <\infty$ and $\EX \big( M\big( u_0)^{\alpha+1}\big)<\infty$.
\end{remark}

\begin{proof} [Proof of Lemma \ref{time_evol_H_2.4}]
Multiplying \eqref{energy} by $M(u_0)^\alpha $, which is ${\mathcal F}_0$-measurable, using the Davis inequality, 
 the Cauchy-Schwarz inequality, then  the mass conservation \eqref{mass_conserv}, and condition {\bf (H2)}, we obtain for $\tau < \tau^*(u_0)$ a.s.
\begin{align*}	
\EX \Big( & M(u_0)^\alpha \sup_{s\leq \tau} H\big( u(s)\big) \Big) \leq  
\EX \big( M(u_0)^\alpha H( u_0)\big) +  \frac{1}{2} M_\phi   \EX \Big(M( u_0)^\alpha \int_0^\tau M(u(s)) ds \Big)  \\ 
&
 \qquad \qquad \qquad \quad   + 3 \EX \Big[ \Big\{ \int_0^\tau \sum_{k\geq 0} 
  \Big( \int_{\RR^n}  M(u_0)^{\alpha}\bar{u}(s,x) \, \nabla u (s,x) \cdot \nabla (\phi e_k)(x) dx \Big)^2 ds
 \Big\}^{\frac{1}{2}} \Big]  \\ 
\leq  & \; \EX \big(M(u_0)^\alpha H( u_0)\big) +  \frac{1}{2} M_\phi  \EX \big(M( u_0)^{\alpha+1} \tau \big) \\
&  \qquad \qquad \qquad \quad   + 3  \EX \Big[ \Big\{ \int_0^\tau  M(u_0)^{2\alpha} 
 \Big( \int_{\RR^n} |\bar{u}(s,x)|^2  f^1_\phi(x) dx \Big) \Big( \int_{\RR^n}|\nabla u(s,x)|^2  dx \Big)  ds \Big\}^{\frac{1}{2}} \Big] \nonumber  \\
\leq  &\;  \EX \big( M(u_0)^\alpha H( u_0)\big) +  \frac{1}{2} M_\phi  \EX \big(M( u_0)^{\alpha +1} \tau \big)  + 3 \sqrt{M_\phi}  \EX \Big( \Big\{\!\! \int_0^\tau 
\!\! M\big( u_0\big)^{2\alpha +1} 
 \| \nabla u(s)\|_{L^2}^2 ds \Big\}^{\frac{1}{2}} \Big) \\
 \leq & \; \EX \big(M(u_0)^\alpha  H( u_0)\big) +  \frac{1}{2} M_\phi  \EX \big(M( u_0)^{\alpha +1}  \tau \big)  + 3 \sqrt{M_\phi} 
\EX \Big( \sqrt{\tau} M( u_0)^{\alpha + \frac{1}{2}} \sup_{s\leq \tau} \| \nabla u(s) \|_{L^2}\Big),
\end{align*} 
which completes the proof of the lemma.
\end{proof}
Now that we have either conservation of mass \eqref{mass_conserv} or bounds on energy \eqref{upp_bd_H}, we study the global behavior of solutions in each $L^2$-critical and intercritical cases. 

\subsection{The $L^2$-critical case}\label{mass_critical_Stra}
In this section we consider $s=0$, $\sigma \in \mathfrak R_{crit}$, that is, $\sigma = \frac{2}{n}$ with $n\leq 4$, where the local well-posedness 
is known as discussed in condition (H3); we remark that our results also hold conditionally (upon the local well-posedness) for $n>4$. 

We prove the following global existence result, first considering 
{\it deterministic} initial data (but evolved under the stochastic NLS flow) and then {\it random} initial data (also evolved under the SNLS flow). 

\begin{Th}\label{gwp_critical}
Let  $u_0\in H^1(\RR^n)$, $\sigma \in \mathfrak R_{crit}$, and $\phi$ satisfy condition {\bf (H1)} with $\kappa > 2n$ and condition {\bf (H2)}. 
Suppose that 
\begin{equation}\label{E:mass-condition}
\|u_0\|_{L^2(\RR^n)}< \|Q\|_{L^2(\RR^n)}.
\end{equation}
Then the unique solution $u(t)$ to \eqref{NLS_Stra} is global,
that is, $\tau^*(u_0)=+\infty$ a.s. .
\end{Th} 

\begin{proof}
In the proof, despite the fact that $u_0$ is deterministic, we write upper estimates in terms of $\EE(M(u_0))$, which later is used for a random initial data. 
Recalling the inequality \eqref{E:sc=0} in the case
when $\sigma = \frac{2}{n}$, 
we let $\mu = 1- \Big( \frac{\|u_0\|_{L^2(\RR^n)}}{\|Q\|_{L^2(\RR^n)}} \Big)^{\frac{4}{n}}$; 
then the condition \eqref{E:mass-condition} guarantees that $\mu \in (0,1)$.

For any $R>0$ and any stopping time $\tau < \tau^*(u_0)$ a.s. set
$$
\tau_R=\inf\{ t\geq 0 \; : \; \|u(t)\|_{\dot{H}^1(\RR^n)} \geq R\} \wedge \tau.
$$ 

Using the mass conservation \eqref{mass_conserv},
we deduce that if $u$ is the solution to \eqref{NLS_Stra}, then 
the inequality \eqref{E:sc=0} yields 
\begin{equation}\label{upper_gradient_masscrit}
\| \nabla u(t) \|_{L^2(\RR^n)}^2 \leq \frac{2 H( u(t) )}{\mu} \quad \mbox{a.s.}
\end{equation}
for any time $t<\tau^*(u_0)$.
Substituting this upper estimate in \eqref{upp_bd_H} with $\alpha =0$, we deduce that 
for any stopping time $\tau <\tau^*(u_0)$ a.s., $R>0$  and $\epsilon >0$, we obtain
\begin{align*}
\EX \Big( &\sup_{s\leq \tau_R} H\big( u(s)\big) \Big) \leq  H( u_0)+  \tfrac{1}{2} M_\phi   \EX \big( M( u_0)\tau_R\big)
+ 3 \sqrt{M_\phi} 
\EX \Big( \sqrt{\tau_R} \sqrt{M\big( u_0\big)} \sup_{s\leq \tau_R}  \sqrt{\tfrac{2}{\mu} \, H\big( u(s)\big)} \Big)  \nonumber \\
&\leq H( u_0) +  \tfrac{1}{2} M_\phi \EX \big( M( u_0)\tau_R\big)
 + 3 \sqrt{M_\phi}  \sqrt{\tfrac{2}{\mu}}   \big[  \EX \big( M( u_0)\tau_R\big) \big]^{\frac{1}{2}}
\Big[ \EX \Big(  \sup_{s\leq \tau_R}  H\big( u(s)\big) \Big)\Big]^{\frac{1}{2}}  \nonumber \\
&\leq  H(u_0) + M_\phi  \EX \big(M( u_0) \tau_R\big) \, \Big[ \frac{1}{2} + \frac{9}{4} \, \frac{2}{\mu} \, \frac{1}{\epsilon} \Big] + \epsilon \,
\EX \Big(  \sup_{s\leq \tau_R}  H\big( u(s)\big) \Big),
\end{align*} 
where in the second line we used the Cauchy-Schwarz inequality and in the last one the Young inequality.  
 
From the definition of the Hamiltonian, it trivially follows that $H(u)\leq \frac{1}{2} \|\nabla u\|_{L^2(\RR^n)}^2$,
and hence, we have $\EX \big[ \sup_{s\leq \tau_R} H\big( u(s)\big) \big] \leq \frac{R^2}{2}<\infty$. 
This implies that for $\epsilon \in (0,1)$  and any $R>0$
\begin{equation}    \label{upper_H}
\EX  \Big( \sup_{s\leq \tau_R} H\big( u(s)\big) \Big) \leq   \frac{1}{1-\epsilon} \Big[  H(u_0) + M_\phi  \EX \big( M( u_0) \tau_R \big) \Big( \frac{1}{2} +
\frac{9}{2\mu \epsilon}\Big) \Big].
\end{equation}
As $R\to +\infty$, we have $\tau_R\to \tau$ a.s. and $\sup_{s\leq \tau_R} H\big( u(s)\big) $ increases to  $\sup_{s\leq \tau} H\big( u(s)\big) $ a.s.. 
Since $u_0$ has finite energy,
the monotone convergence theorem implies that for $\epsilon \in (0,1)$, 
\[ \EX  \Big( \sup_{s\leq \tau} H\big( u(s)\big) \Big) \leq   \frac{1}{1-\epsilon} \Big[  H(u_0) + M_\phi  \EX \big( M( u_0)\tau\big) \Big( \frac{1}{2} +
\frac{9}{2\mu \epsilon}\Big) \Big].
\]
Hence, using  \eqref{upper_gradient_masscrit}, we deduce that  for every stopping time $\tau <\tau^*(u_0)$  and $T>0$, we have 
\begin{equation}    \label{Esup_H^1} 
\EX  \Big( \sup_{s\leq \tau\wedge T } \|\nabla u(s)\|_{L^2}^2  \Big) \leq  
\frac{2}{\mu(1-\epsilon)} \Big[  H(u_0) +  M_\phi  T \, \EX \big(  M( u_0)\big) \Big( \frac{1}{2} +
\frac{9}{2\mu \epsilon}\Big) \Big].
\end{equation}
Suppose that $P(\tau^*(u_0)<\infty)>0$. Then there exists $T_0>0$ such that $P(\tau^*(u_0)< T_0)=a>0$. Let $\{ \tau_N\}_N$ be a sequence of stopping times increasing to $\tau^*(u_0)$. Since mass is preserved, almost surely, using Theorem~\ref{local_wp}, we deduce that for $N$ large enough, we have $P(\sup_{s\leq \tau_N\wedge T_0} \|\nabla u(s)\|_{L^2}^2\geq N)\geq \frac{a}{2}$. 
This implies $\EX \big( \sup_{s\leq \tau_N \wedge T_0} \|\nabla u(s)\|_{L^2}^2 \big) \geq \frac{Na}{2}$, which contradicts \eqref{Esup_H^1} for $N$ large enough. 

Therefore,   $\tau^*(u_0)=+\infty$ a.s., and we have global well-posedness in the $L^2$-critical case. 
\end{proof}

The next result gives sufficient conditions for {\it random} initial data to have a.s. global well-posedness.
\begin{Th}\label{gwp_critical_randomIC}
Let 
$u_0$ be an $H^1(\RR^n)$-valued, $\sigma \in \mathfrak R_{crit}$, ${\mathcal F}_0$-measurable random initial condition, $\phi$ satisfy condition {\bf (H1)} 
with $\kappa > 2n$ and condition {\bf (H2)}. 
Suppose  that for some constant $\nu \in (0,1)$
\begin{equation}\label{E:mass-nu} 
\|u_0\|_{L^2(\RR^n)}\leq  \nu \|Q\|_{L^2(\RR^n)} ~~\rm{a.s.}  
\end{equation} 
and $\EX(\| u_0\|_{H^1(\RR^n)}^2)<+\infty$. 
Then the unique solution $u(t)$ to \eqref{NLS_Stra} is global,
that is, $\tau^*(u_0)=+\infty$ a.s..
\end{Th} 
\begin{proof}
For an $\omega \in \Omega$, set $\tilde{\mu}(\omega) =  1- \Big( \frac{\|u_0(\omega)\|_{L^2(\mathbb R^n)}}{\|Q\|_{L^2(\mathbb R^n)}} \Big)^{\frac{4}{n}}$,
which by the assumption \eqref{E:mass-nu} implies that $\tilde \mu(w) \in (0,1)$. Even more precise, 
\eqref{E:mass-nu} yields $\tilde{\mu}(\omega) \geq \mu:= 1- \nu^{\frac{4}{n}}$ a.s., where the constant $\mu$ is positive.
Notice that since $\EX \big( \| \nabla u_0\|_{L^2}^2\big)<\infty$, we have $\EX \big( H(u_0)\big) <\infty$.
We then proceed as in the proof of Theorem \ref{gwp_critical}, replacing $H(u_0)$ by $\EX \big( H(u_0) \big)$ in upper estimates of the
energy of the solution  \eqref{upper_H}. 
Since $\EX \big( M(u_0)\big) <\infty$, we deduce that  for any stopping time $\tau < \tau^*(u_0)$ and any $T>0$, we have  $\EX\big( M(u_0) (\tau\wedge T) \big) < T \, \EX \big( M(u_0)\big)$.
Proceeding as in the proof of Theorem \ref{gwp_critical} completes this argument. 
\end{proof}

\subsection{Intercritical case}\label{inter_Stra}
In this section we consider the intercritical range $0<s_c<1$, which implies that due to the available local well-posedness (namely, the restriction in dimension 5) from Theorem 
\ref{local_wp}, unconditionally we can only consider the nonlinearity power $\sigma \in \mathfrak R_{inter}$.
Note that our results hold for the rest of the intercritical range {\it conditionally} upon the local existence. 
We next prove one of the main results of this section, which gives an estimate for the existence time of the solution.  
Similar to the deterministic case, we consider a quantity $X(t)$ similar to that in  \eqref{E:defX},  
and as in \eqref{E:X}, we multiply the energy by $M(u_0)^{\alpha}$ with $\alpha$ from \eqref{E:alpha}.

We first deal with a {\it deterministic} initial condition. 

\begin{Th} \label{th-wp-inter-multi}
Consider \eqref{NLS_Stra} with $\sigma \in \mathfrak {R}_{inter}$ and recall $\alpha$ from \eqref{E:alpha}.
Let   $u_0\in H^1(\RR^n)$  
be such that for some $\beta \in (0,1)$, we have  
\begin{equation}\label{Hyp_HM}
H(u_0) M(u_0)^\alpha = \beta  H(Q) M(Q)^\alpha 
\end{equation}
and
\begin{equation}\label{Hyp_gradM}
\| \nabla u_0\|_{L^2(\RR^n)} \|u_0\|_{L^2(\RR^n)}^\alpha < \|\nabla Q\|_{L^2(\RR^n)} \|Q\|_{L^2(\RR^n)}^\alpha.
\end{equation}

For $M_\phi$ as in \eqref{M_phi}, define 
\begin{equation}\label{T_max}
T^\ast = 
\frac{\big[\sqrt{n+{ 2s_c } (1 - \beta)/9\,} - \sqrt{n}\big]^2}{2(1-s_c)}
\frac{9}{M_\phi} \Big( \frac{M(Q)}{M(u_0)}\Big)^{\frac1{s_c}}.
\end{equation}

Then $P(\tau^*(u_0) >T)>0$ for $T<T^*$
and 
\begin{equation}\label{lower_E_tau}
\EX (\tau^*(u_0)) \geq  {T}^*.
\end{equation} 
Furthermore, the $L^2$ norm of the gradient remains uniformly bounded on $[0,t]$ for any $\displaystyle t \leq \sup_{\delta <1} \tau_\delta$, 
where $\tau_\delta$ is defined in \eqref{tau_delta} and the bound is given in \eqref{E:grad-bound1}.
\end{Th}
\begin{proof}
We first prove that the given assumptions on the initial condition imply a certain upper bound on the gradient of the solution.

For any stopping time $\tau < \tau^*(u_0)$ a.s., 
we introduce a quantity similar to \eqref{E:defX}, that is, 
\begin{equation}\label{E:X-1}		
X(\tau)=\sup_{s\leq \tau} \| \nabla u(s)\|_{L^2(\RR^n)} \|u_0\|_{L^2(\RR^n)}^\alpha. 
\end{equation}
Using the mass conservation \eqref{mass_conserv}, the definition of the energy \eqref{M-H} and the Gagliardo-Nirenberg inequality \eqref{GN}, 
in a similar fashion as for \eqref{upper_gradient} we deduce that for $t<\tau^*(u_0)$ a.s.
\begin{align*}
\| \nabla u(t)\|_{L^2(\RR^n)}^2& \| u(t)\|_{L^2(\RR^n)}^{2\alpha} =  \; 
\| \nabla u(t)\|_{L^2(\RR^n)}^2 \| u_0\|_{L^2(\RR^n)}^{2\alpha}\\
 \leq & \; 2 H\big( u(t)\big) \|u_0\|_{L^2(\RR^n)}^{2\alpha} + B
 \| \nabla  u(t)\|_{L^2(\RR^n)}^{n\sigma} \| u_0\|_{L^2(\RR^n)}^{2-(n-2)\sigma + 2\alpha} \quad \mbox{\rm a.s.},
\end{align*}
where $C_{GN}$ is  defined in \eqref{constant_GN}  and $B = \frac{C_{GN}}{\sigma +1}$ as in \eqref{E:X}.
The definition of $\alpha$ implies $2-(n-2)\sigma +2\alpha = \alpha n\sigma$, and hence, we get
$$
X(\tau)^2 \leq 2 \sup_{s\leq \tau} H\big( u(s)\big) \| u_0\|_{L^2(\RR^n)}^{2\alpha} 
+  B X(\tau)^{n\sigma},
$$
which we re-write as the following analog of \eqref{E:X}, 
\begin{equation}\label{E:X-2}
X(\tau)^2 - B X(\tau)^{n\sigma} \leq 2 M(u_0)^\alpha \sup_{s\leq \tau} H\big( u(s)\big).
\end{equation}
Take $\delta \in (\beta,1)$ and let 
\begin{equation}		\label{tau_delta}
\tau_\delta = \inf\Big\{ t\geq 0 \; : \; \sup_{s\leq t} M(u_0)^\alpha H\big( u(s)\big) \geq \delta M(Q)^\alpha H(Q)\Big\} \wedge \tau. \
\end{equation}
Since $u\in C([0, \tau_\delta];H^1(\RR^n))$ a.s., from \eqref{E:X-2} and definition of $\tau_\delta$ we deduce  
\begin{equation}\label{ineq_delta}
X(\tau_\delta)^2 - B X(\tau_\delta)^{n\sigma} \leq 2\delta H(Q) M(Q)^\alpha\quad \mbox{\rm a.s.}. 
\end{equation}
Set $f(x)=\frac{1}{2} \big( x^2-Bx^{n\sigma}\big)$ for $x\geq 0$. 
The function $f$ achieves its maximum at 
\begin{equation}		\label{def_x*}
x^*=\Big( \frac{2}{Bn\sigma}\Big)^{\frac{1}{n\sigma -2}} >0.
\end{equation}
Recall that $n\sigma >2$ (and $B>0$), hence, the function $f$ is an inverted down parabola, which is increasing on the interval $(0,x^*)$ and decreasing
 on $(x^*,\infty)$. Using the definition of $B$ (and Pokhozhaev identities \eqref{E:normQ1}, \eqref{E:normQ2}, \eqref{H_Q}, see also the deterministic case 
 \cite{HR2007}, \cite{HR2008}), the values of $x^*$ and $f(x^*)$ are related with the ground state $Q$ as (see also Figure \ref{F:1})
\begin{equation}\label{identification}
x^*=\| \nabla Q\|_{L^2(\RR^n)} \|Q\|_{L^2(\RR^n)}^\alpha \quad \mbox{\rm  and } \quad f(x^*)=  H(Q) M(Q)^\alpha.
\end{equation}

\vspace{-.5cm}
\begin{figure}[h!]
\includegraphics[width=0.49\hsize,height=0.32\hsize]{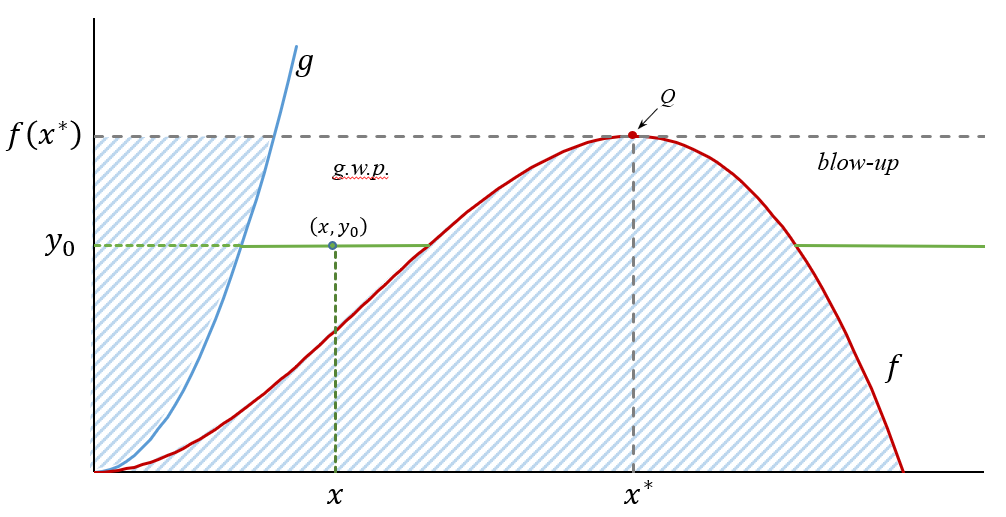}
\includegraphics[width=0.5\hsize,height=0.32\hsize]{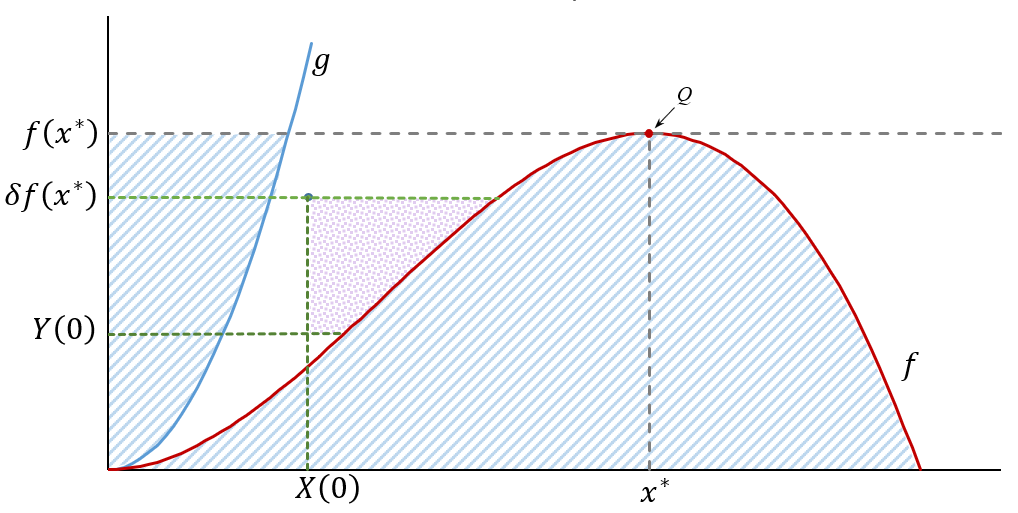}
\caption{\small {Comparison of {\it deterministic} (left) and {\it stochastic} (right) cases in the intercritical setting: $g(x) = \frac12 x^2$, $f(x) = \frac12 (x^2-Bx^{n\sigma})$,   
 $x^{\ast} = \|\nabla Q\|_{L^2} \|Q\|_{L^2}^\alpha$, $f(x^\ast) = H(Q)M(Q)^\alpha$. In the deterministic case initially starting at time $t=0$ at $x = \|\nabla u(t)\|_{L^2} \|u_0\|_{L^2}^\alpha$,
  the solution stays on the green line $y_0 = H(u_0)M(u_0)^\alpha$ (conserved in time), meaning that $\|\nabla u(t)\|_{L^2}$ is bounded for all time $t$, being trapped on the left of $f$
   (and thus, globally well-posed), while in the stochastic case starting at $X(0) = \|\nabla u_0\|_{L^2} \|u_0\|^\alpha_{L^2}$, $Y(0) =  H(u_0) M(u_0)^\alpha$, 
   the solution stays in the purple area up to time $\tau_\delta: = \inf\{s \geq 0: H(u(s))M(u_0)^\alpha \geq \delta f(x^\ast)\}$ with
    $X(\tau_\delta) = \sup_{s \leq \tau_\delta} \|\nabla u(s)\|_{L^2} \|u_0\|^\alpha_{L^2}$, $Y(\tau_\delta) = \sup_{s \leq \tau_\delta} H(u(s)) M(u_0)^\alpha \leq \delta f(x^\ast)$. } }
\label{F:1}
\end{figure}

Since  $X(0)=\| \nabla u_0\|_{L^2(\RR^n)} \|u_0\|_{L^2(\RR^n)}^\alpha <x^*= \|\nabla Q\|_{L^2(\RR^n)} \|Q\|_{L^2(\RR^n)}^\alpha$, 
 by a.s. continuity of $\| \nabla u(t)\|_{L^2(\RR^n)}$ on the (random) time
interval $[0, \tau_\delta]$, we deduce that $X(\tau_\delta)<x^*$ a.s. Therefore,  
\begin{equation}\label{Q_gradQ}
\|\nabla u(s)\|_{L^2(\RR^n)} \|u_0\|_{L^2(\RR^n)}^\alpha < \| \nabla Q\|_{L^2(\RR^n)}
\|Q\|_{L^2(\RR^n)}^\alpha\quad \mbox{\rm a.s. for}~~ s\leq \tau_\delta,
\end{equation}
see also Figure \ref{F:1} right.
We point out that due to the estimate \eqref{Q_gradQ} on the time interval $[0,\tau_\delta]$, we have a uniform bound on the gradient of $u$, namely, 
\begin{equation}\label{E:grad-bound1}
\sup_{s\leq \tau^*(u_0)} \|\nabla u(s)\|^2_{L^2(\RR^n)} \leq \frac{n}{2(1-s_c)}  \frac{M(Q)^{1+\alpha}}{M(u_0)^\alpha}
\quad \mbox{a.s..}
\end{equation}

Let $T>0$. We next find constraints on $M_\phi$ to ensure that $P(\tau_\delta >T)>0$. The Markov inequality, the
 upper estimate \eqref{upp_bd_H} and the inequality \eqref{Q_gradQ}  imply
\begin{align}\label{upper_P}
P(\tau_\delta & \leq T)= P\Big( \sup_{s\leq T\wedge \tau_\delta} M(u_0)^\alpha H\big(u(s)\big)  \geq \delta M(Q)^\alpha H(Q)\Big)  \nonumber \\
\leq & \frac{1}{\delta M(Q)^\alpha H(Q)} \Big[ M(u_0)^\alpha H(u_0) + \frac{1}{2} M_\phi M(u_0)^{\alpha +1} T +
3 \sqrt{M_\phi} M(u_0)^{ \alpha +\frac{1}{2}}  \sqrt{T}  \frac{\|Q\|_{L^2}^\alpha}{\|u_0\|_{L^2}^\alpha} \|\nabla Q\|_{L^2} \Big] \nonumber \\
\leq & \frac{\beta}{\delta} + \frac{1}{2} M_\phi T \, \frac{M(u_0)^{\alpha + 1}}{\delta M(Q)^\alpha H(Q)} + 3 \sqrt{M_\phi T} \, 
\frac{ M(u_0)^{\frac{\alpha+ 1}{2}} \|Q\|_{L^2}^\alpha \|\nabla Q\|_{L^2}}{\delta M(Q)^\alpha  H(Q)} .
\end{align}
Note that on the right-hand side we have a quadratic polynomial in the variable $Y \defeq \sqrt{M_\phi T}$. 

Given $\gamma \in \big( \frac{\beta}{\delta},1\big)$, we would like to 
find $\epsilon(\gamma, \delta )>0$ such that 
\begin{equation}		\label{**}
 P(\tau_\delta \leq T)\leq \gamma \quad \mbox{\rm for } \; M_\phi T\leq \epsilon(\gamma, \delta).
\end{equation}
Denoting 
\begin{equation}		\label{a_b_c}
a=\frac{1}{2} M(u_0)^{\alpha+1}, \; b=3 M(u_0)^{ \frac{\alpha+1}{2}} \|Q\|_{L^2}^\alpha \| \nabla Q\|_{L^2} \; ~\mbox{\rm and } \; ~
c=\Big( \gamma - \frac{\beta}{\delta}\Big) \delta M(Q)^\alpha H(Q),
\end{equation}
the inequality $P(\tau_\delta \leq T) \leq \gamma$ can be rewritten as $aY^2+bY-c \leq 0$ for the above variable $Y$. 
Using \eqref{E:normQ1}, \eqref{E:normQ2} and \eqref{H_Q}, we rewrite  
\begin{equation}\label{E:bc}
b=3 \Big(\frac{n}{2(1-s_c)} \Big)^{\frac12}   
\big( M(u_0) M(Q) \big)^{ \frac{\alpha +1}{2}}  \quad \mbox{\rm and}\quad  
c=(\delta \gamma-\beta ) \frac{s_c}{2(1-s_c)}
M(Q)^{\alpha +1}.
\end{equation} 
The discriminant $D(\gamma, \delta)$ of the second degree polynomial $aY^2+bY-c$ is
$$
D(\gamma, \delta)= \frac{9n + 2 s_c(\gamma \delta - \beta)}{2(1-s_c)}
\big( M(u_0)M(Q) \big)^{\alpha+1} >0,
$$  
since $s_c<1$ and $\gamma > \beta/\delta$. 
 
Furthermore, 
$b^2= \frac{9n}{2(1-s_c)} \big( M(u_0) M(Q) \big)^{\alpha +1}$, and thus, $b^2 < D$. 
Hence, the roots of this polynomial are $Y_1(\gamma, \delta)=\frac{-b-\sqrt{D(\gamma, \delta)}}{2a}<0$ and $Y_2(\gamma, \delta)=\frac{-b+\sqrt{D(\gamma, \delta)}}{2a}>0$. For $Y\geq 0$ we have
$aY^2+bY-c\leq 0$ if and only if $Y\in [0, Y_2(\gamma, \delta)]$, where 
\begin{align}\label{E:Y2}
Y_2(\gamma, \delta)= 3
\Big( \frac{M(Q)}{M(u_0)}\Big)^{\frac{\alpha+1}{2}}
\frac{\sqrt{n+ { 2 s_c} (\gamma \delta - \beta)/9\,} - \sqrt{n}}{\sqrt{2(1-s_c)}}.
\end{align}
Note that as $\delta \to 1$, we have that $\gamma \to 1$ as well, and hence, 
$$ 
Y_2(\gamma, \delta) \to Y^* :=  
3 \Big(\frac{M(Q)}{M(u_0)}\Big)^{\frac{\alpha+1}{2}}  
\frac{\sqrt{n+ { 2 s_c} (1 - \beta)/9\,} - \sqrt{n}}{\sqrt{2(1-s_c)}}.
$$
Recalling that $Y^2 = M_\phi T$ and $\alpha+1 = \frac1{s_c}$ from \eqref{E:alpha}, we define 
$$
T^\ast =  \frac{(Y^\ast)^2}{M_\phi} =   \frac{(\sqrt{n+{ 2 s_c } (1 - \beta)/9\,} - \sqrt{n})^2}{2(1-s_c)}
\frac{9}{M_\phi} \Big( \frac{M(Q)}{M(u_0)}\Big)^{\frac1{s_c}}.
$$
Hence,  for $T<T^\ast$ 
we can find $\delta \in (\beta,1)$ and $\gamma \in \big(\frac{\beta}{\delta},1\big)$ such that  $\sqrt{M_\phi T} = Y_2(\gamma, \delta)$, 
which proves \eqref{T_max}. 

This provides an upper bound on the (deterministic) time $T^*$ such that
we have well-posedness of \eqref{NLS_Stra} on the time interval $[0,T^*)$ with positive probability, i.e., 
$P\big(\tau^*(u_0) \geq T^*\big) >0$.  
We remark that $T^*$ is  decreasing with respect to both the initial mass $M(u_0)$ and the parameter $\beta$ as $\beta \to 1$, which implies that $T^*$ inversely depends on 
$H(u_0) M(u_0)^{\alpha}$. It is also a decreasing
function of the intensity $M_\phi$ of the forcing noise, and as $M_\phi \to 0$, we have $T^*\to +\infty$, which leads to  global existence in the deterministic case. 
\medskip

Finally, we prove the inequality \eqref{lower_E_tau}, which obviously holds if $P(\tau_\delta <\infty)<1$, since in that case $\EX (\tau^*(u_0) )\geq \EX (\tau_\delta)=+\infty$. 
Thus, to obtain a lower bound of $\EX (\tau_\delta)$, we may assume  $\tau_\delta <\infty$ a.s..
By a.s. continuity, we have 
$$
\EX \big( \sup_{s\leq \tau_\delta} M(u_0)^\alpha H( u(s)) \big)=\delta M(Q)^\alpha H(Q).
$$
Plugging this into \eqref{upp_bd_H} and using  the upper bound  \eqref{Q_gradQ} and the Cauchy-Schwarz inequality, we obtain 
\begin{align}		\label{delta_MH_Q}
&\delta M(Q)^\alpha H(Q) \leq \beta M(Q)^\alpha H(Q) +  \tfrac12\,  M_\phi M(u_0)^{\alpha+1} \EX (\tau_\delta) 
+3 \sqrt{M_\phi} M(u_0)^{\frac{\alpha+1}{2}}  \EX (\sqrt{\tau_\delta})  \|Q\|_{L^2}^\alpha \|\nabla Q\|_{L^2}. 
\end{align} 
Set $Z=\sqrt{M_\phi \EX ( \tau_\delta)}$. Then the above inequality can be rewritten as 
${a}Z^2+b Z-\tilde{c} \geq 0$ for $a,b$ defined by \eqref{a_b_c} and $\tilde{c} = (\delta - \beta) M(Q)^\alpha H(Q)$. 
Hence, replacing $c$ by $\tilde{c}$ in the above computations \eqref{E:bc}-\eqref{E:Y2}, we deduce that we must have 
$M_\phi \EX ( \tau_\delta) \geq Z_2^2$, where 
\begin{equation}
Z_2(\delta)= 3
\Big( \frac{M(Q)}{M(u_0)}\Big)^{\frac{\alpha+1}{2}}
\frac{\sqrt{n+ { 2 s_c } (\delta - \beta)/9\,} - \sqrt{n}}{\sqrt{2(1-s_c)}}.
\end{equation}
Since $\EX (\tau^*(u_0)) \geq \EX (\tau_\delta)$ for any $\delta \in (\beta,1)$, taking $\delta \to 1$ we deduce \eqref{lower_E_tau}.
\bigskip
\end{proof}

We next extend the above results to a {\it random} initial condition. 

\begin{Th}\label{th-wp-inter-multi-randomIC}
Consider \eqref{NLS_Stra} with $\sigma \in \mathfrak R_{inter}$ and $\alpha$ as in \eqref{E:alpha}. 
Let   $u_0$ be an $H^1(\RR^n)$-valued, ${\mathcal F}_0$-measurable random variable, which satisfies 
$E\big( M(u_0)^{1/s_c}\big) <\infty$ 
and such that 
\begin{equation} 		\label{Hyp_HM_random}
H(u_0) M(u_0)^\alpha \leq \beta H(Q) M(Q)^\alpha\quad \mbox{\rm  a.s. for some } \; \beta\in (0,1) 
\end{equation}
and
\begin{equation}		\label{Hyp_gradM_random}
\| \nabla u_0\|_{L^2(\RR^n)} \|u_0\|_{L^2(\RR^n)}^\alpha < \|\nabla Q\|_{L^2(\RR^n)} \|Q\|_{L^2(\RR^n)}^\alpha \; \quad \mbox{\rm a.s.} .
\end{equation}

For $M_\phi$ as in \eqref{M_phi}, define 
\begin{equation}\label{T_max_gene_random}
T^\ast = 
\frac{\big[\sqrt{n+2 s_c(1 - \beta)/9\,} - \sqrt{n}\big]^2}{2(1-s_c)}
\frac{9}{M_\phi} \frac{M(Q)^{\frac1{s_c}}}{ \EE \big( M(u_0)^{\frac1{s_c}}\big)}.
\end{equation}
Then $P\big(\tau^*(u_0) >T\big) >0$ for $T<T^*$. 
\smallskip

Furthermore,  if $M(u_0)>0$ a.s., we have 
\begin{equation}\label{lower_E_tau_gene_random}
\EX \big(\tau^*(u_0) M(u_0)^{\frac1{s_c}} \big) \geq {T}^* \,  
{\EX \big(M(u_0)^{\frac1{s_c}} \big) } \equiv \frac{\big[\sqrt{n+2 s_c(1 - \beta)/9\,} - \sqrt{n}\big]^2}{2(1-s_c)} \frac{9}{M_\phi} {M(Q)^{\frac1{s_c}}}.
\end{equation}
\end{Th}
\begin{remark}
Similarly, the $L^2$ norm of the gradient remains uniformly bounded a.s. on the interval $[0, t]$ for any $\displaystyle t \leq \sup_{\delta <1} \{ \tau_\delta \} $,
 where $\tau_\delta$ is defined in \eqref{tau_delta} 
and its expected value can be upper estimated in terms of some negative moment of the mass of the initial  data (see \eqref{E:grad-bound2}). 
To be precise, this bound requires $M(u_0) $ to be ``not too small".
\end{remark} 

\noindent {\it Proof of Theorem \ref{th-wp-inter-multi-randomIC}.}
~ Using the same notation for $\tau_\delta$ due to the mass conservation \eqref{mass_conserv} and similar arguments as in the previous theorem, the upper estimate \eqref{upper_P} is rephrased as
 \begin{align*}		
P(\tau_\delta & \leq T)= P\Big( \sup_{s\leq T\wedge \tau_\delta} M(u_0)^\alpha H\big(u(s)\big)  \geq \delta M(Q)^\alpha H(Q)\Big)   \\
\leq & \frac{1}{\delta M(Q)^\alpha H(Q)} \Big[ \EX \big( M(u_0)^\alpha H(u_0) \big) + \frac{1}{2} T  M_\phi \EX \big( M(u_0)^{\alpha +1}\big)
+ 3 \sqrt{M_\phi} \sqrt{T} \, \|Q\|_{L^2}^\alpha \|\nabla Q\|_{L^2} \EX \big( M(u_0)^{ \frac{\alpha +1}{2}} \big)  \Big]  \\
\leq & \frac{\beta}{\delta} + \frac{1}{2} M_\phi T \frac{\EX \big( M(u_0)^{1+\alpha}\big)}{\delta M(Q)^\alpha H(Q)} + 3 \sqrt{M_\phi T} \,
\frac{ \big\{ \EX \big( M(u_0)^{\alpha+1}\big) \big\}^{\frac{1}{2}} \|Q\|_{L^2}^\alpha \|\nabla Q\|_{L^2}}{\delta M(Q)^\alpha  H(Q)},
\end{align*}
where the last line is deduced from the Cauchy-Schwarz inequality. 
 
The proof of \eqref{T_max_gene_random} is completed using similar to the deterministic initial data argument in Theorem \ref{th-wp-inter-multi}, replacing $M(u_0)^{\alpha +1}$ by
 $\EX \big( M(u_0)^{\alpha +1}\big)$ and $M(u_0)^{\frac{\alpha +1}{2}}$ by $\big\{ \EX \big( M(u_0)^{\alpha +1}\big) \big\}^{\frac{1}{2}}$, respectively.

The inequality \eqref{E:grad-bound1} still holds, which implies a bound of the expected value of the gradient of $u$, namely, 
\begin{equation}\label{E:grad-bound2}
\EE\Big( \sup_{s\in [0, \tau_\delta} \|\nabla u(s)\|^2_{L^2(\RR^n)} \Big) \leq  \frac{n}{2(1-s_c)} M(Q)^{1+\alpha} \EE(M(u_0)^{-\alpha}).
\end{equation}
 
When $u_0$ is random, in order to prove \eqref{lower_E_tau_gene_random}, we may similarly
suppose that
$M(u_0)^{\alpha+1} \tau_\delta$ is a.s. finite. Since $M(u_0)>0$ a.s., we have $\tau_\delta<\infty$ a.s..  
The inequality \eqref{delta_MH_Q} rewrites as
\begin{align*}
\delta M(Q)^\alpha H(Q) &\leq \beta M(Q)^\alpha H(Q) +  \frac12 { M_\phi}  \EX \big( M(u_0)^{\alpha+1} \tau_\delta\big) 
+3 \sqrt{M_\phi}  \|Q\|_{L^2}^\alpha \|\nabla Q\|_{L^2} \EX \big( M(u_0)^{\frac{\alpha+1}{2}}  \sqrt{\tau_\delta}\big)  \\
& \leq  \beta M(Q)^\alpha H(Q) +  \frac12 {M_\phi} \EX \big(  M(u_0)^{\alpha+1} \tau_\delta\big)
+3  \|Q\|_{L^2}^\alpha \|\nabla Q\|_{L^2} \big\{ M_\phi \, \EX \big( M(u_0)^{\alpha+1}  \tau_\delta\big)\big\}^{\frac{1}{2}},
\end{align*}
where the last estimate follows from the Cauchy-Schwarz inequality. 
 
Let $\tilde{Z}=\big\{ M_\phi\, \EX \big( M(u_0)^{\alpha +1} \tau_\delta\big) \big\}^{\frac{1}{2}}$. 
The above inequality can be rewritten as 
$\tilde{a}\tilde{Z}^2+\tilde{b}  \tilde{Z}-\tilde{c} \geq 0$ with $\tilde{a}=\frac{1}{2}$, $\tilde{b}=3 \|Q\|_{L^2}^\alpha \|\nabla Q\|_{L^2}$ and 
$\tilde{c} = (\delta - \beta) M(Q)^\alpha H(Q)$. Using the identities \eqref{E:normQ1}, \eqref{E:normQ2} and \eqref{H_Q}, 
$$
\tilde b= 3 \bigg( \frac{n }{{ 2(1-s_c)}} M(Q)^{\alpha +1} \bigg)^{\frac{1}{2}}  \quad \mbox{\rm and}\quad  
\tilde c=(\delta -\beta ) \frac{s_c}{2(1-s_c)}
M(Q)^{\alpha +1}.
$$

The discriminant of the quadratic polynomial $\tilde{a}\tilde{Z}^2+\tilde{b}  \tilde{Z}-\tilde{c} $ is 
$$
\tilde{D}(\delta) = 
\frac{9n +2 s_c(\delta - \beta)}{2(1-s_c)}
 M(Q)^{\alpha+1} >\tilde{b}^2>0,
$$
and the roots are
$\tilde{Z}_1=-\tilde{b}-\sqrt{\tilde{D}}<0$ and $\tilde{Z}_2=-\tilde{b}+\sqrt{\tilde{D}}>0$. 
Hence, the inequality 
$\tilde{a}\tilde{Z}^2+\tilde{b}  \tilde{Z}-\tilde{c} \geq 0$ for non-negative $\tilde{Z}$ is satisfied for $\tilde{Z}\geq \tilde{Z}_2$. 
As $\delta \to 1$, we obtain \eqref{lower_E_tau_gene_random}. 
\hfill $\Box$

\section{Maximal existence time - Additive noise} \label{S:A} 
In this section, we consider the SNLS \eqref{E:NLS} with an additive noise and obtain a maximal time existence interval (or global-type) solutions with the controlled quantities 
such as mass and gradient.

\subsection{Preliminaries on additive noise and local well-posedness.} \label{S:noise-add}

We suppose that $W$ is a complex-valued noise with 
\begin{equation}\label{E:add-noise}
W(t)= \sum_{k\geq 0} \phi e_k \beta_k(t),
\end{equation}
where 
$\{ e_k\}_{k\geq 0}$ is an orthonormal basis of $L^2(\RR^n)$, $\phi\in L_2^{0,1}$, that is,  a Hilbert-Schmidt operator 
from $L^2(\RR^n)$ to $H^1(\RR^n)$, and $\{ \beta_k\}_{k\geq 0}$ are independent one dimensional Brownian motions. 
Note that in this case the driving noise is complex-valued. With this additive noise $W$, we consider the stochastic NLS equation 
\begin{equation}\label{NLS_add}
\left\{
\begin{array}{l}
i d_t u(t) - \big( \Delta u(t) + |u(t)|^{2\sigma} u(t) \big) dt =  dW(t), \\
u(0)=u_0 \in H^1.
\end{array}
\right.
\end{equation}

We recall the following local well-posedness result together with the blow-up alternative from \cite{deB_Deb_H1}.
 (Note that in the additive noise case, we do not have a restriction $\sigma \geq \frac12$ in local well-posedness as we do in the multiplicative noise case, see the next theorem.) 

\begin{Th}{\cite[Thm 3.1]{deB_Deb_H1} } \label{lwp_add}
Consider $s_c<1$, or equivalently, $\sigma \geq 0$ for $n=1,2$ 
or $\sigma \in \big[ 0, \frac{2}{n-2} \big)$ for $n\geq 3$, 
and let $\phi \in L^{0,1}_2$. 
Then for any ${\mathcal F}_0$-measurable  initial condition $u_0 \in H^1$ a.s., 
there exists a stopping time $\tau^*(u_0) >0$ a.s. such that 
there exists a unique solution to \eqref{NLS_add} on the random time
interval $[0, \tau^*(u_0))$, which belongs  to $C([0,\tau^*(u_0));H^1)$ a.s.. Furthermore, either  
$$
\tau^*(u_0)(\omega)=+\infty \quad \mbox{or} \quad \lim_{t\to \tau^*(u_0)(\omega)} \|u(t)(\omega) \|_{H^1}=+\infty.
$$
\end{Th}

\subsection{Estimates for mass and energy}\label{general_add}

Unlike the multiplicative stochastic perturbation, the mass in the additive case is no longer preserved a.s. in time. 
We recall the  time dependence of mass 
 from  \cite[Prop. 3.2]{deB_Deb_H1}. In particular, for any positive stopping time $\tau < \tau^*(u_0)$ a.s. we have
\begin{equation}		\label{M_add}
M\big( u(\tau)\big) = M(u_0) + \tau \|\phi\|_{L_2^{0,0}}^2 - 2 \, \mbox{\rm Im } \Big( \sum_{k\geq 0} \int_0^\tau \!\!\int_{\RR^n} u(s,x)
\overline{\phi e_k(x)} dx d\beta_k(s) \Big).
\end{equation}
 
We next state the time dependence of energy.
Recalling the energy $H(u(\tau))$ from \cite[Proposition 3.3]{deB_Deb_H1} and the definition \eqref{E:add-noise} of the infinite-dimensional driving noise $W$, we bound the energy as follows: 
for any stopping time $\tau < \tau^*(u_0)$ a.s.  
\begin{align}\label{H_add}
H\big(& u(\tau)\big) = H(u_0) + \tfrac{1}{2} \tau { \sum_{k\geq 0} \| \nabla \phi e_k\|_{L^2(\RR^n)}^2 }
- \mbox{\rm Im } \Big( \int_0^\tau \!\!\int_{\RR^n} \big[  
| u(s,x)|^{2\sigma} \overline{u(s,x)}  + \Delta \overline{u(s,x)}\big] dx dW(s) \Big) 
\\
& \quad -\frac{1}{2} \sum_{k\geq 0} \int_0^\tau \!\!\int_{\RR^n} \Big[ | u(s,x)|^{2\sigma} |\phi e_k(x)|^2 + 2\sigma |u(s,x)|^{2\sigma -2} 
\big( \mbox{\rm Re }\big(\overline{u(s,x)}
\phi e_k(x) \big)\big)^2 
\Big] 
dx ds \label{E:drop-term1} \\
&\leq H(u_0) + \frac{1}{2} \tau \|\phi\|_{L_2^{0,1}}^2 + \mbox{\rm Im } \Big(\sum_{k\geq 0} \int_0^\tau \!\! \int_{\RR^n} 
\nabla \overline{u(s,x)} \nabla (\phi e_k)(x) dx d\beta_k(s) \Big) \nonumber \\
&\quad  - \mbox{\rm Im } \Big( \sum_{k\geq 0}  \int_0^\tau \!\! \int_{\RR^n}  
| u(s,x)|^{2\sigma} \overline{u(s,x)} (\phi e_k)(x) dx d\beta_k(s) \Big), \label{H_add_final}
\end{align}
where to get the last inequality we bounded the $\dot{H}^1$ of the second term from \eqref{H_add} by the $H^1$ norm,  then
we used integration by parts in the third term from \eqref{H_add} and dropped the negative term in \eqref{E:drop-term1}. 

For convenience, and to control the size of additive noise, we denote by $C(\phi)$ the following constant:
\begin{equation}\label{E:Const-phi}
C(\phi) = \|\phi\|_{L^{0,1}_2}^{\frac{n\sigma}{2\sigma+2}} \; \| \phi\|_{L^{0,0}_2}^{\frac{2-(n-2)\sigma}{2\sigma+2}}.
\end{equation}

To proceed, we use the localization to obtain some control on the time evolution of the mass. 
The following lemma describes bounds on the expected value of the mass and localized energy uniformly on the time interval $[0,\tau]$ for any stopping time {$\tau < \tau^\ast(u_0)$ a.s.}

\begin{lemma}\label{lem_MH}
Let the assumptions from Theorem \ref{lwp_add} be satisfied and $\tau^*(u_0)$ be a stopping time as defined there. 

(i) Suppose that $\EX \big(M(u_0)\big)<\infty$. Then for any stopping time $\tau < \tau^*(u_0)$ a.s. and every $\epsilon \in (0,1)$, we have 
\begin{equation}\label{E_sup_M}
\EX \Big( \sup_{s\leq \tau} M\big(u(s)\big) \Big) \leq \frac{1}{1-\epsilon} \Big[ \EX \big( M(u_0)\big) 
+ \frac{9+\epsilon}{\epsilon}  \|\phi\|_{L^{0,0}_2}^2  \EX(\tau)\Big].
\end{equation}

(ii) Let $\{ \Omega(t)\}_{t\geq 0}$ be a decreasing family of $\{ {\mathcal F}_t\}_{t\geq 0}$-measurable sets. Then
\begin{align}\label{E_sup_H}
\EX \Big( \sup_{s\leq \tau} 1_{\Omega(s)} H\big(u(s)\big)\Big) \leq &\; \EX \big(H(u_0) \big)
+ \frac{1}{2} \|\phi\|_{L^{0,1}_2}^2 \EX (\tau) + 3 \|\phi\|_{L^{0,1}_2} \EX \Big(
\sqrt{\tau}  \sup_{s\leq \tau} \big\{ 1_{\Omega(s)} 
\|\nabla u(s)\|_{L^2} \big\} \Big) \nonumber \\
&\; + 3 (C_{GN})^{\frac1{2\sigma+2}}
C(\phi)  \EX \Big( \sqrt{\tau} \sup_{s\leq \tau} \big\{ 1_{\Omega(s)}
 \|u(s)\|_{L^{2\sigma +2}}^{2\sigma +1}\big\} \Big),
\end{align}
where $C(\phi)$ is as defined in \eqref{E:Const-phi}
and $C_{GN}$ as in \eqref{constant_GN}. 
\end{lemma}
\begin{proof}
(i) 
Let $\tau<\tau^*(u_0)$ be a stopping time. For $R>0$ set
\begin{equation}\label{E:explR} 
\tau_R=\inf\{ t\geq 0 \; : \; \|u(t)\|_{H^1} \geq R\} \wedge \tau.
\end{equation}
The identity \eqref{M_add}, the Davis and Cauchy-Schwarz inequalities imply 
\begin{align*}
\EX \Big( \sup_{s\leq \tau_R} M\big(u(s)\big) \Big) \leq &\;  \EX \big( M(u_0)\big)  + \|\phi\|_{L^{0,0}_2}^2 \EX (\tau_R)
+ 6 \,\EX  \Big[ \Big\{ \int_0^{\tau_R}  
\| u(s)\|_{L^2}^2  \Big( \sum_{k\geq 0} \| \phi e_k\|_{L^2}^2\Big) ds \Big\}^{\frac{1}{2}} \Big] \\
\leq & \; \EX \big(  M(u_0) \big) + \|\phi\|_{L^{0,0}_2}^2 \EX(\tau_R) + 6 \, \|\phi\|_{L^{0,0}_2} \EX \Big( \sqrt{\tau_R} \; \sqrt{\sup_{s\leq \tau_R} 
M\big( u(s)\big) } \Big) \\
\leq & \; \EX \big( M(u_0)\big)  +  \epsilon \, \EX \Big( \sup_{s\leq \tau_R} M\big(u(s)\big) \Big)  + \|\phi\|_{L^{0,0}_2}^2 \EX (\tau_R) 
\Big( 1+\frac{9}{\epsilon}\Big),
\end{align*}
where the last  estimate is a consequence of the Young inequality for any $\epsilon >0$.  Taking $\epsilon \in (0,1)$ and using the inequality 
$\EX \big[ \sup_{s\leq \tau_R} M\big(u(s)\big)\big] \leq R^2<\infty$ from the definition \eqref{E:explR} of $\tau_R$, we deduce \eqref{E_sup_M} with the supremum of the mass on the time interval $[0,\tau_R]$ instead of $[0,\tau]$. As $R\to +\infty$, $\tau_R\to \tau$ a.s., and the monotone convergence theorem concludes the proof of \eqref{E_sup_M}.
\smallskip

(ii) Using the identity \eqref{H_add}, the Davis  inequality,  the local property of stochastic integrals  and the Cauchy-Schwarz and H\"older's 
inequalities,  we obtain
\begin{align*}
\EX \Big( \sup_{s\leq \tau} &1_{\Omega(s)} H\big(u(s)\big)\Big) \leq \;\EX \big( H(u_0) \big) + \frac{1}{2}  \|\phi\|_{L^{0,1}_2}^2 \EX (\tau)\\
& \; + 3\EX \Big( \Big\{ \int_0^\tau \! 1_{\Omega(s)} \sum_{k\geq 0} \Big( \int_{\RR^n}\! \! |\nabla u(s,x)| 
 |\nabla (\phi e_k)(x)| dx \Big)^2 ds \Big\}^{\frac{1}{2}} \Big)\\
 &\;  + 3 \EX  \Big( \Big\{\int_0^\tau1_{\Omega(s)}  \sum_{k\geq 0} \Big( \int_{\RR^n} |u(s,x)|^{2\sigma +1} |\phi e_k(x)| \, dx\Big)^2
 \Big\}^{\frac{1}{2}} \Big) \\
 \leq & \; \EX \big( H(u_0) \big) + \frac{1}{2}  \|\phi\|_{L^{0,1}_2}^2 \EX (\tau)
 +3 \EX \Big(  \Big\{ \int_0^\tau \! 1_{\Omega(s)} \sum_{k\geq 0} \|\nabla u(s)\|_{L^2}^2 \|\nabla \phi e_k\|_{L^2}^2 ds \Big\}^{\frac{1}{2}} \Big)     \\
&\;   +3 \EX \Big( \Big\{ \int_0^\tau \! 1_{\Omega(s)} \sum_{k\geq 0}  \Big( \int_{\RR^n} 
|u(s,x)|^{2\sigma +2} dx \Big)^{\frac{2(2\sigma +1)}{2\sigma +2}}
 \Big( \int_{\RR^n} |\phi e_k(x)|^{2\sigma +2} dx \Big)^{\frac{2}{2\sigma +2}} ds \Big\}^{\frac{1}{2}} \Big) .
\end{align*}
Since $\sigma$ is subcritical (i.e., $\sigma < \frac{2}{(n-2)}$ in dimensions 3 and higher, and 
$\sigma$ is any positive number in dimensions 1 and 2), 
the Gagliardo-Nirenberg inequality \eqref{GN} and H\"older's inequality  imply
\begin{align*}
 \sum_{k\geq 0} \|\phi e_k\|_{L^{2\sigma +2}}^2 \leq &( C_{GN})^{\frac{1}{\sigma +1}} \sum_{k\geq 0}
 \|\nabla \phi e_k\|_{L^2}^{\frac{n\sigma}{\sigma +1}} \|\phi e_k\|_{L^2}^{\frac{2- (n-2)\sigma}{\sigma +1}} \\
 \leq & ( C_{GN})^{\frac{1}{\sigma +1}}  \Big( \sum_{k\geq 0} \| \nabla \phi e_k\|_{L^2}^2\Big) ^{ \frac{n\sigma}{2(\sigma +1)}} 
 \Big( \sum_{k\geq 0} \| \phi e_k\|_{L^2}^2\Big) ^{\frac{2- (n-2)\sigma}{2(\sigma +1)}} 
 \\
 \leq & ( C_{GN})^{\frac{1}{\sigma +1}}   \|\phi\|_{L^{0,1}_2}^{\frac{n\sigma}{\sigma +1}}  \| \phi\|_{L^{0,0}_2}^{\frac{2- (n-2)\sigma}{\sigma +1}} \equiv ( C_{GN})^{\frac{1}{\sigma +1}} \, C(\phi)^{\frac12}.
\end{align*}
Therefore, we deduce
\begin{align*}
\EX \Big( \sup_{s\leq \tau} 1_{\Omega(s)} H\big(u(s)\big)\Big) &\leq \;\EX \big( H(u_0) \big) + \frac{1}{2}  \|\phi\|_{L^{0,1}_2}^2 \EX (\tau)
+3 \|\phi\|_{L^{0,1}_2} \EX \Big( \sqrt{\tau} \sup_{s \leq \tau}\Big\{ 1_{\Omega(s)}  \|\nabla u(s)\|_{L^2}\Big\} \Big)  \\
 &  + 3 (C_{GN})^{\frac1{2\sigma+2}} 
C(\phi) \EX \Big( \sqrt{\tau} \sup_{s\leq \tau} \Big\{1_{\Omega(s)} 
  \|u(s)\|_{L^{2\sigma +2}}^{2\sigma +1}\Big\} \Big).
\end{align*}
thus, proving \eqref{E_sup_H} and finishing the proof. 
\end{proof}

 
\subsection{$L^2$-critical case}\label{critical_add}
Consider $s_c=0$, or equivalently, $\sigma = \frac{2}{n}$. In this case as observed in Section \ref{S:Q} in \eqref{E:sc=0}, for $u\in H^1$ we have
\begin{equation}\label{E:grad-bound}
\| \nabla u\|_{L^2}  \Big(1 -  \Big( \tfrac{\|u\|_{L^2}}{\|Q\|_{L^2}}\Big)^{\frac{4}{n}} \Big) \leq 2 H(u). 
\end{equation}

Given $u_0 \in {H}^1$ and the time $\tau^*(u_0)$ as in Theorem \ref{lwp_add}, we take $\beta \in (0,1)$ such that 
\begin{equation}\label{Hyp_add_mass}
\left\{
\begin{array}{l} 
\|u_0\|_{L^2} = \beta \|Q\|_{L^2},\; \mbox{\rm   if } \; u_0\;  \mbox{\rm is deterministic,}\\
\|u_0\|_{L^2} \leq \beta \|Q\|_{L^2} \; \mbox{\rm a.s., if } u_0 \; \mbox{\rm is random.}
\end{array}
\right.
\end{equation}
Let $\delta \in (\beta,1)$. Then for any $R>0$ we define
\begin{align}\label{E:tau-delta1}
&\tau_\delta = \inf\{ t\geq 0 : \|u(t)\|_{L^2} \geq \delta \|Q\|_{L^2} \} \wedge \tau^*(u_0) \; \mbox{\rm and } \\
&\tau_{\delta,R} =\tau_\delta \wedge  \inf\{ t\geq 0 : \|\nabla u(t)\|_{L^2} \geq R \}.\label{E:tau-delta-R}
\end{align}
The following result proves the well-posedness on a set where the mass of the solution remains controlled by the mass of the ground state. 

\begin{lemma}\label{wp_tau_delta}
Let $\sigma = \frac{2}{n}$,  $\phi\in L_2^{0,1}$, 
$u_0$ be ${\mathcal F}_0$-measurable, belong to  $H^1$ a.s. and  satisfy the condition \eqref{Hyp_add_mass} for some $\beta \in (0,1)$. Take $\delta \in (\beta,1)$ and define $\tau_\delta$ as in \eqref{E:tau-delta1}.
Then the problem \eqref{NLS_add} is well-posed on the random time interval $[0,\tau_\delta]$ for $\tau_\delta \leq  \tau^*(u_0)$ a.s., and hence, the mass if the solution remains bounded by the ground state mass a.s. on this time interval. 
\end{lemma}

\begin{proof}
Using \eqref{E:grad-bound} and recalling $\tau_\delta$ from \eqref{E:tau-delta1}, one easily obtains
\begin{equation} 	\label{upp_grad_tau}
\| \nabla u(s)\|_{L^2}^2 \leq \frac{2}{1-\delta^{\frac{4}{n}}} H\big(u(s)\big)\quad \mbox{\rm for ~any~} s\leq \tau_{\delta}.
\end{equation}
Recall $C(\phi)$ from \eqref{E:Const-phi}, which in this case, can be written as $C(\phi)= \| \phi\|_{L^{0,1}_2}^{\frac{n}{n+2}} \; \| \phi \|_{L^{0,0}_2}^{\frac{2}{n+2}}$. 
Using \eqref{E_sup_H} with $\Omega(s)=\Omega$ (the whole probability space) for every $t\geq 0$ together with \eqref{GN} and \eqref{constant_GN}, we deduce 
\begin{align*}
\EX \Big( \sup_{s\leq \tau_{\delta,R}}& H\big( u(s)\big)\Big)  \leq \; \EX \big( H(u_0)\big)  + \frac{1}{2} \| \phi\|_{L^{0,1}_2}^2 \EX ( \tau_{\delta,R}) \\
&\; + 3 \,\| \phi\|_{L^{0,1}_2} 
\sqrt{\frac{2}{1-\delta^{\frac{4}{n}}} } \, \EX \Big( \sqrt{\tau_{\delta,R}}  \sup_{s\leq \tau_{\delta,R}}
 \sqrt{H\big( u(s)\big)} \Big) \\
& \; + 3\, (C_{GN})^{\frac{n}{4+2n}}
C(\phi) 
\EX \Big( \sqrt{\tau_{\delta,R}}  \sup_{s\leq \tau_{\delta,R}} \Big\{ 
(C_{GN})^{\frac{4+n}{4+2n}}
 \| \nabla u(s)\|_{L^2}^{2  \frac{4+n}{4+2n}} \| u(s)\|_{L^2}^{\frac{4}{n}  \frac{4+n}{4+2n}} \Big\}  \Big)\\
 \leq & \; \EX \big( H(u_0) \big) + \frac{1}{2} \| \phi\|_{L^{0,1}_2}^2 \EX (\tau_{\delta,R})  + 3 \| \phi\|_{L^{0,1}_2} \frac{\sqrt{2}}{({1-\delta^{\frac{4}{n}}})^{\frac12} } 
 \big\{ \EX (\tau_{\delta,R})\big\}^{\frac{1}{2}} \Big\{\EX\Big( \sup_{s\leq \tau_{\delta,R}} H\big( u(s) \big)\Big)\Big\}^{\frac{1}{2}}  \\
 &\;  + 3 \, C_{GN} \,
 C(\phi) \big( \delta \|Q\|_{L^2}\big)^{\frac{4}{n} \frac{4+n}{4+2n}} 
 \Big( \frac{2}{1-\delta^{\frac{4}{n}}}\Big)^{\frac{4+n}{4+2n}} \EX \Big[ \sqrt{\tau_{\delta,R}} \Big( \sup_{s\leq \tau_{\delta,R}} 
 H\big( u(s)\big) \Big)^{\frac{4+n}{4+2n}} \Big],
 \end{align*}
where in the last upper estimate we have used the Cauchy-Schwarz inequality.
Let $\epsilon_1, \epsilon_2 >0$ and use  the Cauchy-Schwarz, H\"older  inequalities, and the Young inequality with conjugate exponents $(2,2)$ and $(\frac{4+2n}{n}, \frac{4+2n}{4+n} )$, respectively, to obtain 
 \begin{align*} 
 \EX \Big( \sup_{s\leq \tau_{\delta,R}} H\big( u(s)\big)\Big) \leq 
 \; & \EX \big( H(u_0)\big) + (\epsilon_1 + \epsilon_2)  \EX \Big( \sup_{s\leq \tau_{\delta,R}} 
 H\big( u(s)\big) \Big) 
 + \Big( \frac{1}{2} + \frac{9}{2\epsilon_1\big(1-\delta^{\frac{4}{n}}\big)} \Big)
 \| \phi\|_{L^{0,1}_2}^2 \EX ( \tau_{\delta,R}) \\
 &\;   + \frac{n}{4+2n}  \frac{1}{\epsilon_2^{\frac{4+n}{n}}} 
\Big( \frac{4+n}{4+2n} 3K C(\phi) \Big)^{\frac{4+2n}{n}} \frac{ \delta^{\frac{4(4+n)}{n^2}}}{\|Q\|_{L^2}^{\frac{4}{n}}}  \; 
 \Big( \frac{2}{1-\delta^{\frac{4}{n}}} \Big)^{\frac{4+n}{n}} \EX \Big(\tau_{\delta,R}^{\frac{2+n}{n}}\Big).  
\end{align*}

Since from the energy definition $H(u)\leq \frac{1}{2} \|\nabla u\|_{L^2}^2$ and recalling \eqref{E:tau-delta-R}, we obtain 
$\EX \big[ \sup_{s\leq \tau_{\delta,R}} H\big( u(s)\big) \big] \leq \frac{R^2}{2}<\infty$
for any $R>0$. Choosing  $\epsilon_1, \epsilon_2$ such that $\epsilon_1 + \epsilon_2 <1$,
since $\tau_{\delta,R}\leq \tau_\delta$, we deduce that for any $R>0$,
\begin{align*}
 \EX \Big( \sup_{s\leq \tau_{\delta,R}}& H\big( u(s)\big)\Big)  \leq   \frac{1}{1-(\epsilon_2+\epsilon_2)}\Big[
 \EX \big( H(u_0)\big) + \Big( \frac{1}{2} + \frac{9}{2\epsilon_1\big(1-\epsilon^{\frac{4}{n}}\big)} \Big) \| \phi\|_{L^{0,1}_2}^2 \EX ( \tau_\delta) \\
 &\;  + \frac{n}{4+2n}  \frac{1}{\epsilon_2^{\frac{4+n}{n}}} 
  \Big( \frac{4+n}{4+2n} 3K C(\phi) \Big)^{\frac{4+2n}{n}} \frac{ \delta^{\frac{4(4+n)}{n^2}}}{\|Q\|_{L^2}^{\frac{4}{n}}}  \; 
 \Big( \frac{2}{1-\delta^{\frac{4}{n}}} \Big)^{\frac{4+n}{n}} \EX \big(\tau_\delta^{\frac{2+n}{n}}\big)\Big].  
\end{align*} 
As $R\to +\infty$, the monotone convergence theorem implies that 
$\EX \big[ \sup_{s\leq \tau_\delta} H\big( u(s)\big) \big] $ is bounded by the right-hand side of the above inequality. 

Replacing the stopping time $\tau_\delta$ by $\tau_\delta \wedge T$ for a fixed $T>0$,
we deduce  $\EX  \big[ \sup_{s\leq \tau_\delta \wedge T} H(u(s)) \big] <\infty$, 
and using   \eqref{upp_grad_tau}, we obtain
$\EX \big( \sup_{s\leq \tau_\delta \wedge T} \| \nabla u(s)\|_{L^2}^2\big) <\infty$. Therefore, we have well-posedness on the (random) time interval
$[0, \tau_\delta \wedge T]$. Since $T$ can be arbitrary, we have well-posedness on the time interval $[0,\tau_\delta]$,
 where the mass of the solution $u(t)$ remains below the mass of the ground state a.s. for $t \leq \tau_{\delta}$, where $\delta$ depends on the initial mass of $u_0$ in \eqref{Hyp_add_mass}. 
\end{proof}

The following theorem is the main result of this section. As in Section \ref{mass_critical_Stra}, it is only necessary to control the mass of the 
initial condition. 

\begin{Th}\label{Th_mass_crit_add}
Let $\phi$ satisfy the assumptions of Theorem \ref{lwp_add}, 
$\sigma = \frac{2}{n}$ (i.e., $s_c=0$),
$u_0\in H^1$ a.s. be ${\mathcal F}_0$-measurable such that 
$\|u_0\|_{L^2} \leq  \beta \|Q\|_{L^2}$ a.s. 
for some $\beta \in (0,1)$. Set 
\begin{equation} 		\label{T*_add_mass_crit}
\epsilon^*=3\sqrt{10 - \beta^2}-9\quad  \mbox{\rm and } \; 
 {T}^*= 
\frac{\epsilon^* [ 1-\beta^2-\epsilon^*]}{9+\epsilon^*} \; \frac{M(Q)}{\|\phi\|_{L^{0,0}_2}^2}.
\end{equation}
Let $\tau^*(u_0)$ denote the stopping time defined in Theorem  \ref{lwp_add}. Then 
\begin{equation} \label{crit_add_bound}
 \EX \big(\tau^*(u_0)\big) \geq  {T}^*  \; 
\end{equation}
and 
 with a strictly positive probability the solution exists 
up to any time $T>0$, provided that $T$ is strictly less then $T^\ast$, i.e., 
\begin{equation}
 P\big(T<\tau^*(u_0) \big) >0 ~~ \mbox{\rm  if }\;  T< {T}^*.
\end{equation}
 
\begin{remark}
Note that $T^\ast$ is a decreasing function of both $\beta$ and 
of the strength of the noise $\| \phi\|_{L^{0,0}_2}$. 
\end{remark}
\end{Th}

\begin{proof}
Fix $\delta \in (\beta,1)$. We first prove a lower estimate of $\EX (\tau_\delta)$.
 We may suppose that $\tau_\delta < +\infty$ a.s. and that
$\EX (\tau_\delta)<+\infty$, since otherwise there is nothing to prove. 
Then $\sup_{s\leq \tau_\delta}  M( u(s)) = \delta^2  M(Q) $. Since $\|u_0\|_{L^2} \leq \beta \|Q\|_{L^2}$ a.s., we have
$E\big( M(u_0)\big) \leq \beta^2 M(Q)$. 
Hence, \eqref{E_sup_M} implies for $\epsilon \in (0,1)$, 
$$
\delta^2 M(Q) \leq \frac{1}{1-\epsilon} \Big[ \beta^2 M(Q) + \frac{9+\epsilon}{\epsilon} \|\phi\|_{L^{0,0}_2}^2 E(\tau_\delta)\Big].
$$
Choose $\epsilon$ such that $\frac{\beta^2}{1-\epsilon}<\delta^2$, that is, $\epsilon < 1-\frac{\beta^2}{\delta^2}$. Then
\[ 
\EX (\tau_\delta) \geq \frac{ \epsilon\, [\delta^2 (1-\epsilon) - \beta^2]}{9+\epsilon} \; \frac{M(Q)}{\|\phi\|_{L^{0,0}_2}^2}.
\] 
We optimize the above 
estimate by choosing $\epsilon \in \big( 0, 1-\frac{\beta^2}{\delta^2}\big)$. 
Let $\varphi_\delta(\epsilon) = \frac{\epsilon [ \delta^2(1-\epsilon) - \beta^2]}{9+\epsilon}$. 
Then $\varphi_\delta'(\epsilon)=\frac{\psi_\delta(\epsilon)}{(9+\epsilon)^2}$, where 
$\psi_\delta(\epsilon)=-\delta^2 \epsilon^2 - 18 \delta^2 \epsilon
+9(\delta^2-\beta^2)$. The discriminant of this polynomial is $4 \tilde{\Delta}$, where $\tilde{\Delta} = (9\delta^2)^2 +
9\delta^2(\delta^2-\beta^2)\geq (9 \delta^2)^2>0$. 
The equation $\psi_\delta(\epsilon)=0$ has two solutions, 
\[ \epsilon_\delta(1)= - \tfrac{9\delta^2 + \sqrt{\tilde{\Delta}}}{\delta^2}<0 \quad \mbox{\rm and} \quad
\epsilon_\delta(2)= \tfrac{\sqrt{\tilde{\Delta}}-9\delta^2}{\delta^2}>0.
\]
The function $\varphi_\delta$ is increasing on the interval $[0, \epsilon_\delta(2)]$ and achieves its maximum at $\epsilon_\delta(2)$. 
Furthermore, 
$$ 
\epsilon_\delta(2)= \sqrt{9^2+9\big( 1-\tfrac{\beta^2}{\delta^2} \big) } -9 = \big( 1-\tfrac{\beta^2}{\delta^2}\big) 
\frac{1}{1+\sqrt{ 1+ \frac{1}{9}
\big( 1-\tfrac{\beta^2}{\delta^2}\big)}} < \frac{1}{2} \Big(1-\tfrac{\beta^2}{\delta^2}\Big).
$$
Therefore,
$$
\EX (\tau_\delta) \geq \frac{\epsilon_\delta(2) \big[ \delta^2 (1-\epsilon_\delta(2)) - \beta^2\big]}{9+\epsilon_\delta(2)} \; 
\frac{M(Q)}{\|\phi\|_{L^{0,0}_2}^2}.
$$
As $\delta \to 1$, $\epsilon_\delta(2)\to \epsilon^* = 3\sqrt{10 - \beta^2}-9 < \frac{1-\beta^2}{2}$,  and since $\EX (\tau^*(u_0)) \geq \EX (\tau_\delta)$
for every $\delta \in (0,1)$, we obtain
$$
\EX (\tau^*(u_0) ) \geq \limsup_{\delta \to 1} \EX (\tau_\delta) \geq {T}^* := \frac{\epsilon^* \big[ 1-\beta^2 - \epsilon^*\big]}{9+\epsilon^*}\; 
\frac{M(Q)}{\|\phi\|_{L^{0,0}_2}^2}.
$$
This lower bound is an increasing function of $\beta$
and a decreasing function of the strength $\| \phi\|_{L^{0,0}_2}$ of the stochastic
perturbation. 

We next prove that $P\big(T<\tau^*(u_0) \big) >0$ for  $T<{T}^*$.  Given $T>0$ and $\epsilon \in (0,1)$ we have for $\delta \in (\beta,1)$, 
\begin{align*}
P(\tau_\delta \leq T) =& P\Big( \sup_{s\leq T \wedge \tau^*(u_0)} M(u(s)) \geq \delta^2 M(Q)\Big) \\
\leq & \frac{1}{\delta^2 M(Q)} \EX \Big( \sup_{s\leq T \wedge \tau^*(u_0)} M(u(s)) \Big)\\
\leq & \frac{1}{\delta^2 M(Q)} \frac{1}{1-\epsilon} \Big[ \beta^2 M(Q) + \frac{9+\epsilon}{\epsilon} \|\phi\|_{L^{0,0}_2}^2 T\Big],
\end{align*}
where we have used \eqref{E_sup_M}  in the last inequality. 
To have well-posedness on the time interval $[0,T]$ with positive probability, we require  $P(\tau_\delta \leq T)<1$, that is,
\[ 
\frac{\beta^2}{(1-\epsilon)\delta^2} + \frac{9+\epsilon}{\epsilon (1-\epsilon)}\;  \frac{T\, \|\phi\|_{L^{0,0}_2}^2 }{\delta^2 M(Q)}<1.\]
Let $\epsilon < 1-\frac{\beta^2}{\delta^2}$. Then the above  inequality is satisfied if
\[ T<{T}^*_{\epsilon,\delta}:=  \epsilon \frac{\delta^2 (1-\epsilon)-\beta^2}{9+\epsilon} \; \frac{M(Q)}{\|\phi\|_{L^{0,0}_2}^2}.\]
Using the above computations, the optimal value of this upper bound is obtained when $\epsilon= \epsilon_\delta(2)$. As $\delta \to 1$, we deduce that
\[ P\big(T<\tau^*(u_0)\big)>0 \quad \mbox{\rm if } \; T<{T}^*=\frac{\epsilon^* [ 1-\beta - \epsilon^*]}{9+\epsilon^*} \; \frac{M(Q)}{\|\phi\|_{L^{0,0}_2}^2} 
\quad \mbox{\rm and} \; \epsilon^*=3 \sqrt{10 - \beta^2} -9,\] 
which completes the proof of this theorem. 
\end{proof}

\subsection{Intercritical case}	\label{inter_add}

In this section we consider the intercritical range $0<s_c<1$, i.e.,  
$\sigma > \frac{2}{n}$ for $n=1,2$, and $\sigma \in \big(\frac{2}{n}, \frac{2}{n-2})$ if $n\geq 3$. To ease notation, we denote this range for all dimensions as $\big(\frac2{n}, \frac{2}{(n-2)^+}\big)$.

Suppose that $\|u_0\|_{L^2}\leq \beta \|Q\|_{L^2}$ a.s. for some $\beta \in (0,1)$ and let $\delta \in (\beta,1)$. 
Note that $\EX \big( M(u_0)\big) \leq \beta^2 M(Q)$. 
Unlike the multiplicative case in Section \ref{inter_Stra}, the mass is not conserved, therefore, we first localize the mass $M(u)$ in terms of $M(Q)$. 
For $t>0$ set
\begin{equation}\label{Omega_t}
\Omega_t = \Big\{ \omega : \; t<\tau^*(u_0)(\omega), \; \sup_{s\leq t} \|u(s)\|_{L^2} \leq \delta \|Q\|_{L^2}\Big\} .
\end{equation}
Then \eqref{E_sup_M}, together with the Markov inequality, imply that for $\epsilon >0$,
\begin{align}\label{E:star}
P(\Omega_t^c) \leq &\; \frac{1}{\delta^2 M(Q)} \; \EX \Big( \sup_{s\leq t\wedge \tau^*(u_0)} \|u(s)\|_{L^2}^2\Big) 
\leq \frac{1}{\delta^2 M(Q)} \; \frac{1}{1-\epsilon}\; \Big[ \beta^2 M(Q) + \frac{9+\epsilon}{\epsilon} \, \| \phi\|_{L^{0,0}_2}^2\, t\Big].
\end{align}
If $\epsilon < 1 - \frac{\beta^2}{\delta^2}$, we obtain  $P(\Omega_t^c)<1$ (or in other words, $P(\Omega_t)>0$) for
$t\leq \frac{\epsilon [\delta^2 (1-\epsilon) - \beta^2]}{9+\epsilon}
\; \frac{M(Q)}{\| \phi\|_{L^{0,0}_2}^2}$. 

We also control the energy of the initial condition by  the energy  of the ground state. Suppose that 
\begin{equation}\label{Hyp_add_H}
H(u_0) \leq \gamma H(Q)\quad \mbox{\rm  a.s. for some } \; \gamma \in [0,1),
\end{equation}
and let $A\in (\gamma,1)$. Set
\begin{equation}\label{tauA}
\tau_A=\inf \{ t\geq 0\, : \; H(u(t)) \geq A H(Q) \} \wedge \tau^*(u_0),
\end{equation}
where $\tau^*(u_0)$, the stopping time, defined in Theorem  \ref{lwp_add}.

The following result proves well-posedness on a set, where mass, energy and kinetic  energy of the initial condition  are controlled by the
 corresponding quantities of the ground state. 
\begin{lemma}\label{wp_add_tauA}
Let $u_0$, $\phi$ and $\sigma \in ( \frac{2}{n}, \frac{2}{(n-2)^+})$ (i.e., $0<s_c<1$) satisfy the assumptions of Theorem \ref{lwp_add}. 
Suppose furthermore that $u_0$ satisfies the conditions \eqref{Hyp_add_mass}, \eqref{Hyp_add_H} with positive constants $\beta$ and $\gamma$ such that 
$\beta,\gamma \in (0,1)$ and $\|\nabla u_0\|_{L^2}  < \|\nabla Q\|_{L^2}$ a.s.. 
Let $T > 0$. Then for $\delta \in (\beta,1)$ such that $A\in (\gamma,1)$, the equation \eqref{NLS_add} is well-posed 
on the random time interval $[0, \tau_A\wedge T]$ on the set $\Omega_T$, where $\tau_A$ is defined in \eqref{tauA} and 
$\Omega_T$ is defined in \eqref{Omega_t},   with the uniform bound given in \eqref{E:grad-bound3} on the set $\Omega_T$.
\end{lemma} 
\begin{proof}
Choose $\tau < \tau^*(u_0)$ and $\tau \leq T$ to define (in a spirit of \eqref{E:defX})
\begin{equation*}
X(\tau) = \big( \delta \|Q\|_{L^2}\big)^\alpha \sup_{s\leq \tau} \big( 1_{\Omega_s} \| \nabla u(s)\|_{L^2} \big)
\end{equation*}
for any $\delta \in (\beta,1)$.
Since $\tau < \tau^*(u_0)$, we know that $X(\tau)<\infty$ a.s.. 
Similar to the derivation of \eqref{upper_gradient}-\eqref{E:X}, 
we obtain
\begin{align*}
X(\tau)^2 
\leq & \; 2  (\delta \|Q\|_{L^2})^{2\alpha} \sup_{s\leq \tau} \big[ 1_{\Omega_s} H(u(s))\big]  + 
B\, \big(\delta \|Q\|_{L^2}^{2-(n-2)\sigma +2\alpha}  \big) 
\sup_{s\leq \tau} \big( 1_{\Omega_s} \| \nabla u(s)\|_{L^2}^{n\sigma} \big) \\
\leq & \; 2  (\delta \|Q\|_{L^2})^{2\alpha} \sup_{s\leq \tau} \big[ 1_{\Omega_s} H(u(s))\big] 
+ B \, X(\tau)^{n\sigma}.
\end{align*}
Rewriting, the above inequality becomes (the analog of \eqref{upper_gradient})
$$
X(\tau)^2 - B X(\tau)^{n\sigma} \leq 2\, (\delta \|Q\|_{L^2})^{2\alpha} \sup_{s\leq \tau} \big[ 1_{\Omega_s} H(u(s))\big].
$$

Let $x^*:= \|\nabla Q\|_{L^2} \|Q\|_{L^2}^\alpha$ and  suppose that $X(0)<x^\ast$ a.s., or equivalently, 
\begin{equation}\label{Hyp_add_nablau0}
(\delta \|Q\|_{L^2} )^\alpha \| \nabla u_0\|_{L^2 }  < \| \nabla Q\|_{L^2} \|Q\|_{L^2}^\alpha\quad \mbox{ a.s.}. 
\end{equation}

Let $f(x)=\frac{1}{2}(x^2-Bx^{n\sigma})$ be defined for $x\geq 0$ as in Section \ref{inter_Stra}, see also Figure \ref{F:1}. 
 Using the definition of $x^*$ and the value of $f(x^*)$ from \eqref{identification}, we deduce that  for $s\in [0, \tau_A]$ we have $f(X(s)) \leq A f(x^*)<f(x^*)$ a.s..
The a.s. continuity of $\| \nabla u(s)\|_{L^2}$ implies
$$
(\delta \|Q\|_{L^2})^\alpha  \sup_{s\leq \tau_A} \big[ 1_{\Omega_s}\, \| \nabla u(s)\|_{L^2} \big] \equiv X(\tau_A)<x^* \quad \mbox{\rm a.s.}. 
$$
Therefore, for any $\delta \in (\beta,1)$
\begin{equation}\label{up_grad}
\delta^\alpha  \; 1_{\Omega_s} \; \|\nabla u(s)\|_{L^2}  \leq \| \nabla Q\|_{L^2} <\infty \quad \mbox{\rm a.s. \, for } s\leq \tau_A.
\end{equation}
 Since $\Omega_T \subset \Omega_s$ for $s\leq T$, this proves that the solution $X_s$ exists almost surely 
on the subset $\Omega_T$ for $s\leq \tau_A\wedge T$. 
Letting $\delta \to 1$ we also deduce the bound
\begin{equation}\label{E:grad-bound3}
\|\nabla u(s)\|_{L^2}  \leq \| \nabla Q\|_{L^2} \quad { \mbox{\rm a.s. \, on $\Omega_T$ \, for } } s\leq \tau_A \wedge T.
\end{equation}
\end{proof}
 
The next theorem is the main result of this section. It gives some bound on the (random) time interval, where \eqref{NLS_add} is well-posed
with strictly positive probability. As expected, the time bound is a decreasing function of the noise strength, which converges to $+\infty$ as
the Hilbert-Schmidt norms of $\phi$ converge to 0. 
\begin{theorem}\label{th-wp-inter-add}
Let $\phi \in L^{0,1}_2$ and $u_0$ be an ${\mathcal F}_0$-measurable random variable taking values in $H^1$ a.s., and such that for some
positive constants $\beta$ and $\gamma$ such that $\beta^2+\gamma<1$, one has a.s.
\begin{equation}\label{E:ME-add-inter}
\|u_0\|_{L^2} \leq \beta \|Q\|_{L^2}, \quad  H(u_0)\leq \gamma H(Q) \quad \mbox{\rm and } \|\nabla u_0\|_{L^2} < \| \nabla Q\|_{L^2}.
\end{equation}
Let $\tau^*(u_0)$ denote the stopping time defined in Theorem  \ref{lwp_add} and set $\tilde{\epsilon}=\frac{1-\gamma - \beta^2}{4(1-\gamma)-\beta^2}$. 

Then $P(\tau^*(u_0 )>T)>0$ for $T<\tilde{T}$, where 
\[ \tilde{T}:=  \Big( \frac{2 F(\tilde{\epsilon})}{b^*+\sqrt{ (b^*)^2 + G(\tilde{\epsilon}) F(\tilde{\epsilon})}}\Big)^2
\]
with 
\begin{align*}
 F(\tilde{\epsilon})&=1-\gamma - \frac{\beta^2}{1-\tilde{\epsilon}}, \quad
G(\tilde{\epsilon})=\frac{4 }{M(Q)} \|\phi\|_{L^{0,1}_2}^2 \; \Big( \frac{9+\tilde{\epsilon}}{\tilde{\epsilon}(1-\tilde{\epsilon})} 
+ \frac{1-s_c}{s_c}\Big), \\
b^*&= \;  \frac{3 \|Q\|_{L^2}}{H(Q)}  \Big[  \| \phi\|_{L^{0,1}_2} \, \Big( \frac{n}{2(1-s_c)}\Big)^{\frac{1}{2}}  + 
 K C(\phi)  \Big( \frac{n}{2(1-s_c)}\Big)^{n \sigma \, \frac{2\sigma +1}{2(2\sigma +2)}}\Big] ,
\end{align*}
and $C(\phi)$ as in \eqref{E:Const-phi}. 
\end{theorem}
\begin{proof}
We first prove upper estimates of the expected value of energy localized on the subset $\Omega_t$ where mass is controlled.
Note that by the assumption, for any $\delta \in (\beta,1)$, the inequalities  \eqref{Hyp_add_nablau0} and \eqref{up_grad} are satisfied. 
Since by assumption $\EX \big( H(u_0)\big) \leq \gamma H(Q)$, 
multiplying \eqref{E_sup_H} by $(\delta \|Q\|_{L^2})^{2 \alpha}$, we deduce that for every $T>0$, 
\begin{align}\label{E_sup_H_loc}
\EX \Big( \sup_{s\leq \tau_A\wedge T} \big[ 1_{\Omega_s} H(u(s))  (\delta \|Q\|_{L^2})^{2\alpha} \big] \Big) 
&\leq  
\big[ \gamma H(Q)+ \frac{1}{2} \| \phi\|_{L^{0,1}_2}^2 \EX (\tau_A\wedge T)\big] (\delta \|Q\|_{L^2})^{2\alpha}\\ 
&+T(1)+T(2), \label{E:T1-2}
\end{align}
where  
\begin{align*}
T(1)=& \; 3\, \| \phi\|_{L^{0,1}_2}  \EX \Big( \sqrt{\tau_A \wedge T} \; (\delta \|Q\|_{L^2})^{2\alpha} \; \sup_{s\leq \tau_A \wedge T}  { \big[ 1_{\Omega_s} 
 \| \nabla u(s)\|_{L^2}    \big]   } \Big) , \\
T(2)=&\;  3\, (c_{GN})^{\frac1{2\sigma+2}}
\, C(\phi) \, \EX \Big( \sqrt{\tau_A \wedge T} 
\; (\delta \|Q\|_{L^2})^{2\alpha} \; \sup_{s\leq \tau_A \wedge T} \big[ 1_{\Omega_s} \|u(s)\|_{L^{2\sigma +2}}^{2\sigma +1} \big] \Big).
\end{align*}
The upper bound \eqref{up_grad} together with \eqref{E:normQ1} imply 
\begin{align}\label{T1}
T(1)\leq  &\; 3\, \| \phi\|_{L^{0,1}_2}  \| \nabla Q \|_{L^2} M(Q)^\alpha \delta^\alpha \, \EX \big( \sqrt{\tau_A \wedge T}\big) \nonumber  \\
\leq & \; 3\, \| \phi\|_{L^{0,1}_2} \delta^\alpha \Big( \frac{n}{2(1-s_c)} \Big)^{\frac{1}{2}} \|Q\|_{L^2}^{2\alpha +1} 
\, \EX \big( \sqrt{\tau_A \wedge T}\big).
\end{align}
On the other hand, the inequalities \eqref{GN}, \eqref{up_grad} together with \eqref{constant_GN} and \eqref{E:normQ2} yield
\begin{align}	\label{T2}
&T(2)\leq   \; 3 \Big( \frac{K}{\|Q\|_{L^2}^{2\sigma}} \Big)^{\frac{1}{2\sigma +2} + \frac{2\sigma+1}{2\sigma +2}} \, C(\phi) 
(\delta \|Q\|_{L^2})^{2\alpha} \nonumber \\
&\qquad \qquad \times \EX \Big( \sqrt{\tau_A\wedge T} \sup_{s\leq \tau_A\wedge T} \Big\{ 1_{\Omega_s}\, 
 \| \nabla u(s)\|_{L^2}^{n\sigma \, \frac{2\sigma+1}{2\sigma+2}} 
  \|u(s)\|_{L^2}^{[2-(n-2)\sigma]\frac{2\sigma +1}{2\sigma +2}} \Big\} \Big) \nonumber  \\  
 & \quad \leq 
 \; 3 \frac{K}{\|Q\|_{L^2}^{2\sigma}} 
\, C(\phi) \, \EX \Big( \sqrt{\tau_A \wedge T} \sup_{s\leq \tau_A\wedge T} \Big\{  1_{\Omega_s} 
\big( \| \nabla u(s)\|_{L^2} \delta ^\alpha\big)^{n\sigma \, \frac{2\sigma+1}{2\sigma+2}} \Big\} \nonumber \\
&\qquad \qquad \times 
\|Q\|_{L^2}^{2\alpha + [2-(n-2)\sigma)] \,\frac{2\sigma +1}{2\sigma +2}}\; 
 \delta^{2\alpha + [2-(n-2)\sigma-n\alpha \sigma ] \frac{2\sigma +1}{2\sigma +2}}   \; 
\Big) \nonumber  \\
&\quad \leq  \; 3 \frac{K\, C(\phi) }{\|Q\|_{L^2}^{2\sigma}}  \, \EX \big(\sqrt{\tau_A\wedge T}\big)\, 
 \|\nabla Q\|_{L^2}^{n\sigma \, \frac{2\sigma +1}{2\sigma +2}}  \; 
\|Q\|_{L^2}^{2\alpha +  [2-(n-2)\sigma]  \frac{2\sigma+1}{2\sigma +2}} \; 
 \delta^{2\alpha + 2\sigma +1 - n\sigma (1+\alpha)  \frac{2\sigma+1}{2\sigma+2}}   \nonumber  \\
& \quad \leq \; 3 K C(\phi)   \Big( \frac{n}{2(1-s_c)}\Big)^{\frac{n \sigma}{2}  \, \frac{2\sigma +1}{2\sigma +2}} 
 \|Q\|_{L^2}^{2\alpha +1}
\;  \delta^{2\alpha +2\sigma +1  - n \sigma \, (1+\alpha)  \frac{2\sigma+1}{2\sigma+2}}   \; \EX \big(\sqrt{\tau_A\wedge T}\big).
\end{align}
Plugging the  estimates \eqref{T1} and \eqref{T2} into \eqref{E_sup_H_loc}-\eqref{E:T1-2} and using the fact that $\EX (\tau_A \wedge T)\leq T$ and
$\EX \big(\sqrt{\tau_A \wedge T}\big) \leq \sqrt{T}$, we deduce 
\begin{align}
\EX \Big(& \sup_{s\leq \tau_A\wedge T} \big[ 1_{\Omega_s} H(u(s))  \big] \Big) \leq  \gamma H(Q)
+ \frac{1}{2} \|\phi\|_{L^{0,1}_2}^2 T 
+  3  \frac{\| \phi\|_{L^{0,1}_2}}{ \delta^\alpha} \Big( \frac{n}{2(1-s_c)} \Big)^{\frac{1}{2}} \|Q\|_{L^2}
\, \sqrt{T} \nonumber \\
&+ 3 K C(\phi)   \Big( \frac{n}{2(1-s_c)}\Big)^{\frac{n \sigma}{2}  \, \frac{2\sigma +1}{2\sigma +2}} 
 \|Q\|_{L^2}
\;   \delta^{2\sigma +1  - n \sigma (1+\alpha) \, \frac{2\sigma+1}{2\sigma+2}}   \; \sqrt{T} .
\end{align}   
Using Lemma \ref{wp_add_tauA}, given some positive time $T$,
 we need to upper estimate the probability of the ``bad" set $\Omega_T^c \cup \{ \tau_A <T\}$, where we do
not know if the equation \eqref{NLS_add} is well-posed on the time interval $[0,T]$. 
The upper estimate \eqref{E:star} and the inclusion $\Omega_t\subset \Omega_s$ for $s\leq t$ imply 
\begin{align*}
P(\Omega_T^c \cup & \{ \tau_A <T\}) = P(\Omega_T^c) + P(\Omega_T \cap \{ \tau_A <T\}) \\
\leq & \frac{\beta^2}{\delta^2(1-\epsilon)} + \frac{9+\epsilon}{\epsilon(1-\epsilon)} \frac{\|\phi\|_{L^{0,1}_2}^2 T}{\delta^2 M(Q)} +
P\Big( \Omega_T \cap \Big\{ \sup_{s\leq \tau_A\wedge T} \big[ 1_{\Omega_s} H(u(s)) \big] \geq A H(Q)\Big\} \Big) \\
\leq &  \frac{\beta^2}{\delta^2(1-\epsilon)} + \frac{9+\epsilon}{\epsilon(1-\epsilon)} \frac{\|\phi\|_{L^{0,1}_2} ^2 T}{\delta^2 M(Q)} + 
\frac{1}{A H(Q)} \EX \Big( \sup_{s\leq \tau_A \wedge T}  \big[ 1_{\Omega_s} H(u(s)) \big] \Big)\\
\leq & a  T + b \sqrt{T} +c,
\end{align*}
where by \eqref{E_sup_H_loc} --  \eqref{T2} and \eqref{H_Q}, we have 
\begin{align*}
a= &\|\phi\|_{L^{0,1}_2}^2 \Big[ \frac{9+\epsilon}{\epsilon(1-\epsilon)\delta^2 M(Q)} +  \frac{1}{2 A H(Q)} \Big] =
\frac{\|\phi\|_{L^{0,1}_2}^2 }{M(Q)} \Big[  \frac{9+\epsilon}{\epsilon(1-\epsilon)\delta^2 } + \frac{1-s_c}{A\, s_c}
\Big]      , \\
b=&  \frac{3  \|Q\|_{L^2}}{A H(Q)}  \Big[   \| \phi\|_{L^{0,1}_2}   \Big( \frac{n}{2(1-s_c)}\Big)^{\frac{1}{2}}  \delta^{-\alpha}  
+  K C(\phi)  \Big( \frac{n}{2(1-s_c)}\Big)^{\frac{n \sigma}{2}  \, \frac{2\sigma +1}{2\sigma +2}} 
\;  \delta^{ 2\sigma +1  - n \sigma  \, (1+\alpha)  \frac{2\sigma+1}{2\sigma+2}} \Big],   \\
c=&\frac{\beta^2}{\delta^2(1-\epsilon)}  + \frac{\gamma}{A}.
\end{align*}
To make sure that $c<1$, we require $\frac{\beta^2}{\delta^2}  + \frac{\gamma}{A} <1$. Then choose $\epsilon>0$ such that 
\[  c = \frac{\beta^2}{\delta^2(1-\epsilon)}  + \frac{\gamma}{A}  <1\; ; \quad \mbox{\rm set } \; \rho:=1-c.\]
If these constraints are satisfied, we have to choose $T>0$ such that $P(\Omega_T^c \cup  \{ \tau_A <T\}) <1$, namely, $ aT+b\sqrt{T}-\rho<0$. 
The discriminant of the polynomial $a x^2 +bx -\rho$ is $\overline{\Delta}=b^2 +4a\rho>0$ and the roots are 
$x_1=- \frac{b+\sqrt{\overline{\Delta}}}{2a}<0$ and $x_2=\frac{\sqrt{\overline{\Delta}}-b}{2a}>0$. Thus, 
$a T + b\sqrt{T}<\rho$ for $0<T\leq x_2^2$. 

Let $\|u_0\|_{L^2} \leq \beta \|Q\|_{L^2}$ and $H(u_0)\leq\gamma H(Q)$ a.s. as in the statement of the theorem \eqref{E:ME-add-inter}, 
and suppose that $\beta^2 + \gamma <1$ and let $\epsilon \in (0,1)$ be such that  $\frac{\beta^2}{1-\epsilon}
 + \gamma<1$, that is, $\epsilon \in \big(0, \frac{1-\beta^2-\gamma}{1-\gamma}\big)$. 
Let $\delta \to 1$ and $A\to 1$. Then $P(T<\tau^*(u_0))>0$ for $T<T^*$ with $T^*=(x^*_2)^2$, where $x_2^*$ is the positive root of the polynomial
$a^* x^2 + b^* x -\rho^*=0$ with 
\begin{align*}
 a^* =& \; \frac{\|\phi\|_{L^{0,1}_2}^2 }{M(Q)} \Big[ \frac{9+\epsilon}{\epsilon(1-\epsilon)} + \frac{1-s_c}{s_c} \Big], \\
 b^*=& \;  \frac{3 \|Q\|_{L^2}}{H(Q)} \Big[  \| \phi\|_{L^{0,1}_2} \, \Big( \frac{n}{2(1-s_c)}\Big)^{\frac{1}{2}}  + 
  K C(\phi)  \Big( \frac{n}{2(1-s_c)}\Big)^{\frac{n \sigma}{2} \, \frac{2\sigma +1}{2\sigma +2}} \Big] ,  \\
\rho^*= &1-c^*>0, \quad \mbox{and}\\ 
 c^*=&\frac{\beta^2}{1-\epsilon}  + \gamma.
\end{align*} 
The corresponding discriminant is $\Delta^* = (b^*)^2 +4 a^* \rho^*$ and $x_2^* (\epsilon)= \frac{\sqrt{\Delta^*} -b^*}{2a^*}$.

We next look for the optimal constant $\epsilon \in \big(0, \frac{1-\beta^2-\gamma}{1-\gamma}\big)$. The
conditions required for this parameter are non-trivial 
to solve explicitly; we, therefore, only provide a sub-optimal estimate. 
We have  $x_2^* := x_2^*(\epsilon)= \frac{2 \rho^*}{b^* + \sqrt{(b^*)^2+4 a^*\rho^*}}$ and we want to maximize $x_2^*(\epsilon)$.
 Let 
$ g=\frac{2-(n-2)\sigma}{n\sigma -2} \equiv \frac{1-s_c}{s_c}$ (hence, $1+g = \frac1{s_c}$), $ \tilde{\rho}=1-\gamma-\beta^2>0$, and set  
\[ F(\epsilon)=\frac{\tilde{\rho}-\epsilon(1-\gamma)}{1-\epsilon}\quad \mbox{ and } \;  
G(\epsilon)=\frac{4 \|\phi\|_{L^{0,1}_2}^2 }{M(Q)} \; \Big( \frac{9+\epsilon(1+g) - \epsilon^2 g}{\epsilon(1-\epsilon)} \Big)= \frac{4 \|\phi\|_{L^{0,1}_2}^2 }{M(Q)} \; \Big( \frac{9+\epsilon}{\epsilon(1-\epsilon)} + g\Big).\]
Then 
\[ x_2^*(\epsilon) = \frac{F(\epsilon)}{b^* + \sqrt{(b^*)^2 + G(\epsilon) F(\epsilon)}}.\]
We have $F'(\epsilon) = -\frac{\beta^2}{(1-\epsilon)^2}$ and $G'(\epsilon)= \frac{4}{M(Q)} \|\phi\|_{L^{0,1}_2}^2 \frac{\epsilon^2 + 18 \epsilon -9}{\epsilon^2
(1-\epsilon)^2}$, and $(x_2^*(\epsilon))' = C(\epsilon) D(\epsilon)$, where $C(\epsilon)>0$ and
\begin{align*}
D(\epsilon) = &\Big\{  2 b^*\sqrt{ (b^*)^2 + G(\epsilon) F(\epsilon)} +2\big[ (b^*)^2 + G(\epsilon) F(\epsilon)\big]\Big\} F'(\epsilon) 
-  F(\epsilon) \Big\{  F'(\epsilon) G(\epsilon) + F(\epsilon) G'(\epsilon)  \Big\}\\
\geq &\;  4 F'(\epsilon)  (b^*)^2 + 2 G(\epsilon) F(\epsilon)F'(\epsilon)  -\,   F(\epsilon) \big\{  F'(\epsilon) G(\epsilon) + F(\epsilon) G'(\epsilon)  \big\} \\
\geq & \; A(\epsilon):= 4 F'(\epsilon)  (b^*)^2 +  F(\epsilon) F'(\epsilon) G(\epsilon) - F(\epsilon)^2 G'(\epsilon).
\end{align*}
Furthermore, 
\begin{align*} 
A(\epsilon)& = 
- \frac{4 (b^*)^2 \beta^2}{(1-\epsilon)^2} - \frac{\tilde{\rho} -\epsilon(1-\gamma)}{1-\epsilon} \times \frac{\beta^2}{(1-\epsilon)^2} \times 
\frac{4 \|\phi\|_{L^{0,1}_2}^2 }{M(Q)} \; \times  \frac{9+\epsilon(1+g) - \epsilon^2 g}{\epsilon(1-\epsilon)}  \\
&\qquad - \frac{(\tilde{\rho} -\epsilon(1-\gamma))^2}{(1-\epsilon)^2} \times 
\frac{4\|\phi\|_{L^{0,1}_2}^2}{M(Q)} \times \frac{\epsilon^2 + 18 \epsilon -9}{\epsilon^2 (1-\epsilon)^2}  
= \frac{1}{\epsilon^2 (1-\epsilon)^4} B(\epsilon),
\end{align*} 
where
\begin{align*}
B(\epsilon)=& -4 (b^*)^2 \epsilon^2(1-\epsilon)^2 - \epsilon \big[ \tilde{\rho} -\epsilon(1-\gamma)\big] \beta^2 \frac{4 \|\phi\|_{L^{0,1}_2}^2 }{M(Q)}
\big(9+\epsilon(1+g) - \epsilon^2 g\big)\\
& - \frac{4\|\phi\|_{L^{0,1}_2}^2}{M(Q)} \big( \tilde{\rho} - \epsilon (1-\gamma)\big)^2 \big( \epsilon^2 + 18\epsilon -9\big)\\
=& \frac{36  \|\phi\|_{L^{0,1}_2}^2 }{M(Q)} \tilde{\rho}^2 -\epsilon  \frac{4\|\phi\|_{L^{0,1}_2}^2}{M(Q)}   \big[ 9 \tilde{\rho} \beta^2 + 18 \tilde{\rho}^2
+ 18 \tilde{\rho}(1-\gamma) \big] + C_2 \epsilon^2 + C_3 \epsilon^3 + C_4 \epsilon^4.
\end{align*} 
The linear part of $B(\epsilon)$ is $\frac{4}{M(Q)} \|\phi\|_{L^{0,1}_2}^2 \tilde{B}(\epsilon)$, where
\[ \tilde{B}(\epsilon) = 9 \tilde{\rho}^2 -\epsilon\big[ 9 \tilde{\rho} \beta^2 + 18 \tilde{\rho}^2 + 19 \tilde{\rho} (1-\gamma)\big].\]
Hence, the function $\tilde{B}(\epsilon)$ is positive 
on the interval $[0, \tilde{\epsilon}]$,
where 
\[ \tilde{\epsilon} = \frac{9 \tilde{\rho}^2}{9 \tilde{\rho} \beta^2+ 18 \tilde{\rho}(1-\gamma)  +18 \tilde{\rho}^2} 
= \frac{\tilde{\rho}} {\beta^2 + 2(1-\gamma) +2 \tilde{\rho}}
= \frac{1-\gamma - \beta^2}{4(1-\gamma) - \beta^2}.\]
It is easy to see that since $\gamma + \beta^2<1$, $0<\tilde{\epsilon} < \frac{1-\gamma -\beta^2}{1-\gamma}$, so that the (suboptimal) constant 
$\tilde{\epsilon}$  belongs to the desired interval, and the proof is complete. 
\end{proof}

\section{Blow-up in finite time}\label{S:B}
In this section we show that solutions to \eqref{NLS_Stra} in the intercritical case blow up before some time $T>0$ with positive probability in both multiplicative and additive cases, 
 given that the initial conditions remain under the so-called `mass-energy' threshold (and thus, have positive energy) and sufficiently large $L^2$-norm of the gradient 
 (e.g., as in the deterministic case in \cite{HR2007} or see Theorem \ref{T:main-deter}, Part 2, in the introduction). We require the intensity of the noise to be small.  
\smallskip

As in the deterministic case, the argument is based on the variance $V(u) $ defined on the set 
\begin{equation}\label{def_Sigma}
\Sigma = \Big\{ u\in H^1(\RR^n)\, : \int_{\RR^n}  |x|^2 |u(x)|^2 dx  \Big\}
\end{equation}
by 
\begin{equation}        \label{def_V}
 V(u)=\int_{\RR^n} |x|^2 |u(x)|^2 dx,\quad  u\in \Sigma. 
\end{equation}
We also introduce the momentum 
\begin{equation}    \label{def_G}
G(u) = {\rm Im}\int_{\RR^n} u(x) \, x\!\cdot\!\nabla \bar{u}(x) \, dx, \quad u\in \Sigma.
\end{equation}
We first address the case of the multiplicative noise in \S \ref{S:M-bup} and then the case of the additive noise in \S \ref{S:A-bup}. 

\subsection{Multiplicative noise}\label{S:M-bup}
We start with solutions to \eqref{NLS_Ito}. 
The following lemma is a rephrasing of \cite[Proposition~3.2]{deB_Deb_AnnProb}.
\begin{lemma}   \label{lem_v-G-multi}
Let $\phi$ satisfy the condition {\bf (H1)},  $u$ be the solution to \eqref{NLS_Stra} and assume that $u_0\in \Sigma$ a.s.. Then for any
stopping time $\tau<\tau^*(u_0)$ a.s., we have
\begin{equation}    \label{V(u)}
V(u(\tau)) = V(u_0) +4 \int_0^\tau G(u(s)) ds\quad a.s.
\end{equation}
and 
\begin{align} \label{G(u)}
G(u(\tau)) = G(u_0) + & 2n\sigma \int_0^\tau H(u(s)) ds - 2\sigma s_c \int_0^\tau \|\nabla u(s)\|_{L^2}^2 ds \nonumber \\
&+ \sum_{l\geq 0} \int_0^\tau \int_{\RR^n} |u(s,x)|^2 \, x\!\cdot\!\nabla(\phi e_l)(x) \, dx \, d\beta_l(s)\quad a.s. .
\end{align}
\end{lemma}
\begin{proof} The equation \eqref{V(u)} is proved in \cite{deB_Deb_AnnProb}, see (3.4). To prove \eqref{G(u)}, we re-write $G(u(\tau))$ using It\^o's formula as on page 1095 in \cite{deB_Deb_AnnProb}. This yields
\begin{align*} 
G(u(\tau)) = & G(u_0) + 2\int_0^\tau \int_{\RR^d} |\nabla u(s,x)|^2 dx ds - \frac{\sigma n}{\sigma +1} \int_0^\tau \int_{\RR^n}
|u(s,x)|^{2\sigma +2} dx \\
&+\sum_{l\geq 0} \int_0^\tau \int_{\RR^n} |u(s,x)|^2 \, x\!\cdot\!\nabla(\phi e_l)(x) \, dx \, d\beta_l(s)\\
= & G(u_0) +2n\sigma \int_0^\tau H(u(s)) ds + (2-n\sigma)  \int_9^\tau \| \nabla u(s)\|_{L^2}^2 ds\\
&+\sum_{l\geq 0} \int_0^\tau \int_{\RR^n} |u(s,x)|^2 \, x\!\cdot\!\nabla(\phi e_l)(x) \, dx d\beta_l(s).
\end{align*}
Since $n\sigma -2 = 2\sigma s_c$, this completes the proof of \eqref{G(u)}.
\end{proof}

The following result describes a sufficient condition on the initial condition and on some deterministic positive time with the 
blow-up occurring before that time with positive probability. 
Recall the notation from the condition (H2) that $f^1_\phi = \sum_{l\geq 0} |\nabla (\phi e_l)|^2$ and $M_\phi=\| f^1_\phi\|_{L^\infty}$.
To highlight the main idea of the proof, we first assume that $u_0$ is deterministic. 

\begin{theorem}\label{th_blowup1}
Let $ \sigma \in \mathfrak R_{inter}$. 
Consider $\phi$ satisfying the conditions {\bf(H1)} and {\bf (H2)}.
Let $u_0 \in \Sigma $ be deterministic 
and suppose 
that for some constants $\beta, \delta_0$ such that $0<\beta < 1 < \delta_0$, we have 
\begin{equation} \label{cond-u_0-Q-det}
H(u_0) M(u_0)^\alpha = \beta H(Q) M(Q)^\alpha  \quad \mbox{\rm  and}\quad  \| \nabla u_0\|_{L^2} \|u_0\|_{L^2}^\alpha
= \delta_0 \|\nabla Q\|_{L^2} \|Q\|_{L^2}^\alpha.
\end{equation}
Then for $T>0$ large enough and $M_\phi$ small enough, we have $P\big(\tau^*(u_0)\leq T\big) >0$, where $\tau^*(u_0)$ is defined in Theorem \ref{local_wp}.
\end{theorem}

We first prove a result similar to \cite[Theorem~4.1]{deB_Deb_AnnProb} but for our case when $H(u_0)\geq 0$.
\begin{Prop}\label{Prop-blow-det}
Let $u_0$ and $\phi$ satisfy the assumptions of Theorem \ref{th_blowup1}. 
Suppose that, for some $\epsilon >0$ and $N>0$, a positive time $t$ and $M_\phi$ satisfy the
following conditions
\begin{align}
   4 M(u_0)^{\alpha +1} M_\phi t^2 \big( 1+\frac{t}{3}\big)  &\leq \epsilon,       \label{cond_1}\\
   \frac{32}{15} n \sigma M(u_0)^{\alpha + \frac{1}{2}} \sqrt{M_\phi} N t^{\frac{5}{2}} & \leq \epsilon.    \label{cond_2} 
\end{align}
For $1<\delta<\delta_0$ set 
\begin{align}
    \tau_\delta &= \inf\{ s\geq 0 : \|\nabla u(s)\|_{L^2} \|u_0\|_{L^2}^\alpha \leq \delta \|\nabla Q\|_{L^2} \|Q\|_{L^2}^\alpha\}\wedge \tau^*(u_0),    \label{def-tau_delta-multi}\\
    \tilde{\tau}_N
& = \inf\{ t\geq 0 : \|\nabla u(s)\|_{L^2} \geq N\} \wedge \tau^*(u_0).
\label{def-tildetauK-multi}
\end{align}
Let $\tau = \tau_\delta \wedge \tilde{\tau}_N$ and suppose that
\begin{equation}    \label{cond-3}
    V(u_0) M(u_0)^\alpha + 2\epsilon +4 G(u_0) M(u_0)^\alpha \EE(t\wedge \tau) 
    - 4\sigma s_c (\delta^2-\beta) \|\nabla Q\|_{L^2}^2 M(Q)^\alpha\EE\big( (t\wedge \tau)^2\big) <0. 
\end{equation}
Then $P(\tau^*(u_0)\leq t)>0$. 
\end{Prop}

\begin{proof}
Assume that $t<\tau^*(u_0)$ a.s. and write $V(u)$ using \eqref{V(u)}, \eqref{G(u)} and \eqref{energy}. Then
\begin{align*}
V(u(t&\wedge \tau)) = V(u_0) + 4 G(u_0) (t\wedge \tau)  + 4n\sigma H(u_0) (t\wedge \tau)^2  - 8\sigma s_c \int_0^{t\wedge \tau} ds \int_0^s \|\nabla u(r)\|_{L^2}^2 dr \\
& +8n\sigma \int_0^{t\wedge \tau} \int_0^s \int_0^r \int_{\RR^n} |u(r_1,x)|^2 f^1_\phi(x) \,
dx dr_1 dr ds\\
&-8n\sigma \int_0^{t\wedge \tau} \int_0^s \int_0^r {\rm Im}\sum_{l\geq 0} \int_{\RR^n}  \bar{u}(r_1,x) \nabla u(r_1,x)\!\cdot\!\nabla(\phi e_l)(x) \, dx
d\beta_l(r_1) dr ds\\
& +4 \int_0^{t\wedge \tau}  \int_0^s \sum_{l\geq 0} \int_{\RR^n} |u(r,x)|^2 \, x\!\cdot\!\nabla(\phi e_l)(x) \, dx d\beta_l(r) ds.
\end{align*}
Multiplying by $M(u_0)^\alpha$ and taking expected values, we deduce 
\begin{align}    
  &M(u_0)^\alpha \EE\big(V(t\wedge \tau)\big)   \leq  M(u_0)^\alpha V(u_0) + 4 M(u_0)^\alpha G(u_0) \EE\big( t\wedge \tau\big) 
\nonumber \\
  &  \qquad 
  +4 n\sigma M(u_0)^\alpha  H(u_0) \EE\big( (t\wedge \tau) ^2\big)-8\sigma s_c M(u_0)^\alpha \EE\int_0^{t\wedge \tau} ds \int_0^s \|\nabla u(r)\|_{L^2}^2 dr + \sum_{i=1}^3 T_i,  \label{E:MEV-1b}
\end{align}
where 
\begin{align*}
 T_1= &8 \, \EE\Big( M(u_0)^\alpha \int_0^{t\wedge \tau} ds \int_0^s dr \int_0^r \int_{\RR^n}   |u(r_1,x)|^2 f^1_\phi(x) dx  \, dr_1 \Big), \\
T_2=&  -8n\sigma \EE\Big( M(u_0)^\alpha  \int_0^{t\wedge \tau}ds \int_0^s dr\int_0^r {\rm Im}\sum_{l\geq 0} \int_{\RR^n} 
\bar{u}(r_1,x)\,  \nabla u(r_1,x)\!\cdot\! \nabla(\phi e_l)(x) dx
d\beta_l(r_1) \Big),\\
T_3=&4  \EE\Big(M(u_0)^\alpha \int_0^{t\wedge \tau}  ds\int_0^s \sum_{l\geq 0} \int_{\RR^n} |u(r,x)|^2 \, x\!\cdot\!\nabla(\phi e_l)(x) dx d\beta_l(r) \Big).
\end{align*}
Since for the time $r\leq \tau_\delta$ a.s. we have
\[ M(u_0)^\alpha \|\nabla u(r)\|_{L^2}^2 \geq \delta^2 M(Q)^\alpha \|\nabla Q\|_{L^2}^2, \quad \mbox{\rm  while}\quad 
M(u_0)^\alpha H(u_0) \leq \beta M(Q)^\alpha H(Q),\]
using  the identity \eqref{H_Q},
we deduce that  
\begin{align}\label{compensation}
4n\sigma &M(u_0)^\alpha H(u_0) \EE\big( (t\wedge \tau)^2\big) -8 \sigma s_c M(u_0)^\alpha \,
\EE\Big(\int_0^{t\wedge \tau} ds \int_0^s \|\nabla (u(r))\|_{L^2}^2 dr \Big) 
\nonumber \\
& \leq 4\sigma s_c \beta \|\nabla Q\|_{L^2}^2 M(Q)^\alpha \, \EE\big( (t\wedge \tau)^2\big)  - 8\sigma s_c \EE\Big( \int_0^{t\wedge \tau} ds \int_0^s
\delta^2 \|\nabla Q\|_{L^2}^2 M(Q)^\alpha dr \Big) \nonumber \\
&\leq -4\sigma s_c (\delta^2-\beta) \|\nabla Q\|_{L^2}^2 M(Q)^\alpha \, \EE\big( (t\wedge \tau)^2\big).
\end{align}
The upper estimates  \eqref{E:MEV-1b} and \eqref{compensation} yield 
\begin{align}\label{MEV-2}
   M(u_0)^\alpha \, \EE\big( V(t\wedge \tau)\big)  \leq &   M(u_0)^\alpha V(u_0) + 4 M(u_0)^\alpha G(u_0) \EE\big( t\wedge \tau\big)  \nonumber \\
  & -4 \sigma s_c (\delta^2-\beta) \|\nabla Q\|_{L^2}^2 M(Q)^\alpha \EE\big( (t\wedge \tau)^2\big) + \sum_{i=1}^3 T_i.
\end{align}
We now estimate the last three terms. Since $M_\phi = \|f^1_\phi\|_{L^\infty}$, we have
\begin{equation}    \label{estim_T1}
T_1\leq 8 M(u_0)^{\alpha +1} M_\phi \int_0^t ds \int_0^s dr \int_0^r dr_1 = \frac{4}{3} M(u_0)^{\alpha +1}M_\phi t^3. 
\end{equation}
Using Fubini's Theorem and the Cauchy-Schwarz inequality with respect to $dP$  (the expected value) and then to $dx$, we obtain
\begin{align}\label{estim_T2}
|T_2|\leq & 8n\sigma M(u_0)^\alpha \int_0^t ds \int_0^s dr \, \EE \Big( \Big| \int_0^{r\wedge \tau} \sum_{l\geq 0} \int_{\RR^n} \bar{u}(r_1,x) \nabla u(r_1,x)\!\cdot\!\nabla (\phi e_l) (x) dx d\beta_l(r_1)\Big| \Big) \nonumber \\
\leq &  8n\sigma M(u_0)^\alpha \int_0^t ds \int_0^s dr \Big\{ \EE\Big( \int_0^{t\wedge \tau}  \sum_{l\geq 0}  \Big| 
\int_{\RR^n} \bar{u}(r_1,x) \nabla u(r_1,x)\!\cdot\!\nabla (\phi e_l) (x) dx\Big|^2 dr_1 \Big) \Big\}^{\frac{1}{2}}\nonumber \\
\leq & 8n\sigma M(u_0)^\alpha \int_0^t ds \int_0^s dr \Big\{ \EE\Big[ \int_0^{t\wedge \tau}  \sum_{l\geq 0}  \Big( \int_{\RR^n} |\bar{u}(r_1,x)|^2 
|\nabla (\phi e_l)(x)|^2 dx \Big) \nonumber \\
&\hspace{2cm}  \times \Big( \int_{\RR^n} |\nabla u(r_1,x)|^2 dx\Big) dr_1 \Big) \Big] \Big\}^{\frac{1}{2}} \nonumber \\
\leq & 8\sigma M(u_0)^{\alpha + \frac{1}{2}} \sqrt{M_\phi} \int_0^t ds \int_0^s dr \Big\{ \EE\big( {N^2} (r\wedge \tau)\big) \Big\}^{\frac{1}{2}} ~~({\mbox{recalling ~that}} ~~\tau \leq \tilde\tau_N)  \nonumber \\
\leq & 8n\sigma M(u_0)^{\alpha + \frac{1}{2}} \sqrt{M_\phi} N \frac{4}{15} t^{\frac{5}{2}}.
\end{align}
Similar computations yield
\begin{align*}       
 |T_3|\leq & 4 M(u_0)^\alpha \int_0^t ds \,\EE\Big( \Big| \int_0^{s\wedge \tau} \sum_{l \geq 0}l \int_{\RR^n} |u(r,x)|^2 \,x\!\cdot\!\nabla(\phi e_l)(x) dx d\beta_l(r) \Big| \Big)
 \nonumber \\
 \leq & 4 M(u_0)^\alpha \int_0^t ds \Big\{ \EE\Big( \int_0^{s\wedge \tau} \sum_{l\geq 0} \Big|\int_{\RR^n} |u(r,x)|^2 \, x\!\cdot\! \nabla(\phi e_l)(x) dx\Big|^2  dr  \Big)\Big\}^{\frac{1}{2}}
 \nonumber \\
 \leq & 4 M(u_0)^\alpha \int_0^t ds \Big\{ \EE\Big( \int_0^{s\wedge \tau} \sum_{l\geq 0}  \Big( \int_{\RR^n} |u(r,x)|^2 |\nabla(\phi e_l)(x)|^2 dx\Big)
 \Big( \int_{\RR^n }|x|^2 |u(r,x)|^2 dx\Big) dr \Big] \Big\}^{\frac{1}{2}} \nonumber \\
 \leq & 4 M(u_0)^{\alpha + \frac{1}{2}} \sqrt{M_\phi} \int_0^t ds \Big\{ \EE\int_0^s V\big( u(r\wedge \tau)\big) dr \Big) \Big\}^{\frac{1}{2} }\nonumber \\
 \leq & 4 M(u_0)^{\alpha + \frac{1}{2}} \sqrt{M_\phi} \sqrt{t} \Big\{ \int_0^t ds \int_0^s {\EE} \big( V(u(r\wedge \tau)) dr \big) \Big\}^{\frac{1}{2}}.
\end{align*}
Young's inequality implies that for $\bar{\epsilon}>0$, 
\[ |T_3|\leq \bar{\epsilon} \int_0^t ds \int_0^s \EE\big( V(u(r\wedge \tau))\big) dr + \frac{4}{\bar{\epsilon}} M(u_0)^{2\alpha +1} M_\phi t.  
\]
Choosing $\bar{\epsilon}$ such that $t\bar{\epsilon} = M(u_0)^\alpha$, we obtain
\begin{equation}        \label{estim_T3}
    |T_3| \leq \int_0^t M(u_0)^\alpha \EE\big( V(u(r\wedge \tau))\big) dr + 4 M(u_0)^{\alpha +1} M_\phi t^2. 
\end{equation} 
Collecting estimates for $T_1, T_2, T_3$ and recalling the bounds \eqref{cond_1}-\eqref{cond_2}, we obtain 
$$
\Big| \sum_{i=1}^3 T_i \Big| \leq \int_0^t M(u_0)^\alpha \EE\big( V(u(s\wedge \tau))\big) ds + 2\epsilon. 
$$
Putting this into the estimate \eqref{MEV-2}, we get an inequality on the expected value of the variance, from which we would like to extract the bound on it. 
However, before proceeding (and applying Gronwall's inequality), we need to make sure that this expected value is bounded a.s.. 
For that we refer to \cite{deB_Deb_AnnProb}: since $\tau \leq \tilde{\tau}_N$ and $t<\tau^*(u_0)$ a.s., the upper estimate (6.2) on page 1095 in \cite{deB_Deb_AnnProb} 
implies that $u(s\wedge \tau)\in \Sigma$ a.s. for
every $s\leq t$ and 
\[ \sup_{s\leq t} V(s\wedge \tau)\leq [4N^2t+V(u_0)] e^t \quad \mbox{\rm a.s.}.\]
Therefore, $\sup_{r\leq t} \EE(V(r\wedge \tau)) <\infty $, and Gronwall's lemma yields { for every $t>0$} 
\begin{align}       \label{estim_EV_Gronwall}
    M(u_0)^\alpha \EE\big[ V\big( u(t\wedge \tau)\big) \big] \leq &\big[ { M(u_0)^\alpha V(u_0)} + 2\epsilon +4 M(u_0)^\alpha G(u_0)) \EE(t\wedge \tau) \nonumber \\
 &   -4\sigma s_c (\delta^2-\beta) \|\nabla Q\|_{L^2}^2 M(Q)^\alpha \EE\big( (t\wedge \tau)^2\big)   ] e^t.
\end{align}
Using the assumption \eqref{cond-3}, we deduce that $M(u_0)^\alpha \, \EE\big( V(u(t\wedge \tau))\big) <0$, which brings a contradiction, since
$V(u)\geq 0$ for every $u\in \Sigma$. Therefore,  $P(  \tau^*(u_0) \leq t)>0$, which completes the proof. 
\end{proof}
We next prove Theorem \ref{th_blowup1}.
\smallskip

\noindent {\it Proof of Theorem~\ref{th_blowup1}}. ~
First, we note that if $P(\tau^*(u_0)<\infty)>0$, there exists $T>0$ such that $P(\tau^*(u_0)\leq T)>0$ and the proof is complete. Thus, we suppose that $\tau^*(u_0)=\infty$ a.s. and look for
a contradiction. 

Recalling that  for $B= \frac{C_{GN}}{\sigma +1}$ 
the function 
$f(x)= \frac{1}{2} (x^2-Bx^{n\sigma})$, defined on $[0,+\infty)$,  is strictly increasing on the interval $(0,x^*)$ and strictly decreasing on $(x^*, \infty)$,
where $x^*=\|\nabla Q\|_{L^2} \|Q\|_{L^2}^\alpha$ and $f(x^*) =H(Q) M(Q)^\alpha$,  see  Figure \ref{F:1}, we proceed as follows. 

Since $\|\nabla u(t)\|_{L^2}$ is a.s. continuous in $t$ and initially we have $x=\|\nabla u_0\|_{L^2} \|u_0\|_{L^2}^\alpha$
 greater than $\|\nabla Q\|_{L^2} \|Q\|_{L^2}^\alpha$ (that is, $x > x^*$ implying that $x$ is located on the decreasing side of the graph of $f$, see Figure \ref{F:1}), 
there exists $\delta \in (1,\delta_0)$ and $\gamma \in (\beta,1)$ such that for any $t>0$  we have a.s. 
\begin{equation}     \label{gamma-delta}
\sup_{s\in [0,t] }H(u(s)) M(u_0)^\alpha \leq \gamma H(Q) M(Q)^\alpha \; \Longrightarrow 
\inf_{s\in [0,t]}\|\nabla u(s)\|_{L^2} \|u_0\|_{L^2}^\alpha \geq \delta \|\nabla Q\|_{L^2} \|Q\|_{L^2}^\alpha .
\end{equation}

Let $\tau_\delta$ be as in Proposition \ref{Prop-blow-det} and define the stopping time $\tilde \sigma_\gamma$ by 
\begin{equation}\label{sigma_gamma1} \tilde{\sigma}_\gamma = \inf\{ s\geq 0 : H(u(s)) M(u_0)^\alpha
\geq \gamma H(Q) M(Q)^\alpha\} . 
\end{equation} 

Then \eqref{gamma-delta} implies that $\tilde{\sigma}_\gamma \leq \tau_\delta$.
As $T\to \infty$, we have $T\wedge \tilde\sigma_\gamma \to \tilde\sigma_\gamma$, 
and  the monotone convergence theorem  implies $\EE\big((T\wedge \tau_\delta)^2 \big)
\to \EE\big(\tau_\delta^2\big)$. 
We consider two cases, depending on the size of the last quantity. 
\smallskip

$\bullet$ If $\EE\big(\tau_\delta^2\big)=\infty$, given any $M>0$ there exists $T$ large enough, call it $T_0>0$, to ensure $\EE\big((T_0\wedge \tau_\delta)^2\big) \geq M^2$.\smallskip

$\bullet$ If $\EE(\tau_\delta^2)<\infty$, given any $\lambda \in (0,1)$, we may choose $T_0$ large enough to have 
$\EE\big((T_0 \wedge \tau_\delta)^2\big) \geq \lambda^2 \EE(\tau_\delta^2)$. 

Furthermore, since $\|\nabla u(t)\|_{L^2} \to \infty$ as $t\to \tau^*(u_0)$, we deduce that for $\tilde{\tau}_N$ defined by \eqref{def-tildetauK-multi},
 we have $\tilde{\tau}_N \to \tau^*(u_0)$ as $N \to \infty$. 
Hence, there exists $N_0$ such that for $\bar{\tau}_0=\tau_\delta \wedge \tilde{\tau}_{N_0}$ we have
\[ \EE \big((T_0 \wedge \bar{\tau}_0)^2\big) = \EE\big((T_0 \wedge \tau_\delta \wedge \tilde{\tau}_{N_0})^2\big) \geq 
\lambda^2 \EE\big((T_0 \wedge \tau_\delta)^2\big),\]
which implies that either $\EE\big((T_0 \wedge \bar{\tau}_0 )^2\big) \geq \lambda^2 M^2$
or $\EE\big((T_0 \wedge \bar{\tau}_0)^2\big) \geq \lambda^4 \EE(\tau_\delta^2)$. 

Now, that $T_0$ and $N_0$ have been defined, given $\epsilon >0$, choose $M_\phi$ small enough to ensure that
\begin{equation}    \label{cond1-2}
     4 M(u_0)^{\alpha +1}M_\phi T_0^2\Big( 1+\tfrac{T_0}{3}\Big) \leq \epsilon \quad \mbox{\rm and} \quad \tfrac{32}{15} n\sigma \sqrt{M_\phi} K_0 T_0^{\frac{5}{2}} 
\leq \epsilon, 
\end{equation}
and thus, the conditions \eqref{cond_1} and \eqref{cond_2} are satisfied.  

We next show that 
there exists $\epsilon>0$ (close to 0) and $T>0$ (large enough) such that for $\tau= \tau_\delta \wedge \tilde{\tau}_{N_0}$    
$$
 M(u_0)^\alpha V(u_0)  + 2\epsilon +4 M(u_0)^\alpha G(u_0) \EE(T\wedge \tau)
-4\sigma s_c (\delta^2-\beta) \|\nabla Q\|_{L^2}^2 M(Q)^\alpha \EE\big( (T\wedge \tau)^2\big) <0,
$$
i.e., the expression in square brackets on the right-hand side of \eqref{estim_EV_Gronwall} is negative  for a large time $T$. 

Set 
$$
a=4\sigma s_c (\delta^2-\beta) \|\nabla Q\|_{L^2}^2 M(Q)^\alpha, 
\quad b=4 |G(u_0)| M(u_0)^\alpha \quad \mbox{\rm and} \quad
c=M(u_0)^\alpha V(u_0) +2\epsilon, 
$$
and denote $X:= \big\{ \EE\big( (T\wedge \bar{\tau}_0)^2\big) \big\}^{\frac{1}{2}}$. 
Then the Cauchy-Schwarz inequality implies that 
$$ 
-a \EE\big( (T\wedge \bar{\tau}_0)^2\big) +b\EE\big( T\wedge \bar{\tau}_0 \big) +c \leq  -a X^2 +bX+c.
$$
Therefore, if $-aX^2+bX+c<0$, the condition \eqref{cond-3} is satisfied with $T\wedge \bar{\tau}_0$ instead of $t\wedge \tau$, 
and Prop. \ref{Prop-blow-det} implies that $P(T<\tau^*(u_0)) >0$. 
Let $X_1<X_2$ denote the roots of the polynomial $-aX^2+bX+c$. 
\smallskip

$\bullet$ If $\EE(\tau_\delta^2)=\infty$, we may choose $M$ large enough and $\lambda $ close to one to have $M\lambda >X_2$, which implies 
$X>X_2$. 

$\bullet$ If $\EE(\tau_\delta^2)<\infty$, then $\tilde{\sigma}_\gamma \leq \tau_\delta <\infty$ a.s.
Thus,  the a.s.  continuity  of the energy $H(u(\cdot))$ and  the definition of $\tilde{\sigma}_\gamma$ in \eqref{sigma_gamma1} imply
$H(u(\tilde{\sigma}_\gamma))  M(u_0)^\alpha = \gamma H(Q) M(Q)^\alpha$.
Multiplying \eqref{energy} by $M(u_0)^\alpha$, using \eqref{cond-u_0-Q-det} and $f^1_\phi = \sum_{l\in \NN} |\nabla (\phi e_l)|^2$,  
we deduce that for any stopping time $\tau < \tau^*(u_0)$, we have 
\begin{align}  \label{Mu_0H(u)}
M(u_0)^\alpha  H(u(\tau))= &\beta H(Q) M(Q)^\alpha +\frac{1}{2} M(u_0)^\alpha \int_0^\tau ds \int_{\RR^n} |u(s,x)|^2 f^1_\phi(x) dx  - M(u_0)^\alpha {\rm  Im}\big(I(\tau)\big), 
\end{align}
where  
\[ I(\tau):=  \int_0^\tau \sum_{l\geq 0}  \int_{\RR^n} 
\bar{u}(s,x) \nabla u(s,x)\!\cdot\!\nabla(\phi e_l)(x) dx d\beta_l(s)  .
\]

Let $\sigma_0$ be the stopping time defined by
\[ \sigma_0 = \inf\{ s\geq 0 \; : \; H(u(s)) \leq 0\} .\] 
Then the Cauchy-Schwarz inequality implies 
\begin{align} \label{square-int-mart}
 \EE\big( \big|   {\rm Im}\big(  I(\tilde{\sigma}_\gamma \wedge \sigma_0 \wedge \tilde{\tau}_k) \big) \big|^2 \big)
\leq \EE\Big( \int_0^{\tilde{\sigma}_\gamma\wedge \sigma_0 \wedge \tilde{\tau}_k} \|f^1_\phi\|_{L^\infty}  M(u_0) \, \|\nabla u(s)\|_{L^2}^2 ds\Big) <\infty. 
\end{align}
Therefore, for any $k>0$  the above stochastic integral is square integrable, hence, centered; thus, the Cauchy-Schwarz inequality yields
\[ M(u_0)^\alpha \EE\big(H(u(\tilde{\sigma}_\gamma\wedge  \sigma_0 \wedge \tilde{\tau}_k))\big) \leq \beta H(Q) M(Q)^\alpha + \frac{1}{2}
M(u_0)^{\alpha+1}  M_\phi \EE(\tilde{\sigma}_\gamma).
\]
For a fixed $k>0$ recall that $\tilde{\tau}_k = \inf\{ s\geq 0 : \|\nabla u(s)\|_{L^2} \geq k\}\wedge \tau^*(u_0)$. Then as $k\to \infty$, we get $\tilde{\tau}_k\to \tau^*(u_0)$ a.s.
By a.s. continuity  of $H(u(\cdot))$ we have  $M(u_0)^\alpha H(u(\tilde{\sigma}_\gamma \wedge \sigma_0 \wedge \tilde{\tau}_k))\to M(u_0)^\alpha H(u(\tilde{\sigma}_\gamma \wedge \sigma_0))$ 
as $k\to \infty $  and 
$$ 
M(u_0)^\alpha H(u(\tilde{\sigma}_\gamma \wedge \sigma_0)) \leq  \gamma H(Q) M(Q)^\alpha  1_{\{ \tilde{\sigma}_\gamma \leq  \sigma_0\} }
+ 0 \cdot M(Q)^\alpha \cdot 1_{\{ \sigma_0 < \tilde{\sigma}_\gamma\}} \qquad \mbox{\rm a.s.}.
$$
Furthermore,   
$ M(u_0)^\alpha H(u(\tilde{\sigma}_\gamma \wedge \sigma_0 \wedge \tilde{\tau}_k)) \in [0, \gamma\; H(Q) M(Q)^\alpha]$,  and hence, 
the dominated convergence theorem implies 
\begin{equation} 		\label{upper_1}
\gamma H(Q) M(Q)^\alpha\; P( \tilde{\sigma}_\gamma \leq  \sigma_0)  \leq \beta H(Q) M(Q)^\alpha + \frac{1}{2}
M(u_0)^{\alpha+1}  M_\phi \EE(\tilde{\sigma}_\gamma).
\end{equation} 

We next prove that when the strength of the noise is small enough,  $P( \tilde{\sigma}_\gamma \leq  \sigma_0) $ can be made as close to one as desired. 
Let $\epsilon \in (0,\frac{1}{2})$ be chosen such that $\gamma (1-2\epsilon) >\beta$. Since $\tilde{\sigma}_\gamma<\tau^*(u_0)$ a.s., we may choose $k_0$  and $T_0$ large enough 
to have $P(\tilde{\sigma}_\gamma
\geq  \tilde{\tau}_{k_0} \wedge T_0 ) \leq \epsilon$. 

Furthermore, on the set (of $\omega$) where
$\{\sigma_0\leq \tilde{\sigma}_\gamma \wedge \tilde{\tau}_{k_0}\wedge T_0\}$,  
the identity \eqref{Mu_0H(u)} implies  
$0\geq \beta H(Q) M(Q)^\alpha - M(u_0)^\alpha  {\rm Im}\big( I(\sigma_0)\big)$, so we have the following inclusion:
\begin{equation}		\label{inclusion}
 \{ \sigma_0 \leq \tilde{\sigma}_\gamma \wedge \tilde{\tau}_{k_0} \wedge T_0\} \subset \Big\{ \sup_{s\leq \tilde{\sigma}_\gamma \wedge \tilde{\tau}_{k_0}\wedge T_0}
 M(u_0)^\alpha  {\rm Im}\big(I(s)\big) \geq
\beta H(Q) M(Q)^\alpha  \Big\}.
\end{equation}
Then the Markov and Davis inequalities imply
\begin{align}		\label{estim-IM(u)}
P\Big(& \sup_{s\leq \tilde{\sigma}_\gamma \wedge \tilde{\tau}_{k_0}\wedge T_0}\; M(u_0)^\alpha {\rm Im}\big( I(s)\big) \Big)  \leq \frac{1}{\beta H(Q) M(Q)^\alpha}
 \EE\Big( \sup_{s\leq \tilde{\sigma}_\gamma \wedge \tilde{\tau}_{k_0}\wedge T_0}
M(u_0)^\alpha |I(s)| \Big)  \nonumber \\
\leq & \; \frac{3 M(u_0)^\alpha }{\beta H(Q) M(Q)^\alpha} \EE\Big(  \Big\{ \int_0^{\tilde{\sigma}_\gamma \wedge \tilde{\tau}_{k_0}\wedge T_0} \sum_l \Big( \int_{\RR^n} \bar{u}(s,x) \nabla u(s,x) \, . \nabla (\phi e_l)(x) dx\Big)^2 ds 
\Big\}^{\frac{1}{2}} \Big) \nonumber  \\
\leq &\; \frac{3 M(u_0)^\alpha }{\beta H(Q) M(Q)^\alpha} \EE\Big( \Big\{ \int_0^{\tilde{\sigma}_\gamma \wedge \tilde{\tau}_{k_0}\wedge T_0}  M(u_0) \|f^1_\phi\|_{L^\infty} 
 \Big(\int_{\RR^n} |\nabla u(s)|^2 dx\Big) ds  \Big\}^{\frac{1}{2}} \Big) \nonumber \\
 \leq &\; \frac{3 \sqrt{ k_0 T_0}  M(u_0)^{\alpha + \frac{1}{2}}}{\beta H(Q) M(Q)^\alpha}  \sqrt{M_\phi}, 
\end{align}
where in the last bound we used the definition of $\tilde{\tau}_{k_0}$ and in the one before the last one, the Cauchy-Schwarz inequality with respect to $dx$, 
and  the conservation of mass  identity $M(u(s))=M(u_0)$ a.s. 
Therefore, if $M_\phi$ is small enough, we have $P(\sigma_0 \leq \tilde{\sigma}_\gamma \wedge \tilde{\tau}_{k_0} \wedge T_0) \leq \epsilon$.
Finally,  we obtain the upper bound
\[ P( \sigma_0 < \tilde{\sigma}_\gamma ) \leq P(\sigma_0 \leq \tilde{\sigma}_\gamma  \wedge \tilde{\tau}_{k_0} \wedge T_0) + P(\tilde{\sigma}_\gamma \geq \tilde{\tau}_{k_0} \wedge T_0) \leq 2\epsilon.
\] 
For $M_\phi$ small enough,  the inequality \eqref{upper_1} and the Cauchy-Schwarz inequality imply 
\begin{align*}
 (1-2\epsilon) \gamma H(Q) M(Q)^\alpha \leq&  \;   \beta H(Q) M(Q)^\alpha + \frac{1}{2}
M(u_0)^{\alpha+1}  M_\phi \EE(\tilde{\sigma}_\gamma) \\
\leq &\; 
\beta H(Q) M(Q)^\alpha + \frac{1}{2}
M(u_0)^{\alpha+1}  M_\phi \big\{\EE\big( (\tilde{\sigma}_\gamma)^2\big) \big\}^{\frac{1}{2}}.
\end{align*} 
Hence, for $M_\phi$ small enough, we have 
\[ X = \Big\{ \EE\big( (T\wedge \bar{\tau}_0)^2\big) \big\}^{\frac{1}{2}} \geq \frac{2\lambda^2 (\gamma(1-2\epsilon) -\beta) H(Q) M(Q)^\alpha}{M(u_0)^{\alpha +1} 
M_\phi} > X_2,
\]
hence the condition \eqref{cond-3} is satisfied, which concludes the proof. 
\hfill $\Box$
\smallskip

We next consider {\it random} initial data and prove a result similar to Theorem \ref{th_blowup1} when $u_0$ is random with positive energy.

\begin{theorem}\label{th_blowup2}
Let $ \sigma \in \mathfrak R_{inter}$ 
and $\phi$ satisfy the conditions {\bf(H1)} and {\bf (H2)}.
Let $u_0 \in \Sigma $  a.s. be ${\mathcal F}_0$-measurable such that ~$\EE\big( M(u_0)^{2\alpha +2}\big) <\infty$, 
~$\EE \big( M(u_0)^\alpha V(u_0)\big) <\infty$ and ~ $\EE \big( G(u_0)^2 M(u_0)^{2\alpha}\big) <\infty$. 
Assume that $f^1_\phi$ is bounded and $\tau^*(u_0)$ is the stopping time as defined in Theorem \ref{local_wp}.

Suppose that for some constants $\beta$ and $\delta_0$ such that $0<\beta < 1 < \delta_0$, we have 
\begin{equation} \label{cond-u_0-Q-det-blowup}
H(u_0) M(u_0)^\alpha \leq  \beta H(Q) M(Q)^\alpha \; {\mbox{\rm a.s.} \quad \mbox{\rm  and}\quad 
\| \nabla u_0\|_{L^2} \|u_0\|_{L^2}^\alpha
\geq  \delta_0 \|\nabla Q\|_{L^2} \|Q\|_{L^2}^\alpha\; \mbox{\rm} a.s. .}
\end{equation}
Then for $T>0$ large enough and $M_\phi$ small enough, we have $P\big(\tau^*(u_0) \leq T\big) >0$. 
\end{theorem}
As in the case of a deterministic initial condition, the proof relies on the following technical result.

\begin{Prop}\label{Prop-blow-random}
Let $u_0$ and $\phi$ satisfy the assumptions of Theorem \ref{th_blowup2}, 
replacing $\delta_0$ by $\delta>1$ in the statement of \eqref{cond-u_0-Q-det}. Suppose that, for some $\epsilon >0$ and $N>0$, $t$ and $M_\phi$ satisfy the
following conditions
\begin{align}
    \EE\big(M(u_0)^{\alpha +1} \big) M_\phi t^3 \Big( \frac{4}{3} +18 \Big)  &\leq \epsilon,       \label{cond_1_Bis}\\
   \frac{32}{15} n  \sigma \EE\big( M(u_0)^{\alpha + \frac{1}{2}}\big) \sqrt{M_\phi} K t^{\frac{5}{2}} & \leq \epsilon.    \label{cond_2_Bis} 
\end{align}
Let $\tau_\delta$ and $\tilde{\tau}_N$ be defined by 
\eqref{def-tau_delta-multi}
and \eqref{def-tildetauK-multi}, respectively. Let $\tau = \tau_\delta \wedge \tilde{\tau}_N$, and suppose that
\begin{align}    \label{cond-3-Bis}
   \EE\big( V(u_0) M(u_0)^\alpha\big)& + 2\epsilon +4 \big\{ \EE\big( G(u_0)^2 M(u_0)^{2\alpha}\big) \big\}^{\frac{1}{2}} \big\{  \EE\big((t\wedge \tau)^2\big) \big\}^{\frac{1}{2}} \nonumber \\
  &  - 4\sigma s_c (\delta^2-\beta) \|\nabla Q\|_{L^2}^2 M(Q)^\alpha\EE\big( (t\wedge \tau)^2\big) <0. 
\end{align}
Then $P(\tau^*(u_0)\leq t)>0$. 
\end{Prop}

\begin{proof}
Suppose that $t<\tau^*(u_0)$ a.s.. Write $V(u)$ using \eqref{V(u)}, \eqref{G(u)} and \eqref{energy}. Multiplying by $M(u_0)^\alpha$ and taking 
expected values, we deduce an analog of \eqref{E:MEV-1b}, where the terms $T_i, i=1,2,3$, are defined as in the proof of Prop. \ref{Prop-blow-det}. We thus obtain
\begin{align}   \label{MEV-1Bis}
&\EE\big( M(u_0)^\alpha V(t\wedge \tau) \big) \leq \EE\big( M(u_0)^\alpha V(u_0) \big) + 4\EE\big( M(u_0) G(u_0) (t\wedge \tau)\big) \nonumber \\
&+ 4n\sigma \EE\big( M(u_0)^\alpha H(u_0) (t\wedge \tau)^2 \big) - 8\sigma s_c \,\EE\Big( M(u_0)^\alpha \int_0^{t\wedge \tau}\!\! 
ds \int_0^s \! \|\nabla u(r)\|_{L^2}^2 dr\Big) + \sum_{i=1}^3 T_i.
\end{align} 
Once more, for $r\leq \tau_\delta$ we have a.s. 
\[ M(u_0)^\alpha \|\nabla u(r)\|_{L^2}^2 \geq \delta^2 M(Q)^\alpha \|\nabla Q\|_{L^2}^2 \quad \mbox{\rm and} \quad 
M(u_0)^\alpha H(u_0) \leq \beta M(Q)^\alpha H(Q).
\]
Hence, a computation similar to the one proving \eqref{compensation} yields
\begin{align}   \label{compensation-Bis}
4 n\sigma \EE\big( M(u_0)^\alpha H(u_0) &(t\wedge \tau)^2 \big)- 8 \sigma s_c \EE\Big( M(u_0)^\alpha \int_0^{t\wedge \tau} ds \int_0^s \|\nabla u(r)\|_{L^2}^2
dr \Big)  \nonumber \\
&\leq -4\sigma s_c (\delta^2-\beta) \|\nabla Q\|_{L^2}^2 M(Q)^\alpha \EE\big( (t\wedge \tau)^2\big) .
\end{align}
The upper bounds \eqref{MEV-1Bis}, \eqref{compensation-Bis} and  the Cauchy-Schwarz inequality  imply 
\begin{align}       \label{MEV-2Bis}
 \EE\big( M(u_0)^\alpha V(u(t\wedge \tau))\big) \leq &\; \EE\big( M(u_0)^\alpha V(u_0)\big) 
 + 4\Big\{ \EE\big( M(u_0)^{2\alpha} G(u_0)^2 \big) \big\}^{\frac{1}{2}} \big\{ \EE \big( (t\wedge \tau)^2\big) \big\}^{\frac{1}{2}} \nonumber \\
& -4 \sigma s_c (\delta^2-\beta) \|\nabla Q\|_{L^2}^2 M(Q)^\alpha \EE\big( (t\wedge \tau)^2\big) + \sum_{i=1}^3 T_i.
 \end{align}
We next upper estimate the terms $T_i$, $i=1,2,3$.
A straightforward extension of \eqref{estim_T1} implies
\begin{equation}        \label{estim_T1-Bis}
    T_1 \leq \frac{4}{3} \EE\big( M(u_0)^{\alpha +1} \big) M_\phi t^3. 
\end{equation}
Furthermore, similarly to proving \eqref{estim_T2}, we bound
\begin{align}       \label{estim-T2Bis}
|T_2| \leq& \;  8n\sigma \int_0^t ds \int_0^s dr \,\EE\Big( \Big| \int_0^{r\wedge \tau} M(u_0)^\alpha \sum_{l\geq 0}  \int_{\RR^n} 
\bar{u}(r_1,x) \nabla u(r_1,x)\!\cdot\! \nabla(\phi e_l)(x) dx d\beta_l(r_1)\Big| \Big)  \nonumber \\
\leq & \; 8n\sigma \EE\big( M(u_0)^{\alpha + \frac{1}{2}} \big) \sqrt{M_\phi} N \frac{4}{15} t^{\frac{5}{2}}.
\end{align}
For the upper bound $|T_3|$, we use the Davis and the Cauchy-Schwarz inequalities. This yields for $\tilde{\epsilon}>0$
\begin{align*}
    |T_3|\leq & \; 12 \int_0^t ds \, \EE\Big( \Big\{ \int_0^{s\wedge \tau} M(u_0)^{2\alpha} \sum_{l\geq 0}  \Big| \int_{\RR^n} |u(r,x) |^2 \, x\!\cdot\! \nabla (\phi e_l)(x)
    dx\Big|^2 dr \Big\}^{\frac{1}{2}} \Big) \\ 
    \leq & \; 12 \int_0^t  ds \, \EE\Big( \Big\{ \int_0^{s\wedge \tau} M(u_0)^{2\alpha +1} M_\phi V(u(r)) dr \Big\}^{\frac{1}{2}} \Big)\\ 
    \leq &\;  { 12 } \Big\{ \int_0^t ds \, \EE \Big( \int_0^s M(u_0)^\alpha V(u(r\wedge \tau)) dr \Big) \Big\}^{\frac{1}{2}} \Big\{ \int_0^t ds \, \EE\Big( \int_0^s
    M(u_0)^{\alpha +1} M_\phi dr \Big) \Big\}^{\frac{1}{2}} \\ 
    \leq & \;  \tilde{\epsilon} \int_0^t ds \int_0^s \EE\big(M(u_0)^\alpha V(u(r\wedge \tau))\big) dr  + \frac{36}{\tilde{\epsilon}} \int_0^t ds\int_0^s
    M_\phi \EE\big(M(u_0)^{\alpha +1}\big) dr,
\end{align*}
where the last upper estimate  is deduced from Young's inequality. 
Choosing $\tilde{\epsilon}$ such that $\tilde{\epsilon} t =1$, we obtain
\begin{equation}        \label{estim-T3-random1}
    |T_3| \leq \int_0^t \EE\big(M(u_0)^\alpha  V(s\wedge \tau)\big)\, ds + 18 t^3 M_\phi \EE\big( M(u_0)^{\alpha +1}\big).
\end{equation}

Collecting the upper estimates \eqref{MEV-2Bis}-\eqref{estim-T3-random1}  and using the conditions \eqref{cond_1_Bis} and \eqref{cond_2_Bis}, we deduce
\begin{align*} 
\EE\big( M(u_0)^\alpha& V(u(t\wedge \tau) \big)  \leq  \EE\big( M(u_0)^\alpha V(u_0)\big) + 2\epsilon + 4 \Big\{ \EE(M(u_0)°^{2\alpha} G(u_0)^2 \big) 
\big\}^{\frac{1}{2}}{ \big\{ \EE\big( (t\wedge \tau)^2\big) \big\}^{\frac{1}{2}} }\\
&\; -4 \sigma s_c (\delta^2-\beta) \|\nabla Q\|_{L^2}^2 M(Q)^\alpha {  \EE\big( (t\wedge \tau)^2 \big) }+ \int_0^t { \EE\big(M(u_0)^\alpha  V(u(s\wedge \tau)\big) } \, ds.
\end{align*}
Recall that since $t<{\tau^*(u_0) } $ a.s. and $\tau \leq \tilde{\tau}_N$, once more the upper bound (6.2) in \cite{deB_Deb_AnnProb} implies
$$ 
\sup_{s\leq t} V(s\wedge \tau) \leq [4 N^2 t + V(u_0)] e^t \quad {\rm a.s.}.
$$ 
Since $\EE\big( V(u_0) M(u_0)^\alpha\big) <\infty$, we deduce that $\sup_{s\leq t} \EE\big( M(u_0)^\alpha  V(u(s\wedge \tau)) <\infty$. 
Therefore, 
\begin{align*}
  \EE\big( M(u_0)^\alpha V(u(t\wedge \tau) \big)  \leq & \big[ \EE\big( M(u_0)^\alpha V(u_0)\big) + 2\epsilon + 4 \Big\{ \EE(M(u_0°^{2\alpha} G(u_0)^2 \big) 
\big\}^{\frac{1}{2}} \big\{ \EE\big( t\wedge \tau\big) \big\}^{\frac{1}{2}}\\  
&\quad -4 \sigma s_c (\delta^2-\beta) \|\nabla Q\|_{L^2}^2 M(Q)^\alpha \EE\big( t\wedge \tau\big) \big] e^t. 
\end{align*}

The condition \eqref{cond-3-Bis} yields $\EE\big( M(u_0)^\alpha V(t\wedge \tau)\big) <0$, which brings a contradiction. Therefore, $P( \tau^*(u_0)\leq T) >0$, which completes the proof.
\end{proof}

We now prove Theorem \ref{th_blowup2}.
\smallskip

{\it Proof of Theorem \ref{th_blowup2}.}  Once more we suppose that $\tau^*(u_0)=\infty$ a.s. and look for a contradiction.
As in the proof of Theorem \ref{th_blowup1}, for $\delta_0$ defined by \eqref{cond-u_0-Q-det-blowup} choose $\delta\in (1, \delta_0)$ and $\gamma \in (\beta,1)$ such that \eqref{gamma-delta} holds. For
$\tilde{\sigma}_\gamma = \inf\{ s\geq 0:  H(u(s)) M(u_0)^\alpha \geq \gamma H(Q) M(Q)^\alpha\}$ 
we have $\tilde{\sigma}_\gamma \leq \tau_\delta$.
Recall that  $\tau_\delta$ (resp. $\tilde{\tau}_{N_0}$) is defined by \eqref{def-tau_delta-multi} (resp. \eqref{def-tildetauK-multi}). Note that as $N \to \infty$, we have $\tau_N \to \infty$ a.s., since $\tau^\ast(u_0) = \infty$ a.s.  Setting $\bar{\tau}_0 = \tau_\delta \wedge \tilde{\tau}_{N_0}$, where 
$N_0$ and $T$ are chosen large enough such that
we can make sure that $\bar{\tau}_0 \wedge T$ is as close to $\tau_\delta$ as needed and 
for some fixed $\lambda \in (0,1)$ we also have 
\smallskip

$\bullet$ given some fixed (large) $M_0>0$, one has $\EE(\big(T\wedge \bar{\tau}_0)^2\big)  \geq \lambda^2 M_0^2$ if $\EE(\tau_\delta^2)=\infty$;

$\bullet$ $\EE\big(  (T\wedge \bar{\tau}_0)^2\big) \geq \lambda^2 \EE(\tau_\delta^2)$, if $\EE(\tau^2_\delta)<\infty$.

For this choice of $T$  and $N_0$ 
given $\epsilon >0$ choose $M_\phi$ small enough to ensure that the conditions 
\eqref{cond_1_Bis} and \eqref{cond_2_Bis} are satisfied. 
Set 
\[ a=4\sigma s_c (\delta^2-\beta) \|\nabla Q\|_{L^2}^2 M(Q)^\alpha, \quad b=4\Big\{ \EE\big( G(u_0)^2 M(u_0)^{2\alpha}\big) \big\}^{\frac{1}{2}}
\quad \mbox{\rm and } ~~c=\EE\big( V(u_0) M(u_0)^\alpha\big) +2\epsilon.
\]

Then, if we set $X=\big\{\EE\big( (T\wedge \bar{\tau}_0)^2\big)\big\}^{\frac{1}{2}}$, Proposition \ref{Prop-blow-random}
implies that if $-aX^2+bX+c<0$, we have {$P\big( \tau^*(u_0)\leq T\big) >0$.}

Let $X_1<X_2$ denote the roots of the polynomial $-aX^2+bX+c$. 
\smallskip

$\bullet$ If $\EE(\tau_\delta^2)=\infty$, we may choose $N$ large enough and $\lambda $ close to one to have $N\lambda >X_2$, which implies $X>X_2$. 

$\bullet$ If $\EE(\tau_\delta^2)<\infty$, then $\tilde{\sigma}_\gamma \leq  \tau_\delta <\infty$ a.s., which implies $H(u(\tilde{\sigma}_\gamma))  M(u_0)^\alpha
= \gamma H(Q) M(Q)^\alpha$ by a.s. continuity of the energy $H(u(\cdot))$. 
Multiplying \eqref{energy} by $M(u_0)^\alpha$, we deduce 
that the identity \eqref{Mu_0H(u)} holds for every stopping time $\tau<\tau^*(u_0)$. Furthermore, the inequality \eqref{square-int-mart} holds, so that the martingale in the right-hand side of \eqref{Mu_0H(u)} is centered.  

As in the proof of Theorem \ref{th_blowup1}, let $\sigma_0=\inf\{ s\geq 0 : \; H(u(s)) \leq 0\}$ and choose $\epsilon \in (0, \frac{1}{2})$ such that $\gamma (1-2\epsilon) > \beta$. 
Taking expected values in \eqref{Mu_0H(u)}, we obtain 
\begin{align*}
 \EE\big( M(u_0)^\alpha{H(u(\tilde{\sigma}_\gamma\wedge \tilde{\tau}_k  \wedge \sigma_0) } \big)  
 \leq &\beta H(Q) M(Q)^\alpha + \frac{1}{2} M_\phi \EE\big( M(u_0)^{\alpha +1} (\tilde{\sigma}_\gamma   \wedge \tilde{\tau}_k \wedge \sigma_0) \big) \\
 \leq & \beta H(Q) M(Q)^\alpha + \frac{1}{2} M_\phi \big\{ \EE\big( M(u_0)^{2\alpha +2}\big) \big\}^{\frac{1}{2}} 
 \big\{ \EE\big( \tilde{\sigma}_\gamma^2 \big) \big\}^{\frac{1}{2}}.
\end{align*}
Since as $k\to \infty$, the sequence $M(u_0)^\alpha H(u(\tilde{\sigma}_\gamma \wedge \tilde{\tau}_k  \wedge \sigma_0)) $ converges a.s. to 
$ \gamma M(Q)^\alpha H(Q) 1_{\{ \tilde{\sigma}_\gamma \leq \sigma_0\}}$ and belongs a.s. to the interval $[0, \gamma H(Q) M(Q)^\alpha]$, the integrability assumption made on $M(u_0)$
implies that one can apply the dominated convergence theorem to deduce 
\[ \gamma H(Q)  M(Q)^\alpha  P(\tilde{\sigma}_\gamma \leq \sigma_0) \leq  \beta H(Q) M(Q)  +  \frac{1}{2} M_\phi \big\{ \EE\big( M(u_0)^{2\alpha +2}\big) \big\}^{\frac{1}{2}} 
 \big\{ \EE\big( \tilde{\sigma}_\gamma^2 \big) \big\}^{\frac{1}{2}}. \]
 As in the proof of Theorem \ref{th_blowup1} we next prove that for $M_\phi$ small enough, we have $P(\sigma_0<\tilde{\sigma}_\gamma)<\epsilon$. 
First, for $k_0$ and $T_0$ defined above,  using \eqref{inclusion}  and the fact that $M(u_0)$ is ${\mathcal F}_0$-measurable,  a small modification of the  proof of  \eqref{estim-IM(u)} implies 
\begin{align*}
P(\sigma_0 &< \tilde{\sigma}_\gamma \wedge \tilde{\tau}_{k_0} \wedge T_0) \leq 
P\Big( \sup_{s\leq \tilde{\sigma}_\gamma \wedge \tilde{\tau}_{k_0}\wedge T_0}\; M(u_0)^\alpha {\rm Im}\big( I(s)\big) \Big) \\
&  \leq \frac{1}{\beta H(Q) M(Q)^\alpha}
 \EE\Big( \sup_{s\leq \tilde{\sigma}_\gamma \wedge \tilde{\tau}_{k_0}\wedge T_0}
M(u_0)^\alpha |I(s)| \Big) \\
\leq & \; \frac{3 }{\beta H(Q) M(Q)^\alpha} \EE\Big( M(u_0)^\alpha  \Big\{ \int_0^{\tilde{\sigma}_\gamma \wedge \tilde{\tau}_{k_0}\wedge T_0} 
\sum_l \Big( \int_{\RR^n} \bar{u}(s,x) \nabla u(s,x) \, . \nabla (\phi e_l)(x) dx\Big)^2 ds 
\Big\}^{\frac{1}{2}} \Big) \\
\leq &\; \frac{3 \sqrt{ k_0 T_0}  \EE\big(M(u_0)^{\alpha + \frac{1}{2}}\big) }{\beta H(Q) M(Q)^\alpha}  \sqrt{M_\phi}, 
\end{align*}
so that for $M_\phi$ small enough, we deduce $P(\sigma_0 < \tilde{\sigma}_\gamma \wedge \tilde{\tau}_{k_0} \wedge T_0)<\epsilon$, and then $P(\tilde{\sigma}_\gamma < \sigma_0) \leq 2\epsilon$.
 Since $\big\{ \EE\big( (T\wedge \bar{\tau}_0^2 \big) \big\}^{\frac{1}{2}} \geq \lambda^2 \EE\big(\tilde{\sigma}_\gamma^2\big) $, if $M_\phi$
is small enough to ensure
\[ 
\frac{2\lambda^2 (\gamma  (1-2\epsilon) - \beta) 
 H(Q)  M(Q)^\alpha}{M_\phi \big\{ \EE\big( M(u_0)^{2\alpha +2}\big) \big\}^\frac{1}{2}} >X_2,
\]
then condition \eqref{cond-3-Bis} is satisfied, and $P\big( \tau^*(u_0)\leq T\big) >0$, concluding the proof.
\hfill $\Box$


\subsection{Additive stochastic perturbation} \label{S:A-bup}
In this section we study the solution to \eqref{NLS_add} in the intercritical range with 
$\sigma \in \big(\frac{2}{n}, \frac{2}{n-2}\big)$. 
For an additive stochastic perturbation, we make an additional assumption on the complex-valued noise. 
\medskip

\noindent$\bullet$ {\bf Condition (H4)} ({\it Additional restriction on the additive noise.}) 
\begin{itemize}
\item[(i)] 
The operator $\phi$ is Hilbert-Schmidt from $L^2(\RR^n)$ to $H^1(\RR^n)$ and from $L^2(\RR^n)$ to $\Sigma$, that is,
\begin{equation}\label{C-phi-sigma}
C_\phi^\Sigma := \sum_{k\geq 0}  \| x \, \phi e_k\|_{L^2(\RR^n)}^2 <\infty.
\end{equation} 
\item[(ii)] 
The operator $\phi$ is $\gamma$-Radonifying from $L^2(\RR^n)$ to $L^{2\sigma+2}(\RR^n)$, that is, 
\begin{equation}\label{C-phi}
C(\phi):=\sum_{k\geq 0}  \|\phi e_k\|_{L^{2\sigma +2}(\RR^n)}^2<\infty.
\end{equation}
\end{itemize}

We will also use the constants $C^{(1)}_\phi$ and $C^{(2)}_\phi$ defined by 
$$
C^{(1)}_\phi: = \sum_{k\geq 0} \|\nabla(\phi e_k)\|_{L^2}^2 \leq \| \phi\|_{L^{0,1}_2}^2  \quad \mbox{\rm and} \quad C^{(2)}_\phi: = {\rm Im }\sum_{k\geq 0}  \int_{\RR^n} \phi e_k(x)\  x\!\cdot\!\nabla(\overline{\phi e_k})(x) dx.
$$
Note that the Cauchy-Schwarz inequality implies that under 
the  condition {\bf (H4)} on $\phi$, we have 
$$
C^{(2)}_\phi \leq \sqrt{ C^{(1)}_\phi C^\Sigma_\phi} \leq  \|\phi\|_{L^{0,1}_2(\RR^n)} \sqrt{ C^\Sigma_\phi}.
$$

The following lemma is an adaptation of 
\cite[Lemma~2.2]{deB_Deb_PTRF} to our setting. 
\begin{lemma}   \label{lem_v-G-add}
Let $\phi$ satisfy the condition {\bf (H4)},  $u$ be the solution to \eqref{NLS_add} and assume that $u_0\in \Sigma$ a.s., and $u$ belongs to $L^\infty((0,\tau ); \Sigma)$ a.s.. 
Then for any stopping time $\tau<\tau^*(u_0)$ a.s., we have a.s. 
\begin{equation}    \label{V(u)-add-1}
V(u(\tau)) = V(u_0) +4 \int_0^\tau G(u(s)) ds +2 {\rm Im} \int_0^\tau \sum_{k\geq 0} \int_{\RR^n} |x|^2 \overline{{u}(s,x)} \phi e_k(x) dx d\beta_k(s) + \tau C_\phi^\Sigma, 
\end{equation}
and
\begin{align} \label{G(u)-add}
G(u(\tau)) = & G(u_0) + 2n\sigma \int_0^\tau H(u(s)) ds - 2\sigma s_c \int_0^\tau \|\nabla u(s)\|_{L^2}^2 ds \nonumber \\
&+ {\rm Re}\int_0^\tau  \sum_{l\geq 0}  \int_{\RR^n} |u(s,x)|^2 \big[ 2\ x\!\cdot\!\nabla(\phi e_l)(x) +n \phi e_l(x)\big] dx d\beta_l(s) + \tau C_\phi^\Sigma. 
\end{align}
\end{lemma}
\begin{proof} The equation \eqref{V(u)-add-1} is reformulation of (2.3) in \cite{deB_Deb_PTRF}.

To prove \eqref{G(u)-add}, we start with writing  $G(u(\tau))$ using the identity (2.4)  in \cite{deB_Deb_PTRF} page 84 and \eqref{M-H}, which yields 
\begin{align*}
G\big( u(\tau)\big) = &\;  G(u_0) + 4\int_0^\tau H\big(u(s)\big) ds + \frac{2-n\sigma}{\sigma+1} \int_0^\tau \int_{\RR^n} |u(s,x)|^{2\sigma +2} dx ds \\
&\;  + {\rm Re} \int_0^\tau\sum_{k\geq 0}  \int_{\RR^n}  \overline{{u}(s,x)} \big[ 2 \, x\!\cdot\! \nabla(\phi e_k)(x) + n\phi e_k(x)\big] dx d\beta_k(s) + \tau C^{(2)}_\phi \\
= &\; G(u_0) +2n\sigma \int_0^\tau H\big(u(s)\big) ds  + (2-n\sigma) \int_0^\tau \|\nabla u(s)\|_{L^2}^2 ds \\
&\;  + {\rm Re} \int_0^\tau\sum_{k\geq 0}  \int_{\RR^n}  \overline{{u}(s,x) } \big[ 2 \, x\!\cdot\!\nabla(\phi e_k)(x) + n\phi e_k(x)\big] dx d\beta_k(s) + \tau C^{(2)}_\phi .
\end{align*}
Recalling that $2-n\sigma = -2\sigma s_c$, concludes the proof of \eqref{G(u)-add}.
\end{proof} 

Let $u_0, \sigma$ and $\phi$ be as in Lemma~\ref{lem_v-G-add}. Then plugging the identities \eqref{H_add} and \eqref{G(u)-add} in \eqref{V(u)-add-1}, we deduce 
that for any stopping time $\tau<\tau^*(u_0)$ a.s., we have a.s.  
\begin{align}		\label{V(u)-add-2}
V\big( u(\tau)\big) = & \; V(u_0) + 4\tau G(u_0) + 4n\sigma H(u_0) \tau^2 - 8 \sigma s_c \int_0^\tau ds \int_0^s \| \nabla u(r)\|_{L^2}^2 dr \nonumber \\
& \; + 2 C_\phi^{(2)} \tau^2 + \frac{4}{3} n\sigma  C^{(1)}_\phi  \tau^3 +  C_\phi^{\Sigma} \tau + \sum_{j=1}^4 T_j(\tau),
\end{align}
where
\begin{align*}
T_1(\tau) = 
&\;   8n\sigma\! \! \int_0^\tau \!\! ds\!\!  \int_0^s \!\! dr \!\! \int_0^r \! \sum_{k\geq 0} {\rm Im} \int_{\RR^n} \!\! 
 \big[ \nabla \overline{{u}(r_1,x)} \nabla (\phi e_k)(x) - |u(r_1,x)|^{2\sigma} \overline{{u}(r_1,x)} \phi e_k(x) \big] dx d\beta_k(r_1) , \\
T_2(\tau) =&\; - 2n\sigma \!  \int_0^\tau \!\! ds\!  \int_0^s \!\! dr \! \int_0^r \!\! dr_1 \sum_{k\geq 0} \int_{\RR^n} \big[ 
|u(r_1,x)|^{2\sigma} |\phi e_k(x)|^2  \\
&\qquad \qquad \qquad \qquad  
+ 2\sigma |u(r_1,x)|^{2\sigma -2} \big( {\rm Re} (\overline{{u}(r_1,x)} \phi e_k(x)) \big)^2\big]  dx , \\
T_3(\tau) =&\; +4 {\rm Re} \int_0^\tau ds \int_0^s \sum_{k\geq 0}  \int_{\RR^n}\big[ \, \overline{{u}(r,x)} \big( 2 \,x\!\cdot\! \nabla(\phi e_k)(x) +n \phi e_k(x) \big) \big]  dx d\beta_k(r),  \\
T_4(\tau) =&\; +2 {\rm Im} \int_0^\tau \sum_{k\geq 0}  \int_{\RR^n} |x|^2 \overline{{u}(s,x) }\phi e_k(x) dx d\beta_k(s). 
\end{align*} 

We extend the identity \eqref{E_sup_M} to higher moments of the mass. 
\begin{lemma}		\label{lem-sup-Mp}
Let  $u_0$ be ${\mathcal F}_0$-measurable, $\phi$ be Hilbert-Schmidt 
and $p\in [2,\infty)$ be such that $ \EE\big( M(u_0)^p\big) <\infty$. Then for any $\tilde{\epsilon}\in (0,1)$ 
and any  $t\in (0,\tau^*(u_0))$ a.s.,  we have
\begin{equation}		\label{E_sup_Mp}	
\EE\Big( \sup_{s\leq  t} [M\big(u(s)\big)]^p\Big) \leq {\mathcal C}(p,u_0,t,\phi, \tilde{\epsilon}),
\end{equation}
where
\begin{equation} \label{def_tildeC}
{\mathcal C}(p,u_0,t,\phi, \tilde{\epsilon}) = 
 \frac{1}{1-\tilde{\epsilon}} \Big[ \EE\big([M(u_0)]^p\big)  + \| \phi\|_{L^{0,0}_2}^2 p \Big( 2p-1+\frac{9p}{\tilde{\epsilon}}\Big)t\Big]
e^{  \| \phi\|_{L^{0,0}_2}^2 p \Big( 2p-1+\frac{9p}{\tilde{\epsilon}}\Big)t }.
\end{equation}
\end{lemma} 

\begin{proof}
Recall that by (2.1) in \cite{deB_Deb_PTRF} we have for $t<\tau^*(u_0)$ 
\[ M\big(u(t)\big) = M(u_0) -2{\rm Im} \int_0^t \sum_{k\geq 0}  \int_{\RR^n} u(s,x) \overline{\phi e_k}(x) dx \beta_k(s) + \|\phi\|_{L^{0,0}_2}^2 t. \]
It\^o's formula implies that for $p\in [2,\infty)$,  we have
\begin{align*}
[M\big( u(t) \big)]^p \leq  & \; [M(u_0)]^p -2p\int_0^t [M\big(u(s)\big)]^{p-1} {\rm Im} \sum_{k\geq 0} \int_{\RR^n} u(s,x) \overline{\phi e_k}(x) dx \beta_k(s)  \\
& \; + p\int_0^t [M\big( u(s)\big)]^{p-1} \|\phi\|_{L^{0,0}_2}^2 ds  \\
 & + \frac{p(p-1)}{2} \int_0^t  4 [M\big( u(s)\big)]^{p-2} \sum_{k\in \NN} \Big( \int_{\RR^n} |u(s,x)| | \phi e_k(x)| dx \Big)^2  ds \\
\leq & \;   [M(u_0)]^p  + (2p^2-p)  \|\phi\|_{L^{0,0}_2}^2  
\int_0^t [M\big( u(s)\big)]^{p} ds   + (2p^2-p)  \|\phi\|_{L^{0,0}_2}^2  t  \\
& \;  -2p \int_0^t [M(u(s))]^{p-1} {\rm Im} \sum_{k\geq 0} \int_{\RR^n} u(s,x) \overline{\phi e_k}(x) \, dx \beta_k(s) . 
 \end{align*}
The Davis inequality and the Cauchy-Schwarz and Young inequalities imply for any $\tilde{\epsilon} \in (0,1)$, 
\begin{align*}
\EE\Big( \sup_{s\leq t} \big[M\big(u(s)\big)\big]^p \Big) \leq &\;  \EE\big([M(u_0)]^p\big) +  (2p^2-p)  \|\phi\|_{L^{0,0}_2}^2  t 
+ (2p^2-p)  \|\phi\|_{L^{0,0}_2}^2 \! \int_0^t \!  \EE\Big( \sup_{r\leq s} \big[M\big( u(r)\big)\big]^p \Big) ds  \\
& \; + 6p \EE\Big( \Big\{ \int_0^t \big[M\big(u(s)\big)\big]^{2p-2} \sum_{k\geq 0}  \Big( \int_{\RR^n} |u(s,x)| |\phi e_k(x)| dx\Big)^2 ds \Big\}^{\frac{1}{2}} \Big) 
\\
 \leq &\; \EE([M(u_0)]^p) +  (2p^2-p)  \|\phi\|_{L^{0,0}_2}^2  t  + (2p^2-p)  \|\phi\|_{L^{0,0}_2}^2 \! \int_0^t \!  \EE\Big( \sup_{r\leq s} \big[M\big( u(r)\big)\big]^p \Big) ds  \\
&+6p \|\phi\|_{L^{0,0}_2} \EE\Big( \sup_{s\leq t} \big[M\big( u(s)\big)\big]^{\frac{p}{2}} \Big\{ \int_0^t \big[M\big( u(s)\big)\big]^{p-1} ds \Big\}^{\frac{1}{2}} \Big)\\
\leq & \; \EE([M(u_0)]^p) +  (2p^2-p)  \|\phi\|_{L^{0,0}_2}^2  t  + (2p^2-p)  \|\phi\|_{L^{0,0}_2}^2 \! \int_0^t \!  \EE\Big( \sup_{r\leq s} \big[M\big( u(r)\big)\big]^p \Big) ds  \\
&+ \tilde{\epsilon} \, \EE\Big(  \sup_{s\leq t} \big[M\big( u(s)\big)\big]^p\Big)  + \Big[ 2p^2-p+\frac{9p^2}{\tilde{\epsilon}}\Big]    \|\phi\|_{L^{0,0}_2}^2   t  \\
& +  \Big[ 2p^2-p+\frac{9p^2}{\tilde{\epsilon}}\Big]    \|\phi\|_{L^{0,0}_2}^2   \int_0^t \EE\Big( \sup_{r\leq s} \big[M\big( u(r)\big)\big]^p\Big)  ds.
\end{align*}
Given $N>0$, let $\tau_N = \inf\{ t\geq 0 : M\big(u(t)\big) \geq N\}$ and rewrite the above inequality with $t\wedge \tau_N $ instead of $t$. Since 
$\int_0^{t\wedge \tau_N} \sup_{r\leq s} \big[M\big( u(r)\big)\big]^p ds \leq \int_0^{t} \sup_{r\leq s\wedge \tau_N} \big[M\big( u(r)\big)\big]^p ds $, 
Gronwall's lemma implies that
\begin{align*}
\EE \Big( \sup_{s\leq  t\wedge \tau_N}  [M\big( u(s)\big)]^p \Big) \leq 
& \frac{1}{1-\tilde{\epsilon}} \Big[  \EE\big( [M(u_0)]^p \big)+ \Big( 2p^2-p+   \frac{9p^2}{\tilde{\epsilon}}  
 \Big)    \|\phi\|_{L^{0,0}_2}^2  \EE(t\wedge \tau_N)  \Big]    \\
& \times 
\exp\Big( t \Big( 2p^2-p+   \frac{9p^2}{\tilde{\epsilon}}  
 \Big)    \|\phi\|_{L^{0,0}_2}^2   \Big).
\end{align*}
Since $\EE\big( [M(u_0)]^p\big) <\infty$, as $N\to \infty$  the sequence $\sup_{s\leq t\wedge \tau_N} M\big( u(s)\big)$ increases to  
$\sup_{s\leq t} M\big( u(s)\big)$, and the monotone convergence theorem concludes the proof. 
\end{proof}

The following lemma provides an upper bound of the expected value of the
maximum of the mass slightly different from \eqref{E_sup_M}. \begin{lemma}       \label{lem_sup_M-M0}
Let  $u_0$ be ${\mathcal F}_0$-measurable {such that $\EE(M(u_0))<\infty$}, $\phi$ be Hilbert-Schmidt from $L^2(\RR^n)$ to $L^2(\RR^n)$. 
Then for any $\tilde{\epsilon}\in (0,1)$ 
and any  $t\in (0,\tau^*(u_0))$ a.s.,  we have
\begin{equation}	  \label{E_sup_M-M0}
\EE\Big( \sup_{s\leq t}  \big[ M(u(s))-M(u_0)\big] \Big) \leq \frac{\tilde{\epsilon}}{1-\tilde{\epsilon}} \, \EE\big( M(u_0)\big) 
+\ { \frac{1}{1-\tilde{\epsilon}} \Big( \frac{9}{\tilde{\epsilon}} +1\Big) } t \|\phi\|_{L^{0,0}_2}^2.
\end{equation}
Furthermore, if $u_0\in L^2$ is deterministic, then for any $t\in (0, \tau^*(u_0))$ a.s. and every $
\mu >1$, given $\epsilon>0$ there exists $\delta_0$ such that
\begin{equation}    \label{upp_P_OmegaMc}
\| \phi\|_{L^{0,0}_2} < \delta_0 \quad \mbox{\rm implies } \quad P\Big(  \sup_{s\leq t } M(u(s)) \geq  \mu M(u_0) \Big) \leq \epsilon.    
\end{equation}
\end{lemma}
\begin{proof}
Rewriting \eqref{M_add} and using the Davis inequality, we deduce
\begin{align*}
 \EE\Big( \sup_{s\leq t} \big[ M(u(s))-M(u_0)\big]   \Big) \leq &\;  6 \EE\Big( \Big\{ \int_0^t \| \phi^* u(s)\|_{L^2}^2 ds \Big\}^{\frac{1}{2}} \Big) 
+  t \|\phi\|_{L^{0,0}_2}^2 
 \\
 \leq &\;  6 \EE\Big( \Big\{ \int_0^t \| u(s)\|_{L^2}^2  \| \phi\|_{L^{0,0}_2}^2 
 ds \Big\}^{\frac{1}{2}} \Big)
+ t \|\phi\|_{L^{0,0}_2}^2 
 \\
 \leq &\; 6 \EE\Big( \sup_{s\leq t} \|(u(s))\|_{L^2} \sqrt{t} \|\phi\|_{L^{0,0}_2}\Big) +  t \|\phi\|_{L^{0,0}_2}^2   \\
\leq &\; \tilde{\epsilon}\, \EE \Big(  \sup_{s\leq t} [ M(u(s)-M(u_0) ] \Big) + \tilde{\epsilon}  \, \EE(M(u_0)) 
+ \frac{9}{\tilde{\epsilon}} t \|\phi\|_{L^{0,0}_2}^2
+ t \|\phi\|_{L^{0,0}_2}^2 .
\end{align*}
Since $\EE\Big( \sup_{s\leq t} [M(u(s))-M(u_0)]\Big)<\infty$  by \eqref{E_sup_M}, we deduce \eqref{E_sup_M-M0}.
  
Let $u_0$ be deterministic. For $\mu>1$ and $t\in (0, \tau^*(u_0))$ the Markov inequality implies
  \[ P\Big( \sup_{s\leq t } [M(u(s))-M(u_0)] \geq (\mu-1) M(u_0) \Big) \leq \frac{\tilde{\epsilon}}{(1-\tilde{\epsilon})(\mu-1)} + 
{\frac{1}{ (1-\tilde{\epsilon})  (\mu-1)} \Big( \frac{9}{\tilde{\epsilon}} +1\Big)} t \|\phi\|_{L^{0,0}_2}^2. 
  \]
Fix $\epsilon >0$ and choose $\tilde{\epsilon}>0$ such that $\frac{\tilde{\epsilon}}{(1-\tilde{\epsilon})(\mu-1)}< \epsilon/2$. Then choose
$\delta_0 > 0$ such that  {$\frac{1}{ (1-\tilde{\epsilon})(\mu-1) } \Big( \frac{9}{\tilde{\epsilon}} +1\Big) t \delta_0^2 < \epsilon/2$.} This completes the
 proof of \eqref{upp_P_OmegaMc}. 
\end{proof}

Let $\sigma \in \big( \frac{2}{n}, \frac{2}{(n-2)^+}\big)$ and $\alpha$ as in \eqref{E:alpha}. 
The following result describes a sufficient condition on the initial condition and on some deterministic positive time to have blow-up before that time with positive probability. 
To highlight the main idea in the additive case as well as clarity and conciseness of the proof, we only consider the {\it deterministic} initial data, i.e., $u_0$ is deterministic. 
It would be interesting to investigate the {\it random} initial data, 
in particular, an analog of the previous lemma, as well as conditions and control of the energy.
Note that for an additive noise, blow-up with positive probability before any fixed time $T>0$ has been proven in  \cite{deB_Deb_PTRF} under restrictive conditions
on the nonlinearity and a more regular noise, while in \cite{deB_Deb_AnnProb} this result was obtained for initial data with negative energy (see \cite[Remarks 4.2 and 4.4]{deB_Deb_AnnProb}).

\begin{theorem}\label{th_blowup1-add}
Let $ \sigma \in 
(\frac{2}{n}, \frac{2}{(n-2)^{+}})  $ 
satisfy the condition {\bf (H3)},  $\phi$ satisfy the conditions {\bf(H4)} and  $u_0 \in \Sigma $ be deterministic. 
Suppose that for some constants $\beta_0$ and $\delta_0$ such that $\beta_0 < 1 < \delta_0$ we have 
\begin{equation} \label{cond-u_0-Q-det-add}
H(u_0) M(u_0)^\alpha = \beta_0 H(Q) M(Q)^\alpha \quad \mbox{\rm  and}\quad  \| \nabla u_0\|_{L^2} \|u_0\|_{L^2}^\alpha
= \delta_0 \|\nabla Q\|_{L^2} \|Q\|_{L^2}^\alpha.
\end{equation}
Then for $T$ large enough,  
$\|\phi\|_{L^{0,0}_2}, \|\phi\|_{L^{0,1}_2}, C(\phi)$ and $C^\Sigma_\phi$ 
small enough, we have $P\big(\tau^*(u_0)\leq T\big) >0$, where $\tau^*(u_0)$ is defined in Theorem \ref{lwp_add}.
\end{theorem}

We first prove a result similar to Proposition  \ref{Prop-blow-det}. 
\smallskip

Let ${\mathcal C}(p, u_0,t, \phi):= {\mathcal C}(p, u_0,t, \phi, \frac{1}{2})$, where the constant  ${\mathcal C}(p, u_0,t, \phi, \frac{1}{2})$ is defined in \eqref{def_tildeC} for
 $\tilde{\epsilon}=\frac{1}{2}$,
 i.e.,  we have the following bound on the moments of mass: 
\begin{equation}\label{E:M-moments}
\EE\Big( \sup_{s\leq  t} [M\big(u(s)\big)]^p\Big) \leq {\mathcal C}(p, u_0,t, \phi).
\end{equation}
 Note that 
 ${\mathcal C}(p, u_0,t, \phi)$  is an increasing function of the ``strength of the noise", namely, of $\| \phi\|_{L^{0,0}_2}$.

\begin{Prop}     \label{Prop-blow-det-add}
Let $u_0$ satisfy the assumptions of Theorem \ref{th_blowup1-add}, replacing $\delta_0$ by $\delta>1$ in the statement 
of \eqref{cond-u_0-Q-det-add}.  Let $\|u_0\|_{L^2} = \lambda_0 \|Q\|_{L^2}$,  $\lambda > \lambda_0$, and 
suppose that 
$$
\lambda^{2\alpha} H(u_0) \leq
\beta H(Q) \quad \mbox{for~~ some} \quad  \beta \in [\beta_0,1).
$$

Given $\epsilon>0$, $\mathfrak M>0$  and $t>0$, suppose that $\phi$  satisfies the following conditions:
\begin{align}
  (\lambda^2 M(Q))^{\alpha} \Big[ C_\phi^\Sigma + C_\phi^\Sigma t +  \frac{8\sqrt{2}}{3} n \|\phi\|_{L_2^{0,0}} t^{\frac{3}{2}}+ \big( 2 C_\phi^{(2)} + 32 C_\phi^{(1)} \big)  t^2 +
   \frac{4}{3} n\sigma C_{\phi}^{(1)} \Big] 
   &\leq \epsilon ,     \label{cond_1-add}
   \\
  \frac{32}{15}  n \sigma (\lambda^2 M(Q))^{\alpha }  \mathfrak M  \sqrt{C_{\phi}^{(1)}}  t^2 
  & \leq \epsilon,    \label{cond_2-add}  
   \\
  \frac{32}{15} n \sigma (\lambda^2 M(Q))^{\alpha }  
  C_{GN}^{\frac{2\sigma +1}{2\sigma +2}} \mathfrak M^{n\sigma \frac{2\sigma +1}{2\sigma +2}} 
   \mathcal{C}\Big( [2-(n-2)\sigma] \frac{2\sigma+1}{2\sigma+2},   u_0, t, \phi\Big)^{\frac{1}{2}}    \;\sqrt{C(\phi)}  t^{\frac{5}{2} } 
   & \leq \epsilon,  \label{cond_3-add}\\
  \frac{1}{3} n\sigma (2\sigma +1)    (\lambda^2 M(Q))^\alpha    C_{GN}^{\frac{\sigma}{\sigma+1}} \mathfrak M^{\frac{n\sigma^2}{\sigma+1}}
   {\mathcal C}\Big( [2-(n-2)\sigma] \frac{\sigma}{\sigma+1}, u_0, t, \phi\Big) C(\phi)  t^3 
&  \leq \epsilon.  \label{cond_4-add} 
  \end{align}
For $\delta \in (1, \delta_0]$ and $\mathfrak M>0$, set 
\begin{align}
    \tau_\delta~ &= \inf\{ s\geq 0 : \lambda^\alpha  \|\nabla u(s)\|_{L^2}    \leq \delta \|\nabla Q\|_{L^2} \}\wedge \tau^*(u_0)    \label{def-tau_delta},\\
    ~\tilde{\tau}_{\mathfrak M}
& = \inf\{ t\geq 0 : \|\nabla u(s)\|_{L^2} \geq \mathfrak M \} \wedge \tau^*(u_0).
\label{def-tildetauL-add}
\end{align}
Let $\{ \Omega_t\}_{t\geq 0} $ be a non-increasing family of $\{{\mathcal F_t}\}_{t\geq 0}$-adapted sets. Suppose that for some $t>0$, some non-null set $\Omega_t\in {\mathcal F}_t$ and the constants $\epsilon$, $\mathfrak M>0$  chosen above,  we have  
\begin{align}    \label{cond-3-add}
    P(\Omega_t)  V(u_0) (\lambda^2M(Q))^\alpha 
    &\;+4\epsilon  +4 G(u_0) (\lambda^2 M(Q))^\alpha \EE\big(1_{\Omega_t} (t\wedge \tau)\big) \nonumber \\
   &\;  - 4\sigma s_c (\delta^2-\beta) \|\nabla Q\|_{L^2}^2 M(Q)^\alpha\EE\big( 1_{\Omega_t} (t\wedge \tau)^2\big) <0, \quad \tau = \tau_\delta \wedge \tilde{\tau}_{\mathfrak M}. 
\end{align}
Then $P(\tau^*(u_0)\leq t)>0$. 
\end{Prop}
\begin{proof}
Suppose that for some $t$ we have $t<\tau^*(u_0)$ a.s.. Multiplying \eqref{V(u)-add-2} by $1_{\Omega_t} (\lambda^2 M(Q))^\alpha$ and taking expected values, yields 
\begin{align}		\label{EV-add-deter}
(\lambda^2& M(Q))^\alpha \EE\big( 1_{\Omega_t} V\big(u(t\wedge \tau)\big)   = (\lambda^2 M(Q))^\alpha V(u_0) P(\Omega_t) 
+ 4 (\lambda^2 M(Q))^\alpha G(u_0) \EE\big(1_{\Omega_t} (t\wedge \tau) \big) \nonumber \\
&+ 4n\sigma (\lambda^2 M(Q))^\alpha H(u_0) \EE\big(1_{\Omega_t}  (t\wedge \tau)^2\big)  
-8 \sigma s_c (\lambda^2 M(Q))^\alpha \EE\Big( 1_{\Omega_t} \int_0^{t\wedge \tau} ds \int_0^s \|\nabla u(r)\|_{L^2}^2 dr\Big)  \nonumber \\
&+ (\lambda^2 M(Q))^\alpha \big[  C_\phi^{\Sigma}  \EE\big(1_{\Omega_t} (t\wedge \tau)\big) 
+ 2 C^{(2)}_\phi  \EE\big( 1_{\Omega_t} (t\wedge \tau)\big)^2  + \frac{4}{3} n\sigma C^{(1)}_\phi \EE\big( 1_{\Omega_t} (t\wedge \tau)^3\big) \big] + \sum_{j=1}^4 \tilde{T}_j(t),
\end{align} 
where for $j=1, ..., 4$ we set $\tilde{T}_j(t)=(\lambda^2 M(Q))^\alpha \EE\big(1_{\Omega_t} T_i(t\wedge \tau)\big)$, and the terms $T_j(t\wedge \tau)$ are defined 
in the statement of \eqref{V(u)-add-2}. 
By assumption we have $(\lambda^2 M(Q))^\alpha  H(u_0 )\leq \beta  H(Q) M(Q)^\alpha$; by definition of $\tau_\delta$,
for $0\leq r\leq \tau \leq \tau_\delta$, we have $(\lambda^2 M(Q))^\alpha \|\nabla u(r)\|_{L^2}^2
\geq \delta^2 M(Q)^\alpha \|\nabla Q\|_{L^2}^2$. Hence, \eqref{H_Q} implies
\begin{align}		\label{Hyp-Q-add}
4n\sigma (\lambda^2 M(Q))^\alpha & H(u_0) \EE\big( 1_{\Omega_t} (t\wedge \tau)^2\big) -8\sigma s_c (\lambda^2 M(Q))^\alpha \EE\Big( 1_{\Omega_t} \int_0^{t\wedge \tau} ds \int_0^s \| \nabla u(r)\|_{L^2}^2 dr \Big) \nonumber \\
&\leq \big[ 4n\sigma \beta M(Q)^\alpha H(Q) - 4\sigma s_c \delta^2 M(Q)^\alpha \|\nabla Q\|_{L^2}^2 \big] \EE\big( 1_{\Omega_t} (t\wedge \tau)^2\big)  \nonumber \\
&\leq -4\sigma s_c \big( \delta^2-\beta) M(Q)^\alpha \|\nabla Q\|_{L^2}^2 \EE\big( 1_{\Omega_t}(t\wedge \tau)^2 \big). 
\end{align}

We next give upper estimates of $|\tilde{T}_j(t)| $, $j=1, ..., 4$.  All these upper estimates contain either a multiplicative factor that depends on the strength of the noise or a time integral of $(\lambda^2 M(Q) )^\alpha \EE(1_{\Omega_.} V(\cdot \wedge \tau))$.
First, note that  $|\tilde{T}_1(t) | \leq  \tilde{T}_{1,1}(t) + \tilde{T}_{1,2}(t)$, where
\begin{align*}
\tilde{T}_{1,1}(t)=
&\;  8n\sigma (\lambda^2 M(Q))^\alpha \Big| \EE\Big( \int_0^{t\wedge \tau} \!\!\! ds \int_0^s \!dr \int_0^r \sum_{k\geq 0} \int_{\RR^n} \nabla \overline{{u}(r_1,x)} \nabla (\phi e_k)(x) dx 
d\beta_k(r_1)\Big)\Big|, \\
\tilde{T}_{1,2}(t) = 
&\; {  8 n\sigma} (\lambda^2 M(Q))^\alpha \Big| \EE\Big( \int_0^{t\wedge \tau}\! \!\! ds \int_0^s \! dr \int_0^r \sum_{k\geq 0}  \int_{\RR^n} 
|u(r_1,x)|^{2\sigma} \overline{u(r_1,x)} \phi e_k(x) dx d\beta_k(r_1)\Big) \Big|.
\end{align*}
The Cauchy-Schwarz inequality implies
\begin{align}		\label{upp-tildeT_11}
 \tilde{T}_{1,1}(t) \leq
 &\;  8n\sigma (\lambda^2 M(Q))^\alpha \int_0^t\! \! ds \int_0^s \! dr \Big\{ \EE\Big( \int_0^{r\wedge \tau}  \sum_{k\geq 0}  \Big( \int_{\RR^n} 
 \nabla \bar{u}(r_1,x) \nabla (\phi e_k)(x) dx\Big)^2 dr_1 \Big) \Big\}^{\frac{1}{2}} \nonumber \\
 \leq 
 & \; 8n\sigma (\lambda^2 M(Q))^\alpha \int_0^t\! \! ds \int_0^s \! dr \Big\{ \EE\Big( \int_0^{r\wedge \tau}  \sum_{k\geq 0} \|\nabla u(r_1)\|_{L^2}^2 \| \nabla(\phi e_k)\|_{L^2}^2 dr_1\Big) \Big\}^{\frac{1}{2}} \nonumber \\
 \leq 
 & \; 8n\sigma (\lambda^2 M(Q))^\alpha \mathfrak M    \sqrt{C_\phi^{(1)}} \frac{4}{15} t^{\frac{5}{2}},
\end{align}
where in the last upper estimate we have used the fact that  $r\leq \tau \leq \tilde{\tau}_{\mathfrak M}$. 
A similar computation using H\"older's inequality implies that
\begin{align}
 \tilde{T}_{1,2}(t) \leq & \; 8n\sigma (\lambda^2 M(Q))^\alpha \int_0^t\!\! ds \int_0^s\!\! dr \Big\{ \EE\Big(\int_0^{r\wedge \tau} \sum_{k\geq 0} \Big( \int_{\RR^d} |(u(r_1,x)|^{2\sigma +1} |\phi e_k(x)| dx \Big)^2 dr_1 \Big) \Big\}^{\frac{1}{2}} \notag \\
\leq &\;  8n\sigma (\lambda^2 M(Q))^\alpha \int_0^t\!\! ds \int_0^s\!\! dr \Big\{ \EE\Big(\int_0^{r\wedge \tau} \sum_{k\geq 0}  \|u(r_1)\|_{L^{2\sigma +2}}^{ 2(2\sigma +1)}
\| \phi e_k\|_{L^{2\sigma +2}}^2  dr_1 \Big) \Big\}^{\frac{1}{2}}. \label{E:sub1}
\end{align}
Furthermore, the Gagliardo-Nirenberg inequality implies that
\[ 
\|u(r_1)\|_{L^{2\sigma +2}}^{2(2\sigma +1)} 
\leq C_{GN}^{\frac{2\sigma +1}{\sigma +1 }}   \|\nabla u(r_1)\|_{L^2}^{n\sigma \frac{2\sigma +1}{\sigma +1}}
\|u(r_1)\|_{L^2}^{(2-(n-2)\sigma) \frac{2\sigma +1}{\sigma +1}}. 
\] 
Since $C(\phi)=\sum_k \| \phi e_k\|_{L^{2\sigma +2}}^2 $, we deduce (recalling the definition of  $\tilde{\tau}_{\mathfrak M}$  from \eqref{def-tildetauL-add}, which implies  the bound $\mathfrak M$ 
on the gradient up to the stopping  time 
$\tilde{\tau}_{\mathfrak M}$) by substituting the above Gagliardo-Nirenberg  inequality into \eqref{E:sub1} and using \eqref{E:M-moments}
\begin{equation}	\label{upp-tildeT_12}
\tilde{T}_{1,2}(t) \leq 8n\sigma (\lambda^2 M(Q))^\alpha   C_{GN}^{\frac{2\sigma +1}{2\sigma +2 }}  \sqrt{C(\phi)}  {\mathfrak M}^{\frac{n\sigma (2\sigma +1)}{2\sigma +2}}
{\mathcal C}\Big([2-(n-2)\sigma] \frac{2\sigma+1}{2\sigma+2} ,u_0,t, \phi\Big)^{\frac{1}{2}}  
\frac{4}{15} t^{\frac{5}{2}}. 
\end{equation}
Using again the Cauchy-Schwarz and H\"older inequalities, we deduce
\begin{align*}
|\tilde{T}_2(t)| \leq &\;  2n\sigma (\lambda^2 M(Q))^\alpha \int_0^t\!\! ds \int_0^s\! \! dr \, \EE\Big( \int_0^{r\wedge \tau} 
{ (2\sigma +1)} \sum_{k\geq 0}  \int_{\RR^n} |u(r_1,x)|^{2\sigma} |\phi e_k(x)|^2
 dx dr_1  \Big) \\
\leq &\; { 2n\sigma (2\sigma+1)}  (\lambda^2 M(Q))^\alpha  \int_0^t\!\! ds \int_0^s\! \! dr \, \EE\Big( \int_0^{r\wedge \tau}
\sum_{k\geq 0} {  \|u(r_1)\|_{L^{2\sigma +2}}^{2\sigma} } \|\phi e_k\|_{L^{2\sigma +2}}^2 dr_1 \Big) .
\end{align*}

Using the Gagliardo-Nirenberg inequality and a similar argument for proving \eqref{upp-tildeT_12}, we deduce that
\begin{align}		\label{upp-tilde-T2}
|\tilde{T}_2(t)| \leq 
& 
\;  2n\sigma (2\sigma+1) (\lambda^2 M(Q))^\alpha 
C_{GN}^{\frac{\sigma}{\sigma+1}} {\mathfrak M}^{\frac{n\sigma^2}{\sigma+1}} 
{\mathcal C}\Big( [2-(n-2)\sigma] \frac{\sigma}{2\sigma+2}, u_0, t, \phi\Big)  C(\phi) \frac{t^3}{6}, 
\end{align}
where in the last upper estimate we have used the inequalities $(2-(n-2)\sigma) \frac{\sigma}{\sigma +1} \leq 2$ and $ \frac{2}{2\sigma} \leq 2$. 

We next bound the term $\tilde{T}_3(t)$.  Recall that for $r\leq t$, $1_{\Omega_t} \leq 1_{\Omega_r}$. The Cauchy-Schwarz inequality implies 
\begin{align}		\label{upper-tilde-T3}
\big| \tilde{T}_3(t)\big| \leq & \; 4 (\lambda^2 M(Q))^\alpha \int_0^t \!\! ds \Big\{ \EE\Big( \int_0^{s\wedge \tau} 1_{\Omega_r} \sum_{k\geq 0} 
\Big( \int_{\RR^n} \bar{u}(r,x) \big[ 2\, x\! \cdot\! \nabla(\phi e_k)(x) + n\phi e_k(x)\big] dx \Big)^2 dr \Big)
\Big\}^{\frac{1}{2}}  \nonumber \\
\leq &\; 4 (\lambda^2 M(Q))^\alpha \int_0^t \!\! ds \Big\{ \EE\Big( \int_0^{s\wedge \tau} 1_{\Omega_r}  
\Big[ 2  \Big(\int_{\RR^n} 4 |x|^2  |u(r,x)|^2 dx\Big) C_\phi^{(1)} +2 n^2 \| \phi\|_{L^{0,0}_2}^2 \Big] dr \Big) 
\Big\}^{\frac{1}{2}}  \nonumber \\
\leq &\; 4 \sqrt{2} (\lambda^2 M(Q))^\alpha \Big[  \sqrt{C_\phi^{(1)} } 2  t  \Big\{     \int_0^{t}  \EE\big(1_{\Omega_r}  V(u(r\wedge \tau)) \big) dr \Big) \Big\}^{\frac{1}{2}}  
+ n\|\phi\|_{L^{0,0}_2} 
\frac{2}{3}  t^{\frac{3}{2}} \Big] \nonumber \\
\leq &\; 32 (\lambda^2 M(Q))^\alpha C_\phi^{(1)}  t^2 + \int_0^t (\lambda^2 M(Q))^\alpha \EE\big(1_{\Omega_s}  V(u(s\wedge \tau))\big) ds \nonumber \\
&\qquad + \frac{8\sqrt{2}}{3} (\lambda^2 M(Q))^\alpha {n} \|\phi\|_{L^{0,0}_2} t^{\frac{3}{2}},
\end{align}
where in the last upper estimates we have used the Cauchy-Schwarz and Young inequalities. 

Finally, a similar argument implies (since $\Omega_t\subset \Omega_s$ for $s\leq t$) 
\begin{align}		\label{upper-tilde-T4}
\big| \tilde{T}_4(t)\big| \leq & \; 2 (\lambda^2 M(Q))^\alpha \Big\{ \EE \Big( \int_0^{t\wedge \tau}1_{\Omega_s} 
\sum_{k\geq 0} \Big( \int_{\RR^n} |x|^2 \bar{u}(s,x) \phi e_k(x) dx \Big)^2 ds \Big) \Big\}^{\frac{1}{2}} \nonumber \\
\leq &\; 2 (\lambda^2 M(Q))^\alpha \Big\{ \EE \Big( \int_0^{t\wedge \tau} 1_{\Omega_s}  \sum_{k\geq 0}  \Big( \int_{\RR^n} |x|^2 |u(s,x)|^2 dx \Big) 
\Big( \int_{\RR^n} |x|^2 |\phi e_k(x)|^2 dx \Big) ds \Big) \Big\}^{\frac{1}{2}} 
\nonumber \\
\leq &\; 2 (\lambda^2 M(Q))^\alpha {\sqrt{C_\phi^\Sigma} }\Big\{ \int_0^t \EE\big(1_{\Omega_s}  V(u(s\wedge \tau))\big) ds \Big\}^{\frac{1}{2}}  \nonumber \\
\leq & \; (\lambda^2 M(Q))^\alpha { C_\phi^\Sigma } + \int_0^t (\lambda^2 M(Q))^\alpha \EE\big( 1_{\Omega_s}V(u(s\wedge \tau))\big) ds. 
\end{align}
Collecting the upper estimates \eqref{EV-add-deter}--\eqref{upper-tilde-T4}, we deduce
\begin{align*}
(\lambda^2& M(Q))^\alpha \EE\big(1_{\Omega_t}  V(u(t\wedge \tau))\big) \leq \; (\lambda^2 M(Q))^\alpha V(u_0) P(\Omega_t) 
+ 4 (\lambda^2 M(Q))^\alpha G(u_0) \EE\big( 1_{\Omega_t} (t\wedge \tau) \big) \\
&  -4\sigma s_c (\delta^2-\beta) M(Q)^\alpha \|\nabla Q\|_{L^2}^2 \EE\big( (t\wedge \tau)^2\big)  + (\lambda^2 M(Q))^\alpha \big[ C_\phi^\Sigma t 
 + 2  C_\phi^{(2)} t^2 + \frac{4}{3} n\sigma C_{\phi}^{(1)} t^3 \big] \\
&+ \frac{32}{15} (\lambda^2 M(Q))^\alpha n\sigma \,\mathfrak M \sqrt{C_\phi^{(1)}}  t^{\frac{5}{2}} \\
&+ \frac{32}{15}   (\lambda^2 M(Q))^\alpha n\sigma C_{GN}^{\frac{2\sigma +1}{2\sigma +2}} 
{\mathfrak M}^{n\sigma \frac{2\sigma +1}{2\sigma +2}} {\mathcal C}\Big( [2-(n-2)\sigma] \frac{2\sigma+1}{2\sigma+2}  ,u_0,t,\phi\Big) \sqrt{C(\phi)} t^{\frac{5}{2}}  \\
& 
+ \frac{1}{3} (\lambda^2 M(Q))^\alpha n\sigma (2\sigma +1) C_{GN}^{\frac{\sigma}{\sigma+1}} {\mathfrak M}^{\frac{n\sigma^2}{\sigma+1}}  
{\mathcal C}\Big([2-(n-2)\sigma] \frac{\sigma}{2\sigma+2} , u_0, t, \phi\Big) t^3  \\
&+ 32 (\lambda^2 M(Q))^\alpha C^{(1)}_\phi t^2 + \frac{8\sqrt{2}}{3} (\lambda^2 M(Q))^\alpha n \| \phi\|_{L^{0,0}_2} t^{\frac{3}{2}} 
+ (\lambda^2 M(Q))^\alpha C_\phi^\Sigma \\
&+2 \int_0^t (\lambda^2 M(Q))^\alpha \EE\big(1_{\Omega_s}   V(u(s\wedge \tau))\big) ds.
\end{align*}
Since $\| u(s)\|_{H^1}^2 \leq {\mathfrak M}^2+\EE\big( \sup_{s\leq t} M(u(s))\big) <\infty$ for $ s\leq t\wedge \tau \leq t\wedge \tilde{\tau}_{\mathfrak M}$, the argument used in
\cite[page 85]{deB_Deb_PTRF}  implies that $\EE\big( \sup_{s\in [0,t]} V(u(s))\big) <\infty$. Hence, Gronwall's lemma and the upper estimates
\eqref{cond_1-add}--\eqref{cond_4-add} imply  
\begin{align}   \label{form_V_Q}
(\lambda^2 M(Q))^\alpha &\EE\big(1_{\Omega_t}  V(u(t\wedge \tau)) \big) \leq  \;  e^{2t} 
 \Big[ (\lambda^2 M(Q))^\alpha V(u_0) P(\Omega_t) + 4 (\lambda^2 M(Q))^\alpha G(u_0) \EE\big( 1_{\Omega_t} (t\wedge \tau) \big) 
\nonumber \\
&\qquad  - 4 \sigma s_c (\delta^2-\beta)  M(Q)^\alpha \|\nabla Q\|_{L^2}^2  \EE\big(1_{\Omega_t}  (t\wedge \tau)^2\big) + { 4\epsilon }  \Big].
\end{align} 
The hypothesis \eqref{cond-3-add} implies that $(\lambda^2 M(Q))^\alpha \EE\big( V(u(t\wedge \tau))\big) <0$, which brings a contradiction. Hence, $P(t\leq \tau^*(u_0)) >0$, concluding the proof. 
\end{proof}

{\it Proof of Theorem \ref{th_blowup1-add}.} ~As in the proof of Theorem \ref{th_blowup1}, we may assume that $\tau^*(u_0)=\infty$ a.s. and look for a contradiction.

Let $\lambda_0>0$ be defined by $\|u_0\|_{L^2} = \lambda_0 \|Q\|_{L^2}$. Choose $\lambda > \lambda_0$ such that 
$\beta = \beta_0 \big( \frac{\lambda}{\lambda_0}\big)^{2\alpha} <1$.  
Since $H(u_0) M(u_0) ^\alpha\leq \beta_0  H(Q) M(Q)^\alpha$, we deduce that
$(\lambda^2 M(Q))^\alpha  H(u_0) \leq \beta H(Q) M(Q)^\alpha$. 

Let $\gamma \in (\beta,1)$,   let $\epsilon_0\in \big( 0, \frac{1}{5}\big) $ satisfy $\gamma (1-5\epsilon_0) > \beta$band proceed as in the proof of Theorem \ref{th_blowup1}. Recall {once more that 
for $B=\frac{C_{GN}}{\sigma +1}$, }
the function $f(x)= \frac{1}{2} (x^2-Bx^{n\sigma})$ defined on $[0,+\infty)$  is strictly increasing on the interval $(0,x^*)$ and strictly decreasing on $(x^*, \infty)$,
where $x^*=\|\nabla Q\|_{L^2(\RR^n)} \|Q\|_{L^2(\RR^n)}^\alpha$ and $f(x^*) =H(Q) M(Q)^\alpha$, see Figure \ref{F:1}.  
Then, given $\gamma \in (\beta,1)$, we deduce the existence of $\delta >1$ such that
\begin{equation}\label{E:IFF}
\big(  x>x^* \quad \mbox{\rm and}\quad   f(x)\leq \gamma f(x^*) \big) \quad   \Longrightarrow
\quad  x\geq \delta x^*. 
\end{equation}

Since $\|\nabla u_0\|_{L^2} \|u_0\|_{L^2}^\alpha > x^*$, $\|\nabla u_0\|_{L^2} \|u_0\|_{L^2}^\alpha $ is located on the decreasing side of the graph
of $f$ (refer again to Figure \ref{F:1}). Therefore, by a.s. continuity of $\|\nabla u(s)\|_{L^2}$ and $\|u(s)\|_{L^2}$ in $s$ we deduce that for $s\in (0, \tau^*_0)$ 
such that $H(u(s)) M(u(s))^\alpha \leq  \gamma f(x^*)$ and
$\|\nabla u(s)\|_{L^2} \|u(s)\|_{L^2}^\alpha > x^*$, we have 
$\|\nabla u(s)\|_{L^2} \|u(s)\|_{L^2}^\alpha > \delta x^*$, 
in other words, 
\begin{equation}     \label{gamma-delta-A}
    \sup_{s\in [0,\tau_0^\ast) }H(u(s)) M(u(s))^\alpha \leq \gamma H(Q) M(Q)^\alpha \; \Longrightarrow 
\inf_{s\in [0,\tau_0^\ast)}\|\nabla u(s)\|_{L^2} \|u(s)\|_{L^2}^\alpha \geq \delta \|\nabla Q\|_{L^2} \|Q\|_{L^2}^\alpha.
\end{equation} 

Hence, if $\|u(s)\|_{L^2} \leq \lambda \|Q\|_{L^2}$, we have  for some $\delta \in (1, \delta_0)$
\begin{equation}\label{E:lamba2} 
\| \nabla u(s)\|_{L^2} (\lambda \|Q\|_{L^2})^\alpha \geq \| \nabla u(s)\|_{L^2} \|u(s)\|_{L^2}^\alpha \geq \delta  \|\nabla Q\|_{L^2} \|Q\|_{L^2}^\alpha. 
\end{equation}

Suppose that for some positive $T$ to be chosen later, we have $T<\tau^*(u_0)$ a.s.. 
For $t\in [0,T]$ set 
\[ \Omega_t = \Big\{ \omega : \sup_{s\leq t }  \|u(s)\|_{L^2} \leq  \lambda \|Q\|_{L^2} \Big\}. 
\]
Then $\{ \Omega_t\}_{t\in [0,T]}$ is an non-increasing family of ${\mathcal F}_t$-adapted sets. For $\gamma$ and $\delta$ chosen above,  set 
\begin{equation} \label{sigma_gamma2}
\tilde{\sigma}_\gamma = \inf\big\{ s\geq 0 :    H(u(s)) (\lambda^2 M(Q))^\alpha \geq \gamma H(Q) M(Q)^\alpha\big\}, 
\end{equation}
and
\begin{equation} \label{tau_delta2}
\tau_\delta = \inf\big\{ s\geq 0 : \| \nabla u(s)\|_{L^2} (\lambda \|Q\|_{L^2})^\alpha  \leq  \delta
 \|\nabla Q\|_{L^2} \|Q\|_{L^2}^\alpha \big\} .
\end{equation} 
Then on the set $\Omega_T$ by \eqref{gamma-delta-A} and \eqref{E:lamba2}, we have $\tilde{\sigma}_\gamma \leq \tau_{\delta}$. Furthermore, 
given any stopping time $\tau \leq \tau_\delta$
\[ (\lambda^2 M(Q))^\alpha 1_{\Omega_T} \int_0^{T\wedge \tau}  ds \int_0^s dr \|\nabla u(r)\|_{L^2}^2 \geq 
\delta^2 1_{\Omega_T} \int_0^{T \wedge \tau} ds \int_0^s dr
\|\nabla Q\|_{L^2} ^2 \|Q\|_{L^2}^{2\alpha}.\]

Using once more the identity \eqref{H_Q}, we deduce that
 \begin{align} \label{compensation_add}
       4n\sigma (\lambda^2 M(Q))^\alpha & H(u_0)   \EE\big( 1_{\Omega_T} (T\wedge \tau)^2 \big) -8\sigma s_c (\lambda^2 M(Q))^\alpha \EE\Big( 1_{\Omega_T}
 \int_0^{T\wedge \tau } ds \int_0^s \|\nabla u(r)|_{L^2}^2 dr \Big)  \nonumber \\
 & \leq -4\sigma s_c (\delta^2-\beta) \|\nabla Q\|_{L^2}^2 M(Q)^\alpha \EE\big( 1_{\Omega_T} (T\wedge \tau)^2 \big).
 \end{align}

The upper estimate \eqref{EV-add-deter} and the argument leading to 
 \eqref{form_V_Q} imply for $\tau = \tau_\delta 
 \wedge \tilde{\tau}_N$, where 
 $\tilde{\tau}_N = \inf\{ s\geq 0 : \| \nabla u(s)\|_{L^2} \geq N\}   \wedge \tau^*(u_0) $, 
 \begin{align*}      
&(\lambda^2 M(Q))^\alpha \EE\big(1_{\Omega_T}  V(u(T\wedge \tau)) \big) \leq  \;  e^{2t} 
 \Big[ (\lambda^2 M(Q))^\alpha V(u_0) P(\Omega_T) + 4 (\lambda^2 M(Q))^\alpha G(u_0) \EE\big( 1_{\Omega_T} (T\wedge \tau) \big) 
\nonumber \\
&\quad  - 4 \sigma s_c (\delta^2-\beta)  M(Q)^\alpha \|\nabla Q\|_{L^2}^2  \EE\big(1_{\Omega_T}  (T\wedge \tau)^2\big)
+ (\lambda^2 M(Q))^\alpha \Big[ C_\phi^\Sigma T + 2 C^{(2)}_\phi T^2  { +\frac{4}{3} n\sigma C_\phi^{(1)} T^3 } \Big]
+\sum_{j=1}^4 \tilde{T}_j(T)\Big].
\end{align*}
For $\|\phi\|_{L^{0,0}_2}, \| \phi\|_{L^{0,1}_2}, C(\phi)$ and $C^\Sigma_\phi$ small enough (which implies that {$C^{(1)}_\phi$ and} $C^{(2)}_\phi$ are small),
the upper estimates of the terms $\tilde{T}_j, j=1, ...4$, yield the upper estimates {\eqref{cond_1-add}-- \eqref{cond_4-add} } are satisfied, so that 
for $\tau \leq \tau_\delta$
\begin{align}   \label{Gronwall-V-add}
(\lambda^2 &M(Q))^\alpha \EE\big(1_{\Omega_T}  V(u(T\wedge \tau)) \big) \leq  \;  e^{2T} 
 \Big[ (\lambda^2 M(Q))^\alpha V(u_0) P(\Omega_T) + 4 (\lambda^2 M(Q))^\alpha G(u_0) \EE\big( 1_{\Omega_T} (T\wedge \tau) \big) 
\nonumber \\
&\qquad  - 4 \sigma s_c (\delta^2-\beta)  M(Q)^\alpha \|\nabla Q\|_{L^2}^2  \EE\big(1_{\Omega_T}  (T\wedge \tau)^2\big)
+ { 4\epsilon } \Big] .
\end{align}

We then proceed as in the proof of Theorem \ref{th_blowup1} and first fix a large value of $T>0$. As $T\to \infty$, $T\wedge \tau_\delta \to \tau_\delta$ 
and the monotone convergence
theorem implies that $ \EE\big( (T\wedge \tau_\delta)^2 \big)\to \EE(\tau_\delta^2)$. 
\smallskip

$\bullet$  If $\EE(\tau_\delta^2)=+\infty$, for any fixed $M$ (to be chosen later on) we have
$\EE\big( (T\wedge \tau_\delta)^2\big) \geq M^2$ for some large $T$. 

$\bullet$  If $\EE(\tau_\delta^2)<\infty$, then for any $\mu\in (0,1)$ and close enough to 1, 
for $T$ large enough we have
$\EE\big( (T\wedge \tau_\delta)^2 \big) \geq \mu^2 \EE(\tau_\delta^2)$ and $P\big( \Omega_T \cap \{  T<\tilde{\sigma}_\gamma \}\big) < {\epsilon}_0$, where
the last inequality follows from the fact that $\tilde{\sigma}_\gamma \leq \tau_\delta$ on $\Omega_T$ and $\tau_\delta < \infty$ a.s.
\smallskip

Since $\|\nabla u(t)\|_{L^2} \to \infty$  as $t\to \tau^*(u_0)$, we deduce that 
for $\tilde{\tau}_{M_0} = \inf\{  s\geq 0 : \|\nabla u(s)\|_{L^2} \geq M_0\}$,
for some $M_0$ large enough, if $\bar{\tau}_0 = \tau_\delta \wedge \tilde{\tau}_{M_0}$, 
we have  
$\EE\big( (T\wedge \bar{\tau}_0)^2 \big) = \EE\big( (T\wedge \tau_\delta \wedge \tilde{\tau}_{M_0})^2\big) \geq \mu^2 
\EE\big( ( T\wedge\tau_\delta)^2 \big)$.
Hence, either $\EE\big( (T\wedge \bar{\tau}_0)^2 \big) \geq \mu^2 M_0^2$
or $\EE\big( (T\wedge \bar{\tau}_0)^2 \big) \geq \mu^4 \EE\big(\tau_\delta^2\big)$. 

For this choice of $T$, $M$ and $M_0$,  suppose that $T<\tau^*(u_0)$ a.s., 
set $X= \big\{\EE\big( 1_{\Omega_T} (T\wedge {\tau}_\delta)^2 \big)\big\}^{\frac{1}{2}}$ and consider the polynomial 
$-ax^2+bx+c$, with $a=4\sigma s_c (\delta^2-\beta^2) M(Q)^\alpha$, $b=4(\lambda^2 M(Q))^\alpha |G(u_0)|$ and
$c=(\lambda^2 M(Q))^\alpha V(u_0) + {4\epsilon}$. The Cauchy-Schwarz inequality and \eqref{Gronwall-V-add} imply 
\[ (\lambda^2 M(Q))^\alpha \EE\big( 1_{\Omega_T} V(T\wedge \bar{\tau}_0) \big)\leq e^{2T} \big[ -a X^2 +bX+c\big] .\]
Let $X_1<X_2$ be the roots of $-aX^2+bX+c$. Then for $X>X_2$, we can conclude that $\EE\big( 1_{\Omega_T} V(T\wedge \bar{\tau}_0)\big) <0$,
which proves that
$P(T\leq \tau^*(u_0))>0$.
\smallskip

We next  claim that when the noise is ``small enough", we have 
$\EE\big( (T\wedge \tau_\delta)^2 \big) >2 X_2^2 $.
Choosing $\mu$ close enough to 1, we deduce that for a ``small noise", 
$\EE\big(  (T\wedge \bar{\tau}_{0})^2 \big) >  X_2^2 $.

$\bullet$  If $\EE(\tau_\delta^2)=\infty$,  choosing $M>\sqrt{2} X_2$, we deduce that  
$\EE\big( (T\wedge \tau_\delta)^2\big) > 2 X_2^2$ for $T$ large enough, completing the claim.

$\bullet$ If $\EE(\tau_\delta^2)<\infty$, 
it remains to prove that  $\EE\big( (\tau_\delta \wedge T)^2\big) \geq 2 X_2^2$. 

Since $\EE(\tau_\delta^2)<\infty$ and 
$\tilde{\sigma}_\gamma \leq \tau_\delta$ on $\Omega_T$, by a.s. continuity of $H(u(\cdot))$ on $[0,T]$, we deduce that 
$\lambda^{2 \alpha} H(u(\tilde{\sigma}_\gamma)) = \gamma H(Q)$ a.s. on the set  $\Omega_T$.
The upper bound \eqref{H_add_final} implies that for any stopping time $\tau \leq \tau_\delta$ 
\begin{align*}
 (\lambda^2 M(Q))^\alpha  H\big( u(\tau)\big) \
\leq & (\lambda^2 M(Q))^\alpha H(u_0) + \frac{1}{2} (\lambda^2 M(Q))^\alpha  \|\phi\|_{L^{0,1}_2}^2 \tau  +  (\lambda^2 M(Q))^\alpha  (I_1(\tau) + I_2(\tau)), 
\end{align*}
where 
\begin{align}		\label{I1-I2}
I_1(\tau)= &\;  - {\rm Im} \Big( \int_0^{\tau} \sum_k \int_{\RR^n}
|u(s,x)|^{2\sigma} \overline{u(s,x) } (\phi e_k)(x) dx d\beta_k(s) \Big),  \nonumber \\
I_2(\tau)=  
&\; {\rm Im} \Big( \int_0^{\tau} \sum_k \int_{\RR^n} \nabla \overline {u(s,x)} \nabla(\phi e_k)(x) dx d\beta_k(s) \Big). 
\end{align} 
For a fixed $\mathfrak L$ 
set $\tilde{\tau}_{\mathfrak L} = \inf\{ s\geq 0: \|\nabla u(s)\|_{L^2} \geq \mathfrak L\}$.

H\"older's inequality with conjugate exponents $2\sigma +2$ and $\frac{2\sigma +2}{2\sigma +1}$ implies 
 that for any stopping time $\tau \leq \tau_\delta$, 
\begin{align}		\label{upper-I1}
\EE(1_{\Omega_T} \big| I_1(\tau \wedge T  \wedge \tilde{\tau}_{\mathfrak L} ) |^2 \big)\leq   & 
\; \EE\Big( \int_0^{\tau \wedge T \wedge  \tilde{\tau}_{\mathfrak L}} 
1_{\Omega_s} \sum_k \Big|   \int_{\RR^n}
|u(s,x)|^{2\sigma+1}  (\phi e_k)(x) dx \Big|^2  ds\Big)   \nonumber \\
\; \leq &\; \EE\Big( \int_0^{\tau \wedge T \wedge \tilde{\tau}_{\mathfrak L}} 1_{\Omega_s} \sum_k \|u(s)\|_{L^{2\sigma +2}}^{\frac{2\sigma +1}{\sigma +1} }
\|\phi e_k\|_{L^{2\sigma +2}}^2  ds \Big)  \nonumber \\
\leq &\; C(\phi) C_{GN}^{\frac{2\sigma +1}{\sigma +1}}   {\mathfrak L}^{n\sigma \frac{2\sigma +1}{\sigma +1}}  \EE\Big( \int_0^{\tau \wedge T \wedge \tilde{\tau}_{\mathfrak L}}  1_{\Omega_s}
M(u(s))^{(2-(n-2)\sigma)\frac{2\sigma +1}{2\sigma +2} }   ds \Big)   \nonumber \\
\leq &\; C(\phi) C_{GN}^{\frac{2\sigma +1}{\sigma +1}}  {\mathfrak L}^{n\sigma \frac{2\sigma +1}{\sigma +1}}   
 \EE\Big( \int_0^{\tau \wedge T \wedge \tilde{\tau}_{\mathfrak L}}  1_{\Omega_s} \big[ 1+ M(u(s))^2\big] ds \Big) ,
\end{align}
where  in the last  two upper estimate we have used \eqref{GN} and the inequality  $(2-(n-2)\sigma) \frac{2\sigma +1}{2\sigma +2}\in (0, 2]$.
Therefore, 
Lemma \ref{lem-sup-Mp} applied with $p=2$ implies that
{$\EE\big( \sup_{s\leq T} M(u(s))^2\big) \leq C<\infty$ for some constant $C$ depending on $T, \|\phi\|_{L^{0,0}_2}$ and $ M(u_0)^2$. }
 Hence, $\EE(\big| I_1(\tau\wedge T \wedge \tilde{\tau}_{\mathfrak L} ) |^2\big) <\infty$ and  we deduce 
$\EE\big( I_1(\tau_\wedge T \wedge \tilde{\tau}_{\mathfrak L} ) \big)=0$. 
A similar computation yields
\begin{align*}
 \EE(    I_2(\tau  \wedge T \wedge \tilde{\tau}_{\mathfrak L})|^2) \leq & \; \EE\Big( \int_0^{\tau \wedge T  \wedge \bar{\tau}_{\mathfrak L}}
 \sum_k \Big( \int_{\RR^n} \nabla \bar{u}(s,x) \nabla(\phi e_k)(x) dx \Big)^2 ds \Big)  \\
 \leq & \; \EE\Big( \int_0^{\tau \wedge T  \wedge \tilde{\tau}_{\mathfrak L}} \sum_k \| \nabla (\phi e_k)\|_{L^2}^2 \| \nabla u(s)\|_{L^2}^2 ds \Big)  \leq \|\phi\|_{L^{0,1}_2}^2 {\mathfrak L}^2 T <\infty. 
\end{align*}
Hence, $\EE(I_2(\tau \wedge T \wedge \tilde{\tau}_{\mathfrak L}))=0$, and we deduce that for every ${\mathfrak L}>0$ and
$\tau \leq \tau_\delta$,
\[  (\lambda^2 M(Q))^\alpha \EE(1_{\Omega_T}  H\big( u(\tau \wedge T \wedge \tilde{\tau}_{\mathfrak L})\big) \leq 
 (\lambda^2 M(Q))^\alpha H(u_0)  P(\Omega_T)
 + \frac{1}{2} (\lambda^2 M(Q))^\alpha  \|\phi\|_{L^{0,0}_2}^2 \EE\big(1_{\Omega_T}(\tau \wedge T)
 \big). \]
We next adapt the proof of Theorem \ref{th_blowup1}. 
 Let $\sigma_0= \inf\{ t\geq 0 \; : \; H(u(t))\leq 0\}$. 
  Since $\tilde{\sigma}_\gamma \leq \tau_\delta $ on $\Omega_T$, choosing $\tau = \tilde{\sigma}_\gamma \wedge \sigma_0$, we deduce  
\[ \sup_{\mathfrak L} (\lambda^2 M(Q))^\alpha 1_{\Omega_T} H(u(\tilde{\sigma}_\gamma \wedge \sigma_0 \wedge T\wedge \tilde{\tau}_{\mathfrak L})) \leq 
1_{\Omega_T} \sup_{s\leq \tilde{\sigma}_\gamma  } (\lambda^2 M(Q))^\alpha H(u(s)) 
\leq 1_{\Omega_T} \gamma M(Q)^\alpha H(Q)<\infty\quad \mbox{\rm a.s.}.
\]
Furthermore, by definition of $\sigma_0$, the sequence $\big\{ (\lambda^2 M(Q))^\alpha 1_{\Omega_T} H(u(\tilde{\sigma}_\gamma \wedge \sigma_0 \wedge T\wedge \tilde{\tau}_{\mathfrak L}))\big\}_
{\mathfrak{L}} $
is nonnegative. 
Therefore,  as ${\mathfrak L} \to \infty$ the dominated convergence theorem  implies
$  (\lambda^2 M(Q))^\alpha \EE\big( 1_{\Omega_T} H(u(\tilde{\sigma}_\gamma \wedge \sigma_0 \wedge T \wedge \tilde{\tau}_{\mathfrak L})) \big) \to 
(\lambda^2 M(Q))^\alpha \EE\big( 1_{\Omega_T} H(u(\tilde{\sigma}_\gamma \wedge \sigma_0 \wedge T )) \big)$.
 
The definition of $\tilde{\sigma}_\gamma$ in \eqref{sigma_gamma2} and the a.s. continuity of $H(u(\cdot))$ on $(0, \infty)$ imply $(\lambda^2 M(Q))^\alpha H(u(\tilde{\sigma}_\gamma\wedge \sigma_0 \wedge T)) = \gamma M(Q)^\alpha H(Q)$ a.s.  on $\Omega_T \cap \{ \tilde{\sigma}_\gamma \leq \sigma_0 \wedge T\}$.
 Thus, using once more the inequality $H(u(s))(\omega) \geq 0$ for every $s\leq \sigma_0(\omega)$ and neglecting $(\lambda^2 M(Q))^\alpha \EE\big( 1_{\Omega_T}  H(u(\sigma_0\wedge T))
 1_{\{\sigma_0\wedge T<\tilde{\sigma}_\gamma \}}\big)\geq 0 $, we obtain
\begin{align} 	\label{upp-gamma-beta} 
\gamma M(Q)^\alpha H(Q) P(\Omega_T&  \cap \{ \tilde{\sigma}_\gamma    \leq \sigma_0 \wedge T\}) 
 \leq  P(\Omega_T) (\lambda^2 M(Q))^\alpha H(u_0)  +  
\frac{1}{2} (\lambda^2 M(Q))^\alpha  \|\phi\|_{L^{0,1}_2}^2  
\EE\big( 1_{\Omega_T}  (\tilde{\sigma}_\gamma \wedge T) \big) 
\nonumber \\
\leq & \; P(\Omega_T)  \beta M(Q)^\alpha H(Q)  +  
\frac{1}{2} (\lambda^2 M(Q))^\alpha 
\|\phi\|_{L^{0,1}_2}^2
\big\{\EE( (\tau_\delta \wedge T)^2)\big\}^{\frac{1}{2}}  ,
\end{align} 
where in the last upper estimate we used the definition of $\beta$ in terms of $\lambda, \beta_0$,  the upper estimate $\tilde{\sigma}_\gamma \leq \tau_\delta$ on $\Omega_T$,
the assumption 
\eqref{cond-u_0-Q-det-add} and the Cauchy-Schwarz inequality. 

Furthermore,
\begin{align*}  P(\Omega_T \cap \{ \tilde{\sigma}_\gamma \leq T\})  & = P(\Omega_T \cap \{ \tilde{\sigma}_\gamma \leq \sigma_0\wedge T\}) +
P(\Omega_T \cap \{ \sigma_0 < \tilde{\sigma}_\gamma \leq T\})\\
&\leq  P(\Omega_T \cap \{ \tilde{\sigma}_\gamma \leq \sigma_0\wedge T\}) +  P(\Omega_T \cap \{ \sigma_0 < T\}).
\end{align*} 
On $\{ \sigma_0 <T\}$  we have 
\[ 0=H(u(\sigma_0)) \geq H(u_0)  + {\rm Im}(I_1(\sigma_0)) + {\rm Im}(I_2(\sigma_0)),\]
where  $I_1$ and $I_2$ have been defined in \eqref{I1-I2}. 
Let ${\mathfrak L}_0$ be chosen such that $P(\tilde{\tau}_{{\mathfrak L}_0} \leq T) \leq {\epsilon}_0$. 
Then 
 \[ P(\Omega_T \cap \{ \sigma_0 <T \} ) 
 \leq  P(\tilde{\tau}_{{\mathfrak L}_0}\leq T) +
\sum_{j=1,2}   P\Big( \Omega_T \cap \Big\{ \sup_{s\in [0,T \wedge \tilde{\tau}_{{\mathfrak L}_0} ]} |I_j(s)|  \geq {H(u_0)}/{2} \Big\} \Big). 
 \] 
The Markov  and Davis inequalities, and the computations that  yield \eqref{upper-I1}  imply
  \begin{align*}
  P\Big( \Omega_T \cap \Big\{ \sup_{s\in [0,T \wedge \tilde{\tau}_{{\mathfrak L}_0}  ]} & |I_1(s)|  \geq {H(u_0)}/{2} \Big\} \Big) \leq \frac{6}{H(u_0)}
   \EE\Big( \Big\{ \int_0^{T\wedge \tilde{\tau}_{{\mathfrak L}_0}}
1_{\Omega_s} \sum_k \Big|   \int_{\RR^n}
|u(s,x)|^{2\sigma+1}  (\phi e_k)(x) dx \Big|^2  ds \Big\}\Big)    \nonumber \\
\; \leq &\; \EE\Big( \Big\{ \int_0^{ T \wedge \tilde{\tau}_{{\mathfrak L}_0}} 1_{\Omega_s}  \sum_k \|u(s)\|_{L^{2\sigma +2}}^{\frac{2\sigma +1}{\sigma +1} }
\|\phi e_k\|_{L^{2\sigma +2}}^2 ds \Big\}^{\frac{1}{2}} \Big) \nonumber \\
\leq &\; C(\phi)^{\frac{1}{2}}  C_{GN}^{\frac{2\sigma +1}{2\sigma +2}}  {{\mathfrak L}_0}^{n\sigma \frac{2\sigma +1}{2\sigma +2}}   
 \EE\Big( \Big\{ \int_0^{T \wedge \tilde{\tau}_{{\mathfrak L}_0}}  1_{\Omega_s} \big[ 1+ M(u(s))^2\big] ds\Big\}^{\frac{1}{2}} \Big)\\
 \leq &\;  C_{GN}^{\frac{2\sigma +1}{2\sigma +2}}  {{\mathfrak L}_0}^{n\sigma \frac{2\sigma +1}{2\sigma +2}}    \sqrt{T}\;   [1+\lambda^2 M(Q)]^{\frac{1}{2}} C(\phi)^{\frac{1}{2}} \leq {\epsilon}_0,
 \end{align*}  
 if $C(\phi)$ is small enough. 
 
Using the Markov,  Davis inequalities, and then the Cauchy-Schwarz inequality with respect to  $dx$,  we obtain 
 \begin{align*}
  P\Big( \Omega_T \cap \Big\{ \sup_{s\in [0,T \wedge \tilde{\tau}_{{\mathfrak L}_0}  ]} & |I_2(s)|  \geq {H(u_0)}/{2} \Big\} \Big) \leq \frac{6}{H(u_0)}
   \EE\Big( \Big\{ \int_0^{T\wedge \tilde{\tau}_{{\mathfrak L}_0}}
  \sum_k \Big( \int_{\RR^n} 1_{\Omega_s} \nabla \overline{u(s,x)} \nabla \phi e_k(x) dx \Big)^2 ds \Big\}^{\frac{1}{2}} \Big)\\
  & \leq \frac{6}{H(u_0)}  \EE\Big( \Big\{ \int_0^{T\wedge \tilde{\tau}_{{\mathfrak L}_0}} {\mathfrak L}_0 \; \sum_k \| \nabla (\phi e_k)\|_{L^2}^2 ds \Big\}^{\frac{1}{2}} \Big)\\
  &\leq \frac{6}{H(u_0)} \sqrt{T {\mathfrak L}_0} \|\phi\|_{L^{0,1}_2} \leq {\epsilon}_0,
 \end{align*} 
 if $\|\phi\|_{L^{0,1}_2} $ is small enough. 
 Therefore,  we deduce  $P(\Omega_T \cap \{ \sigma_0\leq T\}) \leq 3\epsilon_0$ for $\| \phi\|_{L^{0,1}_2}$ and $C(\phi)$ small enough.  Since we have chosen
 $T$ large enough to have $P(\Omega_T \cap \{T<\tilde{\sigma}_\gamma\})\leq \epsilon_0$, we obtain 
 \[ P(\Omega_T) \leq P(\Omega_T\cap \{ T<\tilde{\sigma}_\gamma\}) + P(\Omega_T \cap \{ \sigma_0 < T\}) + P(\Omega_T\cap \{ \tilde{\sigma}_\gamma \leq \sigma_0\wedge T\})
 \leq P(\Omega_T\cap \{ \tilde{\sigma}_\gamma \leq \sigma_0\wedge T\})+4\epsilon_0. 
 \] 
 Furthermore,  for $\|\phi\|_{L^{0,0}_2}$ small enough, the upper estimate \eqref{upp_P_OmegaMc} implies that, for example,  $P(\Omega_T^c) \leq \frac{1}{5}$, and thus, 
$P(\Omega_T)\geq \frac{4}{5}$. Hence, the upper estimate \eqref{upp-gamma-beta} yields
\[ \gamma H(Q) M(Q)^\alpha P(\Omega_T) \leq \beta H(Q) M(Q)^\alpha P(\Omega_T) + 
\frac{1}{2} (\lambda^2 M(Q))^\alpha \| \phi\|_{L^{0,1}_2}^2 \big\{ \EE\big( (\tau \wedge T)^2\big) \big\}^{\frac{1}{2}} + 4 \epsilon_0 \gamma H(Q) M(Q)^\alpha  \frac{5 P(\Omega_T)}{4}, 
\] 
which implies 
\[ \EE\big(
 (\tau_\delta \wedge T)^2\big) \geq \frac{4
\big(\gamma (1-5\epsilon_0) -\beta \big)^2 H(Q)^2 P(\Omega_T)^2}
{\lambda^{4\alpha} \|\phi\|_{L^{0,1}_2}^4}.
\]
 Recall that we have chosen $\epsilon_0\in \big(0, \frac{1}{5}\big)$ such that $\gamma (1-5\epsilon_0) > \beta$. 
Therefore, if  $\|\phi\|_{L^{0,1}_2}$ is small enough, we deduce that 
$\EE\big(  (\tau_\delta\wedge T)^2\big) \geq 2 X_2^2$,
which completes the proof.  
\hfill $\square$

\end{document}